\documentclass[11pt]{amsart}
\usepackage{amsmath}
\usepackage{amssymb}
\usepackage{amsfonts}
\usepackage{mathrsfs}
\usepackage[T1]{fontenc}
\numberwithin{equation}{section}       
\usepackage[english, frenchb]{babel}

\makeindex

\theoremstyle{plain}

\newtheorem{prop}{Proposition}[section]

\newtheorem{coro}[prop]{Corollary}
\newtheorem{lemm}[prop]{Lemma}
\newtheorem{fact}[prop]{Fact}
\newtheorem{ques}[prop]{Question}
\newtheorem{theoroman}{Theorem}

\newtheorem{theoalph}{Theorem}

\newtheorem{coroalph}[theoalph]{Corollary}

\theoremstyle{definition}
\newtheorem{defi}[prop]{Definition}

\theoremstyle{remark}
\newtheorem{rema}[prop]{Remark}

\newtheoremstyle{citing}
  {3pt}
  {3pt}
  {\itshape}
  {}
  {\bfseries}
  {.}
  {.5em}
  {\thmnote{#3}}

\theoremstyle{citing}
\newtheorem*{generic}{}

\newcommand{\partn}[1]{{\smallskip \noindent \textbf{#1.}}}

\newcommand{\C}{\mathbb{C}}

\newcommand{\R}{\mathbb{R}}

\newcommand{\cP}{\mathcal{P}}

\newcommand{\fB}{\mathfrak{B}}

\newcommand{\fD}{\mathfrak{D}}

\newcommand{\fG}{\mathfrak{G}}
\newcommand{\fH}{\mathfrak{H}}

\newcommand{\fL}{\mathfrak{L}}

\newcommand{\sA}{\mathscr{A}}

\newcommand{\sC}{\mathscr{C}}

\newcommand{\sF}{\mathscr{F}}

\newcommand{\sL}{\mathscr{L}}
\newcommand{\sM}{\mathscr{M}}

\newcommand{\sP}{\mathscr{P}}
\newcommand{\sQ}{\mathscr{Q}}

\newcommand{\hD}{\widehat{D}}

\newcommand{\hQ}{\widehat{Q}}

\newcommand{\hU}{\widehat{U}}
\newcommand{\hV}{\widehat{V}}
\newcommand{\hW}{\widehat{W}}

\newcommand{\hY}{\widehat{Y}}

\newcommand{\tB}{\widetilde{B}}

\newcommand{\tI}{\widetilde{I}}

\newcommand{\tU}{\widetilde{U}}
\newcommand{\tV}{\widetilde{V}}
\newcommand{\tW}{\widetilde{W}}
\newcommand{\tX}{\widetilde{X}}
\newcommand{\tY}{\widetilde{Y}}

\newcommand{\tgamma}{\widetilde{\gamma}}

\newcommand{\teta}{\widetilde{\teta}}

\renewcommand{\=}{ : = }

\DeclareMathOperator{\diam}{diam}

\DeclareMathOperator{\dist}{dist}

\DeclareMathOperator{\HD}{HD}
\DeclareMathOperator{\BD}{\overline{BD}}
\DeclareMathOperator{\Crit}{Crit}

\newcommand{\CC}{\overline{\C}}

\newcommand{\map}{f}
\newcommand{\CJ}{\Crit'(\map)}
\newcommand{\HDhyp}{\HD_{\operatorname{hyp}}}

\DeclareMathOperator{\mmod}{mod} 
\DeclareMathOperator{\dom}{dom}
\newcommand{\la}{\lambda}
\newcommand{\eps}{\varepsilon}
\newcommand{\wtm}{\widetilde{m}}
\newcommand{\ov}{\overline}

\DeclareMathOperator{\Cr}{Cr} 
\DeclareMathOperator{\dep}{dep}
\DeclareMathOperator{\bad}{bad}
\DeclareMathOperator{\hyp}{hyp}
\newcommand{\Jcon}{J_{\operatorname{con}}} 
\newcommand{\fBr}{\mathfrak{B}^{\operatorname{rel}}}

\newcommand{\poincare}{\delta_{\operatorname{Poin}}}

\newcommand{\taubase}{2}
\newcommand{\koebefactor}{A}

\begin{document}

\title[Weakly hyperbolic one-dimensional maps]{Statistical properties of one-dimensional maps under weak hyperbolicity assumptions
\\ \medskip
Propri{\'e}t{\'e}s statistiques des applications unidimensionelles sous des hypoth{\`e}ses d'hyperbolicit{\'e} faibles}
\author[J. Rivera-Letelier]{Juan Rivera-Letelier$^\dag$}
\author[W. Shen]{Weixiao Shen$^\ddag$}
\thanks{$^\dag$ Partially supported by Research Network on Low Dimensional Systems, PBCT/CONICYT, Chile. Gratefully acknowledges the hospitality of the University of Science and Technology of China and the National University of Singapore.}
\thanks{$^\ddag$ Partially supported by a start-up grant from National University of Singapore (Grant No. R-146-000-128-133).}
\address{$\dag$ Juan Rivera-Letelier, Facultad de Matem{\'a}ticas, Pontificia Universidad Cat{\'o}lica de Chile, Avenida Vicu{\~n}a Mackenna~4860, Santiago, Chile}
\email{riveraletelier@mat.puc.cl}
\address{$\ddag$ Weixiao Shen, Block S17, 10 Lower Kent Ridge Road, Singapore~119076, Singapore}
\email{matsw@nus.edu.sg}

\maketitle

\selectlanguage{english}
\renewcommand{\abstractname}{Abstract}
\begin{abstract}
For a real or complex one-dimensional map satisfying a weak
hyperbolicity assumption, we study the existence and statistical
properties of physical measures, with respect to geometric reference
measures.
We also study geometric properties of these measures.
\end{abstract}

\selectlanguage{frenchb}
\renewcommand{\abstractname}{R\'esum\'e}
\begin{abstract}
Pour une application unidimensionelle r{\'e}elle ou complexe satisfaisant une hypoth{\`e}se d'hyperbolicit{\'e} faible, on {\'e}tudie l'existence et des propri{\'e}t{\'e}s statistiques des mesures physiques, par rapport {\`a} une mesure de r{\'e}f{\'e}rence g{\'e}om{\'e}trique.
On {\'e}tudie aussi des propri{\'e}t{\'e}s g{\'e}om{\'e}triques de ces mesures.
\end{abstract}
\selectlanguage{english}

\section{Introduction}
We study statistical properties of real and complex one-dimensional maps, under weak hyperbolicity assumptions.
For such a map~$f$ we are interested in the existence and statistical properties of an invariant probability measure~$\nu$, supported on the Julia set of~$f$, that is absolutely continuous with respect to a natural reference measure.
The reference measure~$\mu$ could be the Lebesgue measure on the phase space, or more generally a conformal measure supported on the Julia set.
Such a measure~$\nu$, when ergodic, has the important property of being a \emph{physical measure}  with respect to~$\mu$.
That is, for a subset~$E$ of the phase space that has positive measure with respect to~$\mu$, the measure~$\nu$ describes the asymptotic distribution of each forward orbit of~$f$ starting at a point in~$E$.

For maps that are uniformly hyperbolic on their Julia sets, the
pioneering work of Sina{\u{\i}}, Ruelle, and Bowen~\cite{Sin72,Bow75,Rue76}
gives a satisfactory solution to these problems. See also
\cite{Sul83} for an analysis closer to the approach here. However, a
one-dimensional map with a critical point in its Julia set fails to
be uniformly hyperbolic in a severe way. In order to control the
effect of critical points in the Julia set, people often assume
strong expansion along the orbits of critical values. See for
example~\cite{ColEck80,Mis81,BenCar85,NowvSt91,KelNow92,You92,BruLuzvSt03} in
the real setting, and~\cite{Prz98,GraSmi98,GraSmi09} in the complex
setting.
See~\cite{BonDiaVia05} for a broad view.

For smooth interval maps, Bruin, Luzzatto and van Strien gave mixing rates upper bounds closely related to the growth of derivatives at the critical values~\cite{BruLuzvSt03}.
Our results reveal that, rather surprisingly, the mixing rates can be much faster than the growth of derivatives at critical values: an interval map~$f$ satisfying the \emph{Large Derivatives} condition \index{large derivatives condition}
$$ \lim_{n\to\infty} |Df^n(v)| = \infty,
\text{ for each critical value~$v$ of~$f$ in the Julia set}$$
together with other mild conditions, has a super-polynomially mixing absolutely continuous invariant measure.

In the complex setting we show a similar result for a non-renormalizable polynomial~$f$.
These are the first non-exponential upper bounds for mixing rates in the complex setting.
For a general rational map~$f$ without parabolic cylces we show that the summability condition with exponent~$1$ is enough to guarantee the existence of a super-polynomially mixing absolutely continuous invariant measure.

We shall now state two results, one for the real case and another for the complex case, and make comparisons with previous results.
In order to avoid technicalities, we state these results in a more restricted situation than what we are able to handle.
See~\S\ref{ss:backward contracting maps} for a more general formulation of our results and for precisions.

Recall that given a continuous map~$f$ acting on a compact metric space~$X$, an $f$-invariant Borel probability measure~$\nu$
is called \emph{(strongly) mixing} \index{mixing} if for all $\varphi,
\psi\in L^2(X, \nu)$,
$$\sC_n(\varphi, \psi)\=\int_X \varphi\circ f^n \psi d\nu-\int_X
\varphi d\nu\int_X \psi d\nu\to 0$$ as $n\to\infty$.
Given $\gamma > 0$, we
say that~$\nu$ is \emph{polynomially mixing of exponent~$\gamma$} \index{polynomially mixing}
if for
each essentially bounded function~$\varphi$ and each H\"older
continuous function~$\psi$, there exists a constant $C(\varphi,
\psi) > 0$ such that
$$|\sC_n(\varphi, \psi)|
\le
C(\varphi, \psi)n^{-\gamma},
\text{ for all }n=1,2,\ldots.$$
Moreover, we say that~$\nu$ is \emph{super-polynomially mixing} \index{super-polynomially mixing} if for all $\gamma > 0$ it is polynomially mixing of exponent~$\gamma$.

\begin{theoroman}\label{t:super-polynomial real}
Let~$X$ be a compact interval and let $f : X \to X$ be a topologically exact~$C^3$ multimodal map with non-flat critical points, having only hyperbolic repelling periodic points.
Assume that for each critical value~$v$ of~$f$ we have
$$\lim_{n \to \infty} |Df^n(v)| = \infty. $$
Then~$f$ has a unique invariant probability measure that is absolutely continuous with respect to Lebesgue measure.
Moreover, this invariant measure is super-polynomially mixing.
\end{theoroman}

The topological exactness is assumed to obtain uniqueness and the mixing property of the absolutely continuous invariant measure.
For an interval map as in the theorem, the existence of the absolutely continuous invariant measure was proved before in~\cite{BruShevSt03,BruRivShevSt08}, although the argument in this paper provides an alternative proof.
As mentioned above, our result on mixing rates significantly strengthens the previous result~\cite{BruLuzvSt03}, where super-polynomial mixing rates were only proved under the condition that for each~$\alpha > 0$ and each critical value~$v$ of~$f$, we have $|Df^n(v)|/n^\alpha\to\infty$.
In fact, only assuming $\liminf_{n \to \infty} |Df^n(v)|$ sufficiently large, our methods provide a definite polynomial mixing rate.

We now state a result for a complex rational map~$f$ of degree at least two.
Often the Lebesgue measure of the Julia set~$J(f)$ of~$f$ is zero.
So the Lebesgue measure cannot be used as a reference measure in general.
Instead people often use a conformal measure on the Julia set as a reference
measure.
Following Sullivan~\cite{Sul83}, we use conformal measures of exponent $\HD(J(f))$ as geometric reference measures, where
~$\HD(J(f))$ denotes the Hausdorff dimension of~$J(f)$.

\begin{theoroman}\label{t:super-polynomial complex}
Let~$f$ be either one of the following:
\begin{enumerate}
\item[1.]
an at most finitely renormalizable polynomial of degree at least two, that has only hyperbolic periodic points, and such that for each critical value~$v$ of~$f$ in the Julia set,
$$ \lim_{n \to \infty} |Df^n(v)| = \infty; $$
\item[2.]
a complex rational map of degree at least two, without parabolic cycles, and such that for each critical value~$v$ of~$f$ in the Julia set,
$$ \sum_{n = 1}^{\infty} \frac{1}{|Df^n(v)|} < \infty. $$
\end{enumerate}
Then~$f$ has a unique conformal measure~$\mu$ of exponent~$\HD(J(f))$; this measure is supported on the conical Julia set and its Hausdorff dimension is equal to~$\HD(J(f))$.
Furthermore, there is a unique invariant probability measure~$\nu$ that is absolutely continuous with respect to~$\mu$, and the measure~$\nu$ is super-polynomially mixing.
\end{theoroman}

Recall that for an integer~$s \ge 1$, a complex polynomial~$f$ is \emph{renormalizable of period~$s$} \index{renormalizable}
if there are Jordan disks~$U \Subset V$ such that the following hold:
\begin{itemize}
\item
$f^s: U\to V$ is proper of degree at least two;
\item
the set $\{z\in U: f^{sn}(z) \in U \text{ for all } n=1,2,\ldots\}$ is a connected proper subset of $J(f)$;
\item
for each critical point $c$ of $f$, there exists at most one $j\in \{0,1,\ldots, s-1\}$ with $c\in
f^j(U)$.
\end{itemize}
We say that~$f$ is \emph{infinitely renormalizable} if there are infinitely many~$s$ for which~$f$ is renormalizable of period~$s$.

For a complex polynomial~$f$, hypothesis~$1$ of Theorem~\ref{t:super-polynomial complex} is weaker than hypothesis~$2$.

Theorem~\ref{t:super-polynomial complex} gives the first non-exponential mixing rates in the complex setting.
As for the existence of the absolutely continuous invariant measure, this result gives a significant improvement of the previous result of Graczyk and Smirnov~\cite[Theorem~$4$]{GraSmi09}.
Their result applies to a rational map~$f$ satisfying the following strong form of the summability condition, for a sufficiently small~$\alpha \in (0, 1)$,
\begin{equation}
\label{e:strong summability}
\sum_{n = 1}^{\infty} \frac{n}{|Df^n(v)|^\alpha} < \infty, \text{ for every critical value~$v$ of~$f$ in~$J(f)$}.
\end{equation}
For each~$\alpha \in (0, 1)$, the Fibonacci quadratic polynomial~$f_0$ fails to satisfy this condition, although for every~$\alpha > 0$
$$ \sum_{n = 1}^{\infty} \frac{1}{|Df_0^n(v)|^\alpha} < \infty,
\text{ where~$v$ is the finite critical value of~$f_0$}, $$
see Remark~\ref{r:Fibonacci}.
So Theorem~\ref{t:super-polynomial complex} implies that the Fibonacci quadratic polynomial~$f_0$ has a super-polynomially mixing absolutely continuous invariant measure.

\begin{rema}
In the proof of Theorems~\ref{t:super-polynomial real} and~\ref{t:super-polynomial complex} we construct the absolutely continuous invariant measure by way of an inducing scheme with a super-polynomial tail estimate and some additional technical properties, see~\S\ref{ss:reduced statements}.
The results of~\cite{You99} imply that this measure is super-polynomially mixing
and that it satisfies the Central Limit Theorem for H{\"o}lder continuous observables.
It also follows that the absolutely continuous invariant measure has other statistical properties, such
as the Local Central Limit Theorem, and the Almost Sure Invariance
Principle, see \emph{e.g.}~\cite{Gou05,MelNic05,MelNic08,Tyr05}.
\end{rema}

For a map~$f$ satisfying the hypotheses of Theorem~\ref{t:super-polynomial real} or Theorem~\ref{t:super-polynomial complex} we show the density of the absolutely continuous invariant
measure has the following regularity: if we denote by~$\ell$ the
maximal order of a critical point of~$f$ in the Julia set, then for
each~$p \in (0, \ell / (\ell - 1))$ the invariant density belongs to
the space~$L^p$.
We note that for each~$p > \ell / (\ell - 1)$ the invariant density does not belong to~$L^p$, see Remark~\ref{r:optimality L^p}.
In the real case the regularity of
the invariant density was shown in~\cite[Main Theorem]{BruRivShevSt08}; see also~\cite{NowvSt91} for the case of
unimodal maps satisfying a summability condition with a certain
exponent.
In the complex setting our result seems to be the first
unconditional one. For rational maps satisfying a summability
condition with a sufficiently small exponent, a similar result  was
shown in~\cite[Corollary~$10.1$]{GraSmi09} under an integrability
assumption on the conformal measure~$\mu$ that was first formulated in~\cite{Prz98}.
Actually, in the complex case we shall prove for each~$\varepsilon > 0$ the following regularity of the conformal measure~$\mu$: for every sufficiently small~$\delta > 0$ we have for every $x \in J(f)$,
$$ \delta^{\HD(J(f)) + \varepsilon}
\le
\mu(B(x, \delta)) \le \delta^{\HD(J(f)) - \varepsilon}.$$
The lower bound is~\cite[Theorem~$1$]{LiShe08a}, while the upper bound is new and
implies the integrability condition for each exponent~$\eta<\HD(J(f))$, see~\eqref{e:integrability condition} in~\S\ref{ss:backward contracting maps}.

Let us say a few words on our strategy.
Prior to this work, it has been shown that a map satisfying the assumptions of
Theorem~\ref{t:super-polynomial real} or of
Theorem~\ref{t:super-polynomial complex}
has the following two expanding properties: ``\emph{expansion away from critical points}'' and ``\emph{backward contraction}''.
Roughly speaking, the first property means that outside any given neighborhood of the critical points the map is uniformly hyperbolic; the second property means that a return domain to a ball of radius~$\delta$ centered at a critical value is much smaller than~$\delta$.
See~\S\ref{ss:backward contracting maps} for the precise definitions and references, as well as our
``Main  Theorem'' stated for maps satisfying these two expanding properties.

In this paper, we provide a finer quantification of the expansion features of a map that satisfies the two properties stated above.
Firstly, we show that the components of the preimages of a small ball intersecting the Julia set shrink \emph{at least at a super-polynomial rate} (Theorem~\ref{t:polynomial shrinking} in~\S\ref{subsec:psc}).
This unexpected result represents a significant improvement on the estimate of the same type in \cite[Proposition~$7.2$]{GraSmi09}, for rational maps satisfying the summability condition with a sufficiently small exponent.
In our proof, the moduli of annuli are used to estimate diameter of sets.
The situation in the real setting is much trickier than in the complex one, due to a significant control loss of the modulus of a thin annulus under pull-back, see Lemma~\ref{lem:Rmodulus} and the remark before it.
We develop a ``quasi-chain'' construction to treat this problem.
Secondly, we introduce a dimension-like parameter we call ``\emph{badness exponent}'', that measures the combined size of all ``bad pull-backs'' of a suitably chosen small neighborhood of the critical points.
A bad pull-back is a pull-back that is not contained in any diffeomorphic pull-back.
This notion was first introduced in \cite{PrzRiv07} and it has some resemblance with the pull-backs corresponding to a backward orbit of a critical point ``with sequence 11\ldots 1'', as used in \cite{{GraSmi98},GraSmi09}.
Using the local Markov structure ({\em nice sets}) provided by the backward contracting property, we show that the badness exponent is \emph{zero} (Theorem~\ref{t:finiteness of bad} in~\S\ref{subsec:be}).
A direct consequence of this result is that the conical Julia set has codimension zero in the Julia set (Corollary~\ref{cor:HDhyp} in~\S\ref{ss:dimensionestimate}).

These expanding properties are converted to statistical properties of the system through the construction of an induced Markov map.
The approach is conventional but the construction is often technical. In this paper, this is done by applying techniques developed in \cite{PrzRiv07, PrzRiv11} with modification.
We obtain a tail estimate in terms of the rate of shrinking of components of preimages of small sets, and of the badness exponent only (Theorem~\ref{t:tail estimate} in~\S\ref{sss:inducing}).
We also obtain the existence and regularity of a geometric conformal measure.
The existence of a super-polynomially mixing absolutely continuous invariant measure then follows from a
well-known result of Young~\cite{You99}.

The result on the regularity of the invariant density is obtained through an upper
bound of the Poincar{\'e} series (Theorem~\ref{t:regularity density}
in~\S\ref{sss:regularity density}).

Finally let us mention a few by-products of our approach. For a map satisfying the hypotheses of Theorem~\ref{t:super-polynomial real} or of Theorem~\ref{t:super-polynomial complex}, we show that several notions of fractal dimension of the Julia set coincide (Theorem~\ref{t:fractal dimensions}).
For a complex polynomial that is expanding away from critical points and backward contracting, we show that the Julia set is locally connected when connected (Corollary~\ref{c:local connectivity}), has Hausdorff dimension less than~$2$ (Corollary~\ref{coro:inducing}) and is holomorphically removable (Theorem~\ref{t:holomorphic removability}).

\section{The Main Theorem and reduced statements}\label{s:reduced statements}
In this section we recall the definition of the properties ``expanding away from critical points'' and ``backward contracting'', and then state our Main Theorem (\S\ref{ss:backward
contracting maps}) from which we deduce
Theorems~\ref{t:super-polynomial real} and~\ref{t:super-polynomial complex} as direct consequences. Then we state five intermediate
results, Theorems~\ref{t:polynomial shrinking},~\ref{t:finiteness of
bad}, \ref{t:tail estimate} and~\ref{t:regularity density}, and
Corollary~\ref{coro:inducing} (\S\ref{ss:reduced statements}), and
deduce the Main Theorem (\S\ref{ss:proof of Main Theorem}). Finally,
in~\S\ref{ss:further results} we state some further results, which
are proved in~\S\ref{s:regularity density}.

The proofs of Theorems~\ref{t:polynomial
shrinking},~\ref{t:finiteness of bad}, and~\ref{t:tail estimate} are
independent, and are shown in~\S\S\ref{s:polynomial shrinking},
\ref{s:finiteness of bad}, and~\ref{s:tail estimate}, respectively.
Corollary~\ref{coro:inducing} is deduced from Theorem~\ref{t:tail
estimate} in \S\ref{ss:proof of inducing}. The proof of
Theorem~\ref{t:regularity density} depends on
Corollary~\ref{coro:inducing}, and it is given
in~\S\ref{s:regularity density}.

\subsection{The Main Theorem}
\label{ss:backward contracting maps}
We say that a map~$f : X \to X$
from a compact interval~$X$ of~$\R$ into itself is \emph{of
class~$C^3$ with non-flat critical points} if~$f$ is of class~$C^1$
on~$X$ and satisfies the following properties:
\begin{itemize}
\item $f$ is of class~$C^3$ outside $\Crit(f) \= \{x\in X : Df(x)=0\}$;
\item for each $c \in \Crit(f)$, there exists a number $\ell_c>1$ (called \emph{the order of~$f$ at~$c$}) \index{order}
and diffeomorphisms $\phi, \psi$ of~$\R$ of class~$C^3$ with $\phi(c)=\psi(f(c))=0$ such that
$$|\psi\circ f(x)|=|\phi(x)|^{\ell_c}$$
holds on a neighborhood of~$c$ in~$X$.
\end{itemize}

We use~$\sA_{\R}$ \index{$\sA_{\R},\sA_{\C},\sA$}to denote the collection of all~$C^3$ interval
maps with non-flat critical points, without neutral periodic points, and that are \emph{boundary-anchored}, \emph{i.e.}, for each~$x \in \partial\dom (f)$, we have~$f(x)\in \partial \dom (f)$ and~$Df(x) \neq 0$.
The last condition is convenient when considering pull-backs of
sets.

We use~$\sA_{\C}$ to denote the collection of all rational maps of
degree at least~$2$ without neutral periodic points.
As in the real case, for $f \in \sA_{\C}$ we
denote by~$\Crit(f)$ \index{$\Crit$} the set of critical points of~$f$, and for
each~$c \in \Crit(f)$ we denote by~$\ell_c$ the local degree of~$f$
at~$c$, that we will also call \emph{the order of~$f$ at~$c$}.

For $f\in\sA\=\sA_{\R}\cup \sA_{\C}$, we denote by~$\dom(f)$ the
Riemann sphere~$\CC$ if~$f$ is a complex map, and the compact
interval where~$f$ is defined otherwise.
The \emph{Julia set} \index{Julia set} \index{$J(f)$}
$J(f)$ of~$f$ is, by definition, the set of all
points $x\in \dom (f)$ with the following property: for any
neighborhood~$U$ of~$x$, the family $\{f^n|U\}_{n=0}^\infty$ is not
equicontinuous.
This is a forward invariant compact set.
We shall mainly be interested in the dynamics on~$J(f)$, where the chaotic
dynamical behavior concentrates.
It is known that for $f\in\sA_{\C}$, the set~$J(f)$ is the closure of repelling periodic
points and that for $f\in\sA_{\R}$, the set~$J(f)$ is the complement of the basins of periodic attractors.
See~\cite{Mil06} and \cite{dMevSt93} for more background.

For maps without critical points in the Julia set all of our results are either well-known or vacuous.
So for~$f \in \sA$ we will implicitly assume that the set
$$ \Crit'(f) \= \Crit(f)\cap J(f) $$
\index{$\Crit'$}
is nonempty.
We also put
$$ \ell_{\max}(f) \= \max \{ \ell_c : c \in \CJ \}. $$
\index{$\ell_{\max}$}
Given $\ell > 1$ we will denote by~$\sA(\ell)$ (resp.
$\sA_{\R}(\ell)$, $\sA_{\C}(\ell)$) 
\index{$\sA_{\R}(\ell)$, $\sA_{\C}(\ell)$,$\sA(\ell)$}
the class of all those~$f \in
\sA$ (resp. $\sA_{\R}$, $\sA_{\C}$) such that~$\ell_{\max}(f) \le
\ell$.

\begin{defi}\label{def:eafcp}
We say that a map~$f \in \sA$ is \emph{expanding away from
critical points} \index{expanding away from critical points},  if for every neighborhood~$V'$ of~$\CJ$ the
map~$f$ is uniformly expanding on the set
$$
A=\{ z \in J(f) : \text{ for every $n \ge 0$, $f^n(z) \not \in
V'$} \},
$$
\emph{i.e.}, there exist constants $C>0$ and $\lambda>1$ such that
for any $z\in A$ and $n\ge 0$, we have $|Df^n(z)|\ge
C\lambda^n$.
\end{defi}

A theorem of Ma{\~n}{\'e} asserts that every map~$f\in\sA_{\R}$ is expanding away from critical points, see~\cite{Man85b}.
Although the analogous statement for a map~$f$ in~$\sA_{\C}$ is false in general, it does hold if we assume in addition that~$f$ is a polynomial that is at most finitely renormalizable, see~\cite{KozvSt09}, and also~\cite{QiuYin09} for the totally disconnected case.

We will now recall the ``backward contraction property'' introduced
in~\cite{Riv07} in the case of rational maps, and
in~\cite{BruRivShevSt08} in the case of interval maps.
Let~$f \in \sA$ be given.
When studying $f\in\sA_{\R}$, we use the standard metric on the interval~$\dom(f)$, while when studying~$f\in\sA_{\C}$, we shall use the spherical metric on~$\CC$.
For a critical point~$c$ and~$\delta > 0$ we denote by~$\tB(c, \delta)$
the connected component of~$f^{-1}(B(f(c), \delta))$ containing~$c$.

\begin{defi}\label{def:bc}
Given a constant~$r > 1$ we will say a map~$f \in \sA$ is
\emph{backward contracting with constant~$r$} \index{backward contracting} if there is a constant~$\delta_0 > 0$ such that for every $\delta \in (0, \delta_0)$, every~$c \in
\CJ$, every integer~$m \ge 0$, and every connected component~$W$
of~$f^{-m}(\tB(c, r \delta))$,
$$ \dist(W, f(\Crit(f))) \le \delta \text{ implies } \diam (W) < \delta. $$
Furthermore, we say that~$f$ is \emph{backward contracting} if, for every~$r > 1$, it is backward contracting with constant~$r$.
\end{defi}

\begin{rema} The specific choice of metric we used is not
important. If we use a different \emph{conformal metric}, then we
obtain a different class of backward contracting maps for which the
results in this paper hold with the same proof.
Observe that the Koebe principle is still valid independently of the choice of the conformal metric used to measure norms,
since the ratio of two different conformal
metrics is a positive continuous function.
Furthermore, cross-ratios are only affected by a bounded multiplicative constant, so the results from~\S\ref{ss:modulus} remain essentially unchanged.
\end{rema}

For a map $f$ in~$\sA$ and~$\alpha > 0$, a \emph{conformal measure
of exponent~$\alpha$} \index{conformal measure} for~$f$ is a Borel probability measure on~$\dom(f)$ such that for each Borel set~$U$ on which~$f$ is
injective,
$$ \mu(f(U)) = \int_U |Df|^\alpha d\mu. $$
On the other hand, the \emph{conical Julia set~$\Jcon(\cdot)$
of~$f$} \index{conical Julia set} \index{$\Jcon(f)$}
is the set of all those points~$x\in J(f)$ for which
there is a constant~$\delta > 0$ and infinitely many integers~$m \ge 1$
satisfying the following property: $f^m$ is a diffeomorphism
between the connected component of~$f^{-m}(B(f^m(x), \delta))$
containing~$x$ and~$B(f^m(x), \delta)$.

For a subset~$A$ of~$\R$, or of the Riemann sphere~$\CC$,
we denote by~$\HD(A)$ \index{$\HD(\cdot)$, $\overline{\text{BD}}(\cdot)$, $\underline{\text{BD}}(\cdot)$} the \emph{Hausdorff dimension} of~$A$ and by
$\overline{\text{BD}}(A)$ (resp. $\underline{\text{BD}}(A)$) the
upper (resp. lower) box counting dimensions. See for
example~\cite{Fal90} for the definitions.
The Hausdorff dimension of a Borel probability measure $\mu$ on
$X=\R$ (or~$\CC$) is defined as
$$\HD(\mu)
\=
\inf \{\HD(Y): Y\subset X, \mu(Y)=1\}.$$

By definition, a \emph{hyperbolic set}~$A$ \index{hyperbolic set}
of $f \in \sA$ is a
forward invariant compact set on which~$f$ is uniformly
expanding.
The \emph{hyperbolic dimension} \index{hyperbolic dimension} of a map $f \in \sA$ is by definition,
\begin{equation*}
\HDhyp(f) \= \sup \left\{ \HD(A) : A \text{ is a hyperbolic
set}\right\}.
\end{equation*}
\index{$\HDhyp(\cdot)$}

Recall that for $f\in \sA_{\C}$, the map $f: J(f)\to J(f)$ is
topologically exact, \emph{i.e.}, for any nonempty open subset~$U$
of~$J(f)$ there exists an integer~$N\ge 1$ such that $f^N(U)=J(f)$.
It is however too restrictive to assume an interval map to be
topologically exact on the Julia set and boundary-anchored
simultaneously.
For this reason we introduce the following definition.
We say that a map $f\in\sA_\R$ is \emph{essentially topologically exact on $J(f)$} \index{essentially topologically exact}
if there exists a forward invariant
compact interval~$X_0$ containing all critical points of~$f$ such
that $f: J(f|X_0)\to J(f|X_0)$ is topologically exact and such that
the interior of the compact interval $\dom (f)$ is contained in
$\bigcup_{n=0}^\infty f^{-n}(X_0)$.

\begin{generic}[Main Theorem]
For every $\ell > 1, h > 0, \varepsilon>0$, $\gamma > 1$ and~$p \in \left(0, \tfrac{\ell}{\ell - 1} \right)$ there is a constant~$r > 1$
such that the following properties hold. Let $f \in \sA(\ell)$ be
backward contracting with constant~$r$, expanding away from critical
points, and such that $\HD(J(f))\ge h$.
Suppose furthermore in the case $f\in \sA_{\R}$ that~$f$ is essentially topologically exact on the Julia set.
Then the following hold:
\begin{enumerate}
\item[1.]
The hyperbolic dimension $\HDhyp(f)$ is equal to $\HD(J(f))$ and there is a conformal measure~$\mu$ of exponent~$\HD(J(f))$ for~$f$ that is ergodic, supported on $\Jcon(f)$, satisfies $\HD(\mu)=\HD(J(f))$ and is such that for every sufficiently small~$\delta > 0$ and every~$x \in J(f)$,
\begin{equation}\label{e:regularity conformal}
\mu (B(x, \delta))
\le
\delta^{\HD(J(f)) - \varepsilon}.
\end{equation}
Furthermore, any other conformal measure for~$f$ supported on~$J(f)$ is of exponent strictly larger than~$\HD(J(f))$ and supported on a set of Hausdorff dimension less than~$h$.
\item[2.]
There is a unique invariant probability measure~$\nu$ that is absolutely continuous with respect to~$\mu$, and this invariant measure is polynomially mixing of exponent~$\gamma$.
Furthermore, the density of~$\nu$ with respect to~$\mu$ belongs to~$L^p(\mu)$.
\end{enumerate}
\end{generic}

Note that if $J(f)$ has positive Lebesgue measure, then the measure~$\mu$ is proportional to the Lebesgue measure, since after suitable normalization the Lebesgue measure on $J(f)$ is clearly a conformal measure of exponent $\HD(J(f))$.
In fact, this is already the case if~$J(f)$ has the same Hausdorff dimension as the domain of $f$.
See part~$1$ of Corollary~\ref{coro:inducing} in~\S\ref{sss:conformal and invariant}.

Note that~\eqref{e:regularity conformal} implies that for each~$\eta
\in (0, \HD(J(f)) - \varepsilon)$ the conformal measure~$\mu$
satisfies the following integrability condition, first introduced
in~\cite{Prz98}; see also~\cite[\S$10$]{GraSmi09}.
There is a constant~$C > 0$ such that for each~$x_0 \in J(f)$,
\begin{equation}
  \label{e:integrability condition}
  \int_{J(f)} \dist(x, x_0)^{-\eta} d\mu(x) \le C.
\end{equation}

Let us now deduce Theorems~\ref{t:super-polynomial real}
and~\ref{t:super-polynomial complex} from the Main Theorem and the following fact.
\begin{fact} \label{prop:ldbc}
A map~$f$ is backward contracting if one of the following holds:
\begin{itemize}
\item[1.] $f\in \sA_{\R}$ and for all $c\in\Crit'(f)$, we have $|Df^n(f(c))|\to \infty$ as $n\to\infty$;
\item[2.] $f\in \sA_{\C}$ is a polynomial that is at most finitely renormalizable and is such that for all $c\in\Crit'(f)$, we have $|Df^n(f(c))| \to \infty$ as $n \to \infty$;
\item[2'.]
$f\in \sA_{\C}$ is a rational map such that for all $c\in\Crit'(f)$, we have
$\sum_{n=0}^\infty |Df^n(f(c))|^{-1}< \infty$.
\end{itemize}
\end{fact}
\begin{proof}
These are~\cite[Theorem~$1$]{BruRivShevSt08},
\cite[Theorem~A]{LiShe10b} and~\cite[Theorem~A]{Riv07},
respectively. The first result is stated for maps without a periodic
attractor, but the proof works without change under the current
assumption.
\end{proof}
\begin{proof}[Proof of Theorem~\ref{t:super-polynomial real}]
We may extend $f:X\to X$ to be a boundary-anchored map
$\tilde{f}:\tX\to \tX$ with all periodic points repelling,
$\Crit(\tilde{f})=\Crit(f)$ and such that $\textrm{int}(\tX)\subset
\bigcup_{n=0}^\infty \tilde{f}^{-n}(X)$.
Then~$\tilde{f}$ is essentially topologically exact on its Julia set.
By part~$1$ of Fact~\ref{prop:ldbc}, $\tilde{f}$ is backward contracting.
By Ma{\~n}e's theorem, $\tilde{f}$ is expanding away from critical points.
So, by the Main Theorem, $\tilde{f}$ has an invariant measure~$\nu$ that is absolutely continuous with
respect to the Lebesgue measure, and that is super-polynomially mixing.
(Since $J(\tilde{f})=\tX$ has positive measure, the conformal measure $\mu$
is proportional to the Lebesgue measure on~$\tX$.)
Note that~$\nu$ is supported on~$X$, so it is an invariant measure of~$f$ with the desired properties.
\end{proof}
\begin{proof}[Proof of Theorem~\ref{t:super-polynomial complex}]
By parts~$2$ and~$2'$ of Fact~\ref{prop:ldbc}, $f$ is backward contracting.
The fact that~$f$ is expanding away from critical points is known: this follows from either~\cite{KozvSt09} or~\cite[Corollary~$8.3$]{Riv07} in the first case and from \cite[proof of Lemma~$3.1$]{Prz98} in the second case.
Since $\HD(J(f))>0$, applying the Main Theorem completes the proof.
\end{proof}
\begin{rema}
\label{r:Fibonacci}
Let~$f_0(z) = z^2 + c$ be the Fibonacci quadratic polynomial studied in~\cite{LyuMil93}.
This map satisfies the summability condition for every exponent~$\alpha > 0$,
$$ \sum_{n = 1}^\infty \frac{1}{|Df_0^n(c)|^{\alpha}} < \infty, $$
see~\cite[Lemma~$5.9$]{LyuMil93}.
Using the results of~\cite{LyuMil93} we will show that for every~$\alpha \in (0, 1)$,
\begin{equation}
\label{e:divergence}
\sum_{n = 1}^\infty \frac{n}{|Df_0^n(c)|^{\alpha}} = \infty.
\end{equation}
Let $ \{ u(k) \}_{k = 1}^\infty$ be the \emph{sequence of Fibonacci numbers} defined by~$u(1) = 1$, $u(2) = 2$, and for~$k \ge 3$ defined recursively by~$u(k) = u(k - 1) + u(k - 2)$.
If we put~$\varphi = (1 + \sqrt{5})/2$, then an induction argument shows that for every integer~$k \ge 1$ we have~$u(k) \ge \varphi^{k - 1}$.
By~\cite[Lemma~$5.8$]{LyuMil93} there is a constant~$C > 0$ such that for every~$k \ge 1$ we have~$|Df^{u(k)}(c)| \le C 2^{2k/3}$ and thus for each~$\alpha \in (0, 1)$,
$$ \frac{u(k)}{|Df^{u(k)}(c)|^{\alpha}}
\ge
C \varphi^{k - 1} 2^{-2 \alpha k / 3}
\ge
C \varphi^{-1} (\varphi 2^{-2/3})^k \ge C \varphi^{-1}. $$
This proves~\eqref{e:divergence}.
\end{rema}
\subsection{Reduced statements}
\label{ss:reduced statements}

\subsubsection{Polynomial Shrinking}
\label{subsec:psc}

\begin{defi}\label{def:psc}
Given a sequence $\Theta=\{\theta_n\}_{n=1}^\infty$ of positive
numbers, we say that a map~$f \in \sA$ satisfies the
\emph{$\Theta$-Shrinking Condition}, \index{$\Theta$-Shrinking Condition}
if there exist constants~$\rho> 0$ and~$C>0$ such that for every~$x \in J(f)$ and every integer~$m \ge
1$, the connected component~$W$ of~$f^{-m}(B(f^m(x), \rho))$
containing~$x$ satisfies
$$ \diam(W) \le C \theta_m. $$

Given~$\beta \ge 0$ we say that $f$ satisfies the \emph{Polynomial
Shrinking Condition with exponent~$\beta$}, \index{Polynomial Shrinking Condition} if $f$ satisfies the
$\Theta$-Shrinking Condition with $\Theta \= \{ n^{-\beta} \}_{n =
1}^{\infty}$.
\end{defi}
\begin{theoalph}\label{t:polynomial shrinking}
For every $\ell > 1$ and $\beta > 0$ there is a constant~$r > 1$ such that each map in $\sA(\ell)$ that is expanding away from critical points and that is backward contracting with constant~$r$ satisfies the Polynomial Shrinking Condition with exponent~$\beta$.
\end{theoalph}
In what follows, for a map $f \in \sA$ we denote by~$\beta_{\max}(f)$ \index{$\beta_{\max}(\cdot)$} the best polynomial shrinking exponent of~$f$; i.e., the supremum of all $\beta\ge 0$ for which~$f$ satisfies the Polynomial Shrinking Condition with exponent~$\beta$.
So~$\beta_{\max}(f) = 0$ means that~$f$ is not polynomially shrinking with any positive exponent.\footnote{Note that every map in~$\sA$ satisfies the Polynomial Shrinking Condition with exponent~$\beta = 0$.}

The following result is a direct consequence of
Theorem~\ref{t:polynomial shrinking}, of \cite[Theorem~$2$]{Mih11},
and of~\cite[Corollary~$8.3$]{Riv07}.
\begin{coro}[Local connectivity]
\label{c:local connectivity} For every integer~$\ell \ge 2$ there
is an~$r > 1$ such that, for any $f\in \sA_{\C}(\ell)$ that is backward contracting with constant~$r$ the Julia set of~$f$ is locally connected when it is connected.
\end{coro}

\subsubsection{Badness exponent}
\label{subsec:be}
Let us start by introducing ``nice sets''.\footnote{In the case~$f$ is an interval map, the concept of nice set we use here differs from the usual concept of ``nice interval''. A nice interval is an interval~$V$ such that for every integer~$n \ge 1$, we have~$f^n(\partial V) \cap V = \emptyset$.
Thus, a nice set is a neighborhood of~$\Crit'(f)$ formed by a union of nice intervals that satisfy some additional properties.}
For~$f \in \sA$, a
set~$V$, and an integer~$m \ge 1$, each connected component
of~$f^{-m}(V)$ is called a \emph{pull-back of~$V$ by~$f^m$}. \index{pull-back}
\begin{defi}
  For a map $f \in \sA$,
  we will say that $V\subset \dom (f)$ is a \emph{nice set} \index{nice set} if the following hold:
  \begin{itemize}
  \item
$\overline{V}$ is disjoint from the forward orbits
  of critical points not in~$J(f)$ and periodic orbits not in $J(f)$;
  \item $V\supset \Crit'(f)$;
\item
each connected component of~$V$ is an open interval (resp. topological disk) and
  contains precisely one point in $\Crit'(f)$;
  \item
for every integer~$n \ge 1$ we have~$f^n(\partial V) \cap V = \emptyset$.
  \end{itemize}
For~$c \in \Crit'(f)$ we denote by~$V^c$ the connected component
of~$V$ containing~$c$. A nice set~$V$ is called \emph{symmetric} \index{symmetric}
if for each~$c \in \Crit'(f)$ we have $f(\partial V^c) \subset \partial
f(V^c)$.
  Moreover, a \emph{(symmetric) nice couple for $f$} \index{nice couple}  is a pair of (symmetric) nice sets~$(\hV, V)$
  such that $\overline{V} \subset \hV$, and such that for every integer~$n \ge 1$ we have~$f^n(\partial V) \cap \hV = \emptyset$.
\end{defi}

The following fact is proved for maps in~$\sA_{\C}$ in
\cite[Proposition~$6.6$]{Riv07}.
See Lemma~\ref{l:very nice} for the
general case.
\begin{fact}\label{f:nice couples}
For each~$\ell > 1$ there is a constant~$r > 1$ such that each $f
\in \sA(\ell)$ that is backward contracting with constant~$r$
possesses arbitrarily small (symmetric) nice couples.
\end{fact}

Fix~$f \in \sA$ and a set~$V$.
If $W$ is a pull-back of~$V$ by $f^m$, we define an integer~$d_V(W)
\ge 1$ \index{$d_{\cdot}(\cdot)$}
in the following way:
\begin{itemize}
\item
If $f$ is a rational map, then~$d_V(W)$ is the degree of~$f^m : W
\to f^m(W)$, \emph{i.e.}, the maximal cardinality of $f^{-m}(x)\cap
W$ for $x\in V$.
\item
If~$f$ is an interval map, then $d_V(W) \= 2^{N}$, where $N$ is the
number of those $j \in \{ 0, \ldots, m - 1 \}$ such that the
connected component of~$f^{-(m-j)}(V)$ containing~$f^j(W)$
intersects $\Crit(f)$.
\end{itemize}
For a component $W$ of $V$, we define $d_V(W)=1$. When~$V$ is clear from the context, we shall often drop the
subscript $V$, and write $d(W)$ instead of $d_V(W)$.

Let~$V$ be an open set and let~$W$ be a pull-back of~$V$ by $f^m$.
If~$f^{m}$ is a diffeo\-morphism between~$W$ and a component
of~$V$, then we say that~$W$ is a \emph{diffeomorphic pull-back
of~$V$}. Note that in the case when~$f$ is a rational map, this
occurs if and only if~$f^{m}$ is univalent on~$W$. \index{diffeomorphic pull-back}

\begin{defi}
\label{def:badpb}
Given~$f \in \sA$ and an open set~$V$, we will say that a
pull-back~$W$ of~$V$ by $f^m$, $m\ge 1$, is \emph{bad} \index{bad pull-back} if, for every
integer~$m' \in \{ 1, \ldots, m \}$ such that~$f^{m'}(W) \subset~V$,
the pull-back of~$V$ by~$f^{m'}$ containing~$W$ is not
diffeomorphic.
Furthermore, we denote by~$\fB_m(V)$ \index{$\fB_m(\cdot)$} the collection of
all bad pull-backs of~$V$ by $f^m$ and put
$$\delta_{\bad} (V)
\= \inf \left\{ t>0: \sum_{m=1}^\infty\sum_{W\in \fB_m(V)}
d_V(W)\diam (W)^t <\infty \right\}.$$
\end{defi}
\index{$\delta_{\bad}(\cdot)$}
The \emph{badness exponent of~$f$} \index{badness exponent} is defined as
\begin{equation}\label{eqn:badexp}
\delta_{\bad} (f) \= \inf \{\delta_{\bad}(V): V \text{ is a nice set of }f\}.
\end{equation}

We shall prove in Lemma~\ref{lem:db} that $\delta_{\bad}(V)\le
\delta_{\bad}(V')$ for any nice sets $V\subset V'$.
Thus if we have a sequence of nice sets $V_1\supset V_2\supset \cdots \searrow\Crit'(f)$, then $\delta_{\bad}(f)=\lim_{n \to \infty}
\delta_{\bad}(V_n)$.

\begin{theoalph}\label{t:finiteness of bad}
For every $\ell > 1$ and $t > 0$ there is a constant~$r \ge 2$ such
that for each map $f \in \sA(\ell)$ that is backward contracting
with constant $r$, we have $\delta_{\bad}(f)<t$.
\end{theoalph}

\begin{rema}
\label{r:fat Fibonacci}
Given~$\ell > 2$ close to~$2$, let~$f$ be a Fibonacci unimodal map whose critical point~$c$ has order~$\ell$.
Then~$f$ gives an example of a map that is backward contracting with a large constant and such that~$\delta_{\bad}(f) > 0$.
In fact, the results of~\cite{KelNow95} imply that such a map $f$ is backward contracting with a large constant while the postcritical set~$\omega(c)$ has positive Hausdorff dimension.
On the other hand, $f$ is persistently recurrent, so $\omega(c)$ is contained in~$J(f) \setminus \Jcon(f)$.
Thus, by Lemma~\ref{lem:confmeas},
$$ \delta_{\bad}(f)
\ge
\HD(J(f) \setminus \Jcon(f))
\ge
\HD(\omega(c))
>
0. $$
\end{rema}
\begin{rema}\label{r:TCE badness exponent}
The arguments in Lemmas~$7.1$ and~$7.2$ of~\cite{PrzRiv07} show that the
badness exponent of a rational map satisfying the Topological Collet-Eckmann (TCE) condition is zero.
When restricted to the class of rational maps with a unique critical point in the Julia set, the TCE condition is equivalent to the Collet-Eckmann condition~\cite{Prz00}.
So Theorem~\ref{t:finiteness of bad} is significantly stronger within this class of maps.
\end{rema}

In view of the results that follow, it would be interesting to have an answer for the following question:
\begin{ques}
For $f\in\sA$, does~$\beta_{\max}(f)=\infty$ imply $\delta_{\bad}(f)=0$?
\end{ques}

\subsubsection{Canonical induced Markov map}
\label{sss:inducing}
Let~$\sA^*$ \index{$\sA^*$} be the set of $f\in \sA$ that satisfies the following:
\begin{enumerate}
\item [(A1)] $f$ is expanding away from critical points;
\item [(A2)] $\Crit'(f)\not=\emptyset$ and~$f$ has arbitrarily small symmetric nice couples;
\item [(A3)] if $f\in \sA_{\R}$, then $f$ is essentially topologically exact on the Julia set.
\end{enumerate}
Moreover, put~$\sA^*_{\R} \= \sA^* \cap \sA_\R$ and~$\sA^*_{\C} \= \sA^* \cap \sA_\C$.

Through an inducing scheme, we can convert Theorems~\ref{t:polynomial shrinking} and~\ref{t:finiteness of bad} into statistical properties
of maps $f\in \sA^*$.
The following
definitions appeared first in~\cite{PrzRiv07}.

Given a nice couple $(\hV, V)$ of $f$, we say that an integer~$m \ge
1$ is \emph{a good time} for a point~$x$ if~$f^m(x) \in V$ and if
the pull-back of~$\hV$ containing~$x$ is diffeomorphic. We denote
by~$D$ the set of all those points in~$V$ having a good time, and
for each~$x \in D$ we denote by~$m(x)$ the least good time of~$x$.
Note that $m(x)$ is constant in any component $W$ of $D$, so $m(W)$
makes sense.
We say that~$m(x)$ (resp. $m(W)$) is the \emph{canonical inducing time} \index{canonical inducing time}
of~$x$ (resp.~$W$) with respect to $(\hV, V)$.
The \emph{canonical induced map associated to the nice couple~$(\hV, V)$} \index{canonical induced map}
is by definition the map~$F : D \to V$ defined
by~$F(x) = f^{m(x)}(x)$. We denote by~$J(F)$ the maximal invariant
set of~$F$; that is the set of all those points in~$V$ having infinitely many good times.

We say that a sequence $\{\theta_n\}_{n=1}^\infty$ of positive numbers is \emph{slowly varying} \index{slowly varying}
if
$\theta_n/\theta_{n+1}\to 1$ as $n\to\infty$. For instance, $\{n^{-\beta}\}_{n=1}^\infty$ and $\{\exp({-\sigma n^\alpha})\}_{n=1}^\infty$ are slowly varying for any $\beta, \sigma,\alpha>0$, but for each~$\theta \in (0, 1)$ the sequence $\{\theta^n\}_{n=1}^\infty$ is not slowly varying.

\begin{theoalph}
\label{t:tail estimate}
Fix $f\in\sA^*$. If $\delta_{\bad}(f)<\HD(J(f))$, then $\HD(J(f))=\HD_{\hyp}(f)$.
Moreover, for each sufficiently small nice couple~$(\hV, V)$, the associated canonical induced map $F: D\to V$ satisfies:
$$\HD(J(F)\cap V^c)=\HD(J(f)), \text{ for all } c\in \Crit'(f).$$
Furthermore, fix $t\in (\delta_{\bad}(f), \HD(J(f)))$ and assume
that~$f$ satisfies the $\Theta$-Shrinking Condition for some slowly
varying and monotone decreasing sequence of positive numbers $\Theta
= \{\theta_n\}_{n=1}^\infty$. Then for each sufficiently small
symmetric nice couple $(\hV, V)$,  there exists a constant $\alpha_0
= \alpha_0(\hV, V) \in (t, \HD(J(f)))$ such that, for all $\alpha
\ge \alpha_0$ and $\sigma \in [0, \alpha - t)$, there is a
constant~$C > 0$ such that, for each $Y \subset V$ and each integer~$m \ge 1$,
$$ \sum_{W \in \fD \, : \,  m(W) \ge m, \, W \subset Y} \diam (W)^\alpha
\le
C \diam(Y)^\sigma \sum_{n=m}^\infty \theta_n ^{\alpha - t - \sigma}, $$
where $\fD$ is the collection of all components of~$D$.
\end{theoalph}
\begin{rema}
Of course, the latter part of the theorem is useful only when $\sum_{n = 1}^\infty \theta_n^\eta < \infty$ for some $\eta \in (0, \HD(J(f))-\delta_{\bad}(f))$.
\end{rema}
\begin{rema}
\label{r:exponential tail estimate}
If for an exponentially decreasing sequence~$\Theta$ the map~$f$ in Theorem~\ref{t:tail estimate} satisfies the $\Theta$\nobreakdash-Shrinking Condition, then Theorem~\ref{t:tail estimate} allows one to obtain an exponential tail estimate, as follows.
As~$f$ is certainly super-polynomially shrinking, Theorem~\ref{t:tail estimate} shows that there exists a constant~$\alpha\in (0, \HD(J(f)))$ such that $K\=\sum_{W \in \fD \, : \, m(W)\ge 1}
\diam (W)^\alpha$ is finite, and thus
$$\sum_{W\in\fD \, : \, m(W)\ge m}\diam (W)^{\HD(J(f))}
\le
K \max_{W\in\fD \, : \, m(W)\ge m}\diam (W)^{\HD(J(f))-\alpha}$$
is exponentially small in~$m$.
Notice that we lose control of the exponent significantly.
A similar argument was used in the proof of~\cite[Theorem~C]{PrzRiv07}.
\end{rema}

\subsubsection{Conformal and invariant measures}
\label{sss:conformal and invariant}
Define
\begin{equation}\label{eqn:pf}
\gamma(f) \= \beta_{\max}(f) \left(\HD(J(f))-\delta_{\bad}(f)\right).
\end{equation}
We use the following convention: the product of $+\infty$ with a
real number $a$ is $+\infty$ (resp. $0$, $-\infty$) if $a>0$ (resp.
$a=0$, $a<0$).
So $\gamma(f) > 0$ is equivalent to
$$ \delta_{\bad}(f) < \HD(J(f))
\text{ and }
\beta_{\max}(f) > 0. $$

Since for~$f \in \sA^*$ we have~$\HD(J(f))\ge \HD_{\hyp}(f)>0$, Theorems~\ref{t:polynomial shrinking} and~\ref{t:finiteness of bad} imply that when~$f$ is backward contracting we have~$\gamma(f) = \infty$.

\begin{coroalph}\label{coro:inducing}
For~$f \in \sA^*$ the following properties hold.
\begin{enumerate}
 \item[1.] If $\gamma(f)>1$, then either $\HD(J(f))< \HD (\dom (f))$ or~$J(f)$ has a nonempty interior.
Moreover, there exists a conformal measure~$\mu$ of exponent~$\HD(J(f))$ for~$f$ that is ergodic, supported on the conical Julia set,
satisfies $\HD(\mu)=\HD(J(f))$, and is such that for each~$\varepsilon
> \delta_{\bad}(f) + \beta_{\max}(f)^{-1}$ the following holds: for each sufficiently small $\delta > 0$ we have for every $x \in J(f)$,
\begin{equation}
  \label{e:upper regularity conformal}
\mu(B(x, \delta)) \le \delta^{\HD(J(f)) - \varepsilon}.
\end{equation}
Furthermore, any other conformal measure for~$f$ supported on~$J(f)$ is of exponent strictly larger than~$\HD(J(f))$ and supported on a set of Hausdorff dimension less than or equal to~$\delta_{\bad}(f)$.
 \item[2.] If $\gamma(f)>2$, then there is an invariant probability measure $\nu$
that is absolutely continuous with respect to~$\mu$ and this invariant measure~$\nu$ is polynomially mixing of each exponent~$\gamma \in (0, \gamma(f)-2)$.
\end{enumerate}
\end{coroalph}

\subsubsection{Regularity of the invariant density}
\label{sss:regularity density}
Given~$f \in \sA$, let~$q(f)$ be the infimum of those constants~$q > 0$ for which there is a constant~$C > 0$ such that the following property holds: for each~$x \in J(f)$, $\delta > 0$, $m \ge 1$, and each pull-back~$W$ of~$B(x, \delta)$ by~$f^m$ such that~$f^m : W \to B(x, \delta)$ is a diffeomorphism whose distortion is bounded by~$2$, we have
$$ \diam(W) \le C \delta^{1/q}. $$
We clearly have~$q(f) \ge \ell_{\max}(f)$.
The following is a simple consequence of~\cite[Proposition~$2$]{LiShe08a}, see~Lemma~\ref{lem:blsquantify}.
\begin{fact}\label{f:asymptotic criticality}
For each~$\ell > 1$ and~$q > \ell$ there is a constant~$r > 1$ such that if~$f \in \sA^*(\ell)$ is backward contracting with constant~$r$, then~$q(f) <
q$.
In particular, if~$f$ is backward contracting, then $q(f)=\ell_{\max}(f)$.
\end{fact}
\begin{theoalph}\label{t:regularity density}
Let~$f \in \sA^*$ be such that~$\gamma(f) > 2$, and let~$\mu$ be the
conformal measure and~$\nu$ the invariant measure given by
Corollary~\ref{coro:inducing}. Then for each
$$ p \in \left( 1, q(f) \frac{1 - (\delta_{\bad}(f) + \beta_{\max}(f)^{-1})\HD(J(f))^{-1}}{q(f) - 1 + (\delta_{\bad}(f) + 2 \beta_{\max}(f)^{-1})\HD(J(f))^{-1}} \right), $$
the density of~$\nu$ with respect to~$\mu$ belongs to~$L^p(\mu)$.
\end{theoalph}
\begin{rema}
\label{r:optimality L^p} For a map $f \in \sA^*$ that is backward
contracting, the theorem implies that the density of~$\nu$ with
respect to~$\mu$ is in $L^p(\mu)$ for all
$$ p <
p(f) \= q(f) / (q(f) - 1)
=
\ell_{\max}(f)/(\ell_{\max}(f)-1). $$
If~$J(f)$ has nonempty interior in~$\dom(f)$, then this estimate is optimal in the sense that the density never belongs to the space~$L^{p(f)}(\mu)$, as we shall now explain.
In this case, $\mu$ is a rescaling (of a restriction) of the Lebesgue measure and the Lyapunov exponent of~$\nu$ is strictly positive and its density is bounded from below by
a positive constant almost everywhere in~$J(f)$, see~\cite{Led81} or~\cite[Theorem~$6$]{Dob1304} for the real case, and~\cite{Led84} or~\cite[Theorem~$8$]{Dob12} for the complex case.
It thus follows from the invariance of~$\nu$ and the
conformality of~$\mu$ that, if we denote by~$h \in \{1, 2 \}$ the dimension
of~$\dom(f)$ and by~$c \in J(f)$ a critical point of~$f$ of maximal
order, then there is a constant~$C > 0$ such that the density is bounded from below by the function~$C
\dist(\cdot, f(c))^{- h / p(f)}$, on a set of full
Lebesgue measure in~$J(f)$. Thus the density cannot belong
to~$L^{p(f)}$.

Using the lower bound on the conformal measure given
by~\cite[Theorem~$1$]{LiShe08a}, a similar argument shows that if~$f \in \sA^*_{\C}$
is backward contracting and~$p > p(f)$, then the invariant density
does not belong to $L^p(\mu)$.
\end{rema}
\begin{ques}
Suppose $f\in\sA^*$ is such that~$\gamma(f) = \infty$, and let~$\mu$
and~$\nu$ be as in Corollary~\ref{coro:inducing}. Is it true that
$d\nu/d\mu\not\in L^{p(f)}(\mu)$?
\end{ques}
We state the following corollary for future reference.
For the definition of the TCE condition, see for example~\cite{NowPrz98} in the real case and~\cite{PrzRoh98} in the complex case.
\begin{coro}
\label{c:TCE}
Let~$f \in \sA$ be a map satisfying the TCE condition and that it is not uniformly hyperbolic.
In the real case, assume furthermore that~$f$ is essentially topologically exact on~$J(f)$.
Then~$f$ belongs to~$\sA^*$ and~$\gamma(f) = \infty$.
Moreover, if we denote by~$\mu$ the conformal measure and~$\nu$
the invariant measure given by Corollary~\ref{coro:inducing}, then~$\nu$ is exponentially mixing, for each~$p \in \left( 0, q(f)/ (q(f) - 1) \right)$ the density of~$\nu$ with respect to~$\mu$ belongs to~$L^p(\mu)$, and for each~$\varepsilon > 0$ we have for every sufficiently small~$\delta > 0$ and every~$x \in J(f)$
$$ \mu(B(x, \delta)) \le \delta^{\HD(J(f)) - \varepsilon}. $$
\end{coro}
When~$f$ is in~$\sA_{\C}$ the assertion that~$\nu$ is exponentially mixing is shown in~\cite[Theorem~C]{PrzRiv07}; the remaining assertions of the corollary are new.
\begin{proof}
By~\cite{NowPrz98} in the real case and~\cite{PrzRoh98} in the complex one, there is an exponentially decreasing sequence~$\Theta$
such that~$f$ satisfies the~$\Theta$\nobreakdash-Shrinking Condition.
Thus $\beta_{\max}(f) = \infty$ and~$f$ is expanding away from critical points.
On the other hand, $f$ has arbitrarily small
nice couples by~\cite[Theorem~E]{PrzRiv07}, and we also have~$\delta_{\bad}(f) = 0$, as pointed out in Remark~\ref{r:TCE badness exponent}.
This proves that~$f$ is in~$\sA^*$ and that~$\gamma(f) = \infty$.
In particular, $f$ satisfies the hypotheses of Corollary~\ref{coro:inducing} and Theorem~\ref{t:regularity density}.
Denote by~$\mu$ the conformal measure and~$\nu$ the invariant measure given by Corollary~\ref{coro:inducing}.
That~$\nu$ is exponentially mixing is given by the exponential tail estimate in Remark~\ref{r:exponential tail estimate}, combined with well-known arguments (similar to those used in the proof of part~$2$ of Corollary~\ref{coro:inducing}).
The remaining assertions of the corollary are given by Corollary~\ref{coro:inducing} and Theorem~\ref{t:regularity density}.
\end{proof}

\subsection{Proof of the Main Theorem}\label{ss:proof of Main Theorem}
Now we will complete the proof of the Main Theorem.
If $\Crit'(f)=\emptyset$, then~$f$ is uniformly expanding on $J(f)$ and
the statements of the Main Theorem are well-known. So we assume that
$\Crit'(f)\not=\emptyset$.

Let $\ell > 1$, $h > 0$, $\varepsilon > 0$, $\gamma> 1$ and  $p
\in \left( 0, \tfrac{\ell}{\ell - 1} \right)$ be given, and let~$q
\in \left( \ell, \tfrac{p}{p - 1} \right)$.
By Theorems~\ref{t:polynomial shrinking} and~\ref{t:finiteness of bad}
and by Facts~\ref{f:nice couples} and~\ref{f:asymptotic criticality}, there exists a constant~$r>2$ such that if~$f \in \sA(\ell)$ is backward contracting with constant~$r$ and satisfies the other assumptions of the Main Theorem, then $f\in\sA^*(\ell)$,
$\beta_{\max}(f)> 2(\gamma+2)/h$, $\delta_{\bad}(f)<h/2$,
$\delta_{\bad}(f) + \beta_{\max}(f)^{-1} < \varepsilon$, $q(f) < q$ and
$$ q \frac{1 - (\delta_{\bad}(f) + \beta_{\max}(f)^{-1})\HD(J(f))^{-1}}{q - 1 + (\delta_{\bad}(f) + 2 \beta_{\max}(f)^{-1})\HD(J(f))^{-1}} > p. $$
Hence $\delta_{\bad}(f) < \HD(J(f))$ and by Theorem~\ref{t:tail
estimate} we have~$\HD(J(f)) = \HD_{\hyp}(f)$.
On the other hand $\gamma(f)> \gamma+2$, and by Corollary~\ref{coro:inducing} and
Theorem~\ref{t:regularity density} there exist a conformal
measure~$\mu$ and an invariant measure~$\nu$ with the desired
properties.
Moreover, the ergodicity of~$\mu$ implies that~$\nu$ is
the only invariant measure of~$f$ that is absolutely continuous with
respect to~$\mu$.

\subsection{Fractal dimensions and holomorphic removability of Julia sets}
\label{ss:further results} In this section we state a result related
to fractal dimensions (Theorem~\ref{t:fractal dimensions}), and
another related to holomorphic removability of Julia sets in the
complex setting (Theorem~\ref{t:holomorphic removability}). Both are
independent of the Main Theorem and are shown in~\S\ref{s:regularity
density}.

To state our result on the equality of fractal dimensions, we make
the following definition. Given~$f \in \sA$, $s> 0$ and a point~$x_0
\in \dom(f)$ we define the \emph{Poincar{\'e} series of~$f$ at~$x_0$
with exponent~$s$} \index{Poincar\'e series} \index{$\cP(x_0;s)$}, as
$$ \cP(x_0; s) \= \sum_{n = 1}^{\infty} \sum_{x \in f^{-n}(x_0)} | Df^n(x)|^{-s}. $$
We say that a point~$x$ is \emph{exceptional} \index{exceptional}
if the set
$\bigcup_{n=0}^\infty f^{-n}(x)$ is finite, and we say that~$x$ is
\emph{asymptotically exceptional} \index{asymptotically exceptional}
if its $\alpha$-limit set is
finite. The \emph{Poincar{\'e} exponent of~$f$} \index{Poincar\'e exponent}
is by definition,
\begin{multline*}
\poincare(f)
\=
\inf \left\{ \{ 0 \} \cup \left\{ s>0: \cP(x_0;s)<\infty \text{ for some~$x_0$}
\right. \right. \\ \left. \left.
\text{that is not asymptotically exceptional} \right\} \right\}.
\end{multline*}
Note that every point in the $\alpha$-limit set of an asymptotically
exceptional point is exceptional. It is well-known that for a
rational map of degree at least~$2$ each asymptotically exceptional
point is exceptional, that there are at most~$2$ exceptional points,
and that they are not in the Julia set.
Note however that for~$f\in\sA_\R$, any point in~$\dom(f)\setminus X_0(f)$ is
asymptotically exceptional, where~$X_0(f)$ is the minimal forward
invariant closed interval that contains~$\Crit(f)$.

\begin{theoalph}[Equality of fractal dimensions]
\label{t:fractal dimensions} If~$f \in \sA^*$ satisfies
$\gamma(f)>1$, then
$$\poincare(f) =
\overline{\text{BD}}(J(f))=\HD(J(f))=\HD_{\hyp}(f)>0.$$
\end{theoalph}
See~\eqref{eqn:pf} in~\S\ref{ss:reduced statements} for the definition of~$\gamma(f)$ and Proposition~\ref{p:Poincare series} for some divergence/convergence properties of the Poincar{\'e} series.

Equalities of dimensions were shown in~\cite{LiShe08a} for backward
contracting rational maps without parabolic cycles, in~\cite{Prz98} for rational maps whose
derivatives at critical values grow at least as a stretched
exponential function, in~\cite[Theorem~$7$]{GraSmi09} for
rational maps satisfying a summability condition with a small
exponent and without parabolic cycles, and in~\cite{Dob06} for interval maps without
recurrent critical points.
These equalities were shown for a class of
infinitely renormalizable quadratic polynomials in~\cite{AviLyu08}.

We will say that a compact subset~$J$ of the Riemann sphere is
\emph{holomorphically removable} \index{holomorphically removable}
if every homeomorphism $\varphi :
\CC \to \CC$ that is holomorphic outside~$J$ is a M{\"o}bius
transformation.

\begin{theoalph}[Holomorphic removability]\label{t:holomorphic removability}
If $f\in \sA_{\C}^*$ is a polynomial such that
$$ \beta_{\max}(f)\left(2-\delta_{\bad}(f)\right)>1, $$
then the Julia set of~$f$ is holomorphically removable.
In particular, for every integer $\ell \ge 2$, there is a
constant~$r > 1$ such that the Julia set of a complex
polynomial~$f \in \sA_{\C}(\ell)$ that is backward contracting with
constant~$r$ is holomorphically
removable.
\end{theoalph}
See~\cite{Jon95,Kahremovability} and~\cite[Theorem~$8$]{GraSmi09} for
other removability results of Julia sets.

\section{Preliminaries}

\subsection{Notation}
\label{ss:notation}
For~$f \in \sA$ and for a neighborhood~$V$ of~$\Crit'(f)$ we put
\begin{equation}
  \label{eqn:defKV}
  K(V) \= \{ x \in \dom(f) : \text{ for every integer $n \ge 0$, } f^n(x) \not \in V \}.
\end{equation}
\index{$K(V)$}
Denote by~$\fL_V$ the collection of connected components of~$\dom(f) \setminus K(V)$.
\index{$\fL_V$}

For $V\subset \dom (f)$ and an integer $m\ge 0$, let $\sM_m(V)$ \index{$\sM_m(\cdot)$, $\sM(\cdot)$} denote
the collection of all components of $f^{-m}(V)$. Moreover, let
$\sM(V) \= \bigcup_{m=0}^\infty \sM_m(V)$.

\subsection{Koebe distortion lemma}
We shall frequently use the following Koebe distortion lemma.
\begin{lemm}\label{lem:koebeC}
For every~$f\in \sA$ there is a constant~$\eta_* > 0$ such that for every~$\varepsilon\in (0, 1)$ there exists a constant~$K(\eps)>1$ such that the following holds.
Given $x\in \dom(f)$, $\eta\in (0, \eta_*)$, and $n \ge 1$, let~$W$ (resp.~$W(\eps)$) be the component of $f^{-n}(B(f^n(x), \eta))$ (resp. $f^{-n}(B(f^n(x), \eps \eta)$) that contains~$x$.
Suppose that $f^n:W\to B(f^n(x),\eta)$ is a diffeomorphism. Suppose also that $\dist(f^n(x), J(f))\le \eta_*$ in the case $f\in\sA_\R$.
Then the distortion of~$f^n$ on~$W(\eps)$ is bounded by $K(\eps)$.
That is, for every $z_1, z_2\in W(\eps)$,
$$|Df^n(z_1)|/|Df^n(z_2)|\le K(\eps).$$
Moreover, $K(\eps)= 1+O(\eps)$ as $\eps\to 0$.
\end{lemm}
\begin{proof}
For the case $f\in \sA_{\C}$, see for
example~\cite[\S$2.4$]{PrzRiv07}. For the case $f\in \sA_{\R}$, see
~\cite[Theorem C ($2$)(ii)]{vStVar04}.
Recall that, by definition, maps in~$\sA_{\R}$ have no neutral cycles.
\end{proof}
\subsection{Modulus}\label{ss:modulus}
We shall use moduli of annuli to compare the size of nested sets.
This method is popular in the complex setting. Recall that if $A
\subset \CC$ is an annulus that is conformally isomorphic to the
round annulus $\{z\in \C: 1<|z|<R\}$, then $\mmod(A)\=\log R$. More
generally, if~$V\subsetneq \CC$ is open and~$U\Subset V$, then we
define~$\mmod(V;U)$ as the supremum of the moduli of those annuli
contained in~$V$ that separate $U$ from $\CC\setminus V$. \index{$\mmod(V;U)$}

In order to deal with interval maps, let us introduce a similar
notion in the real setting. If $J\Subset I$ are bounded intervals in
$\R$, then we define $\mmod(I;J) \= \mmod(D_*(I); D_*(J))$, where $D_*(I)$ denotes the
round disk in~$\C$ that has~$I$ as a diameter and~$D_*(J)$ the
corresponding disk for~$J$.
More generally, if~$V$ is a bounded open subset of~$\R$, $U \Subset V$ is an interval, and if we denote by~$V_0$ the connected component of~$V$ containing~$U$, then we put
$$ \mmod(V; U) \= \mmod(V_0;U). $$

\begin{lemm}
\label{lem:grotzch}
Let $V_2\Supset V_1 \Supset V_0$ be either bounded intervals
of~$\R$, or open and connected proper subsets of~$\C$. Then
$$\mmod(V_2; V_0)
\ge \mmod(V_2; V_1) + \mmod(V_1; V_0).$$
\end{lemm}
\begin{proof} In the complex case, this lemma is known as Grotzch's
inequality, see for example \cite[Corollary~A.$5$]{Mil00b}. The real
case follows from the complex one by definition.
\end{proof}
The following lemma relates modulus to diameter of sets.
\begin{lemm}\label{lem:modsize}
There exists a constant~$C_0>0$ such that the following property
holds.
Let $U \Subset V$ be either bounded intervals contained in~$\R$, or open and connected subsets of~$\CC$,
such that $\diam(V)\le \diam(\CC)/2$.
Then, letting $\mu=\mmod(V; U)$,
$$\diam (U)\le C_0\exp(-\mu) \diam (V).$$
\end{lemm}
\begin{proof} We only need to consider the complex case as the real
case will then follow by definition. We may assume that $\mu$ is
large.
In this case, $V\setminus U$ contains a round annulus~$A$
with $\mmod(A)=\mu-O(1)$. See for example \cite[Theorem~$2.1$]{McM94}.
The lemma follows.
\end{proof}

We shall now consider distortion of modulus under pull-back. In the
complex case, we have the following well-known lemma, see for
example~\cite[Lemma~$4.1.1$]{GraSwi98}.
A sequence $\{U_j\}_{j=0}^s$ of simply connected open sets is
called \emph{a chain} \index{chain}
if for each $0\le j<s$, the set~$U_j$ is a component of
$f^{-1}(U_{j+1})$ and $U_j\cap J(f)\not=\emptyset$.

\begin{lemm}\label{lem:Cmodulus}
Consider $f\in\sA_{\C}$.
Let $\{\tU_j\}_{j=0}^s$ and
$\{U_j\}_{j=0}^s$ be chains of topological disks such that $U_j
\Subset \tU_j$, and let
$$\nu
=
\#\{0\le j<s: \tU_j \text{ intersects } \Crit'(f)\}.$$
Then
$$\mmod (\tU_0; U_0)\ge \ell_{\max}(f)^{-\nu}
\mmod(\tU_s; U_s).$$
\end{lemm}

In the real case, modulus is similarly distorted under a diffeomorphic
pull-back, but the situation can be much worse under critical
pull-back.
For example, let $f(x)=x^2+a$ be a real quadratic polynomial, and
$$ \tU_1 \=(-\delta^2+a, \delta^2+a) \supset U_1 \= (-\delta^2 \eps +a, \delta^2 (1-\eps) +a), $$
where $\delta>0$ and $\eps\in (0,1)$.
Then a direct computation shows that, as $\eps\to 0$,
$$ \mmod \left( f^{-1}(\tU_1); f^{-1}(U_1) \right) \asymp \eps
\text{ and }
\mmod(\tU_1; U_1)\asymp \eps^{\tfrac{1}{2}}. $$
Nevertheless, the following Lemma~\ref{lem:Rmodulus} will be enough for our application.

Given~$f \in \sA_\R$, if a bounded open interval~$I$ contains a unique critical value~$v$ of~$f$, we define
\begin{equation}\label{eqn:sharp}
I^{\sharp}=(v - |I|, v +  |I|);
\end{equation}
otherwise, we write $I^\sharp =I$. Moreover, we say that a map
$f\in\sA_\R$ is \emph{normalized near critical points} \index{normalized near critical points}, if for each
$c\in\CJ$, the equation $|f(x)-f(c)|=|x-c|^{\ell_c}$ holds in a
neighborhood of $c$.
\begin{lemm}\label{lem:Rmodulus}
Consider $f\in\sA_{\R}(\ell)$.
\begin{enumerate}
\item[1.] For any $\lambda\in (0,1)$ there exists a constant~$\eta>0$ such that if
$\{\tU_j\}_{j=0}^s$ and $\{U_j\}_{j=0}^s$ are chains of intervals
such that $U_j \Subset \tU_j$, such that
$\tU_j\cap\Crit(f)=\emptyset$ for all $j=1,\ldots, s-1$ and such
that $|\tU_s|<\eta$, then
\begin{equation*}
\mmod (\tU_1; U_1) \ge \lambda \mmod(\tU_s; U_s).
\end{equation*}
Moreover, there exists a constant~$K_0 > 0$ depending only on~$\ell$ such that
\begin{equation*}
\mmod(\tU_0; U_0) + K_0
\ge
\ell^{-1} \mmod(\tU_1; U_1)
\ge
\lambda\ell^{-1} \mmod(\tU_s; U_s).
\end{equation*}
\item[2.] Assume that $f$ is normalized near critical points and let $c\in \Crit'(f)$. Let $V_1\Supset U_1$ be intervals
and let $V_0$ (resp. $U_0$) be a component of $f^{-1}(V^\sharp_1)$
(resp. $f^{-1}(U^\sharp_1)$) such that $U_0\subset V_0$. If $c\in
V_0$ and $|V_1|<\eta$, then
$$\mmod(V_0; U_0)\ge \ell_c^{-1} \mmod(V_1; U_1),$$
where $\ell_c$ is the order of~$f$ at~$c$.
\end{enumerate}
\end{lemm}

We need two preparatory lemmas to prove this result. Recall that the
\emph{cross-ratio} \index{cross-ratio} of bounded intervals~$J \Subset I$ of~$\R$ is
defined as
$$\Cr(I, J) \= \frac{|I||J|}{|L||R|},$$
where $L, R$ are the components of $I\setminus J$.
\begin{lemm}\label{lem:realmodulus}
For bounded intervals $J\Subset I$ of~$\R$,
$$\mmod (I;J)
=
2\log \left( \sqrt{\Cr(I,J)^{-1}}+\sqrt{1+\Cr(I,J)^{-1}}\right).$$
\end{lemm}
\begin{proof} There exists a M{\"o}bius transformation~$\sigma$ such that
$\sigma(I)=(-T, T)$ and $\sigma(J)=(-1,1)$, where $T=\exp(\mmod
(I;J))$. Since $$\Cr(I, J)=\Cr(\sigma(I), \sigma(J))=4T/(T-1)^2,$$
the lemma follows.
\end{proof}

\begin{lemm}\label{lem:realmcriti}
For each $\ell>1$ and $0 \le  a < b < 1$,
$$\mmod((-1,1); (a, b))\ge \ell^{-1}\mmod ((-1,1); (a^\ell, b^{\ell})). $$
\end{lemm}
\begin{proof} Let us first consider the case $a=0$.
Let~$\beta \in (0, 1)$ be such that~$\beta + \beta^{-1} = 2b^{-1}$, so that
$$ \mmod((-1, 1);(-\beta, \beta)) = \mmod((-1, 1);(0, b)). $$
Thus, if we let~$\hat{b} \in (0, 1)$ be defined by~$2 \hat{b}^{-1} = \beta^{-\ell} + \beta^\ell$, then,
\begin{align*}
\mmod((-1, 1); (0, \hat{b}))
& =
\mmod((-1, 1);(-\beta^{\ell}, \beta^\ell))
\\ & =
\ell \mmod((-1, 1);(-\beta, \beta))
\\ & =
\ell \mmod((-1, 1);(0, b)).
\end{align*}
So we just need to show that~$\hat{b} \le b^{\ell}$.
This follows from the power mean inequality,
$$ \hat{b}^{-1} = \frac{\beta^{-\ell} + \beta^\ell}{2}
\ge
\left( \frac{\beta^{-1} + \beta}{2} \right)^\ell = b^{-\ell}, $$
see for example~\cite[$\mathbf{16}$]{HarLitPol52}.

Now let us consider the case $a>0$.
Let~$t$ be the unique number in~$(b,1)$ such that
$$\mmod((-1,1); (a^\ell, b^\ell))
=
\mmod((-1,1); (0, t^\ell))+\mmod((0, t^\ell); (a^\ell, b^\ell)).$$
Then as above,
$$\mmod((-1,1); (0,t))\ge \ell^{-1} \mmod((-1,1); (0,t^\ell)).$$
Note that
$$ \Cr((0, t); (a, b))
=
\frac{\frac{b}{a} - 1}{1 - \frac{b}{t}}
\le
\frac{\left(\frac{b}{a} \right)^{\ell} - 1}{1 - \left( \frac{b}{t} \right)^{\ell}}
=
\Cr ((0, t^\ell),(a^\ell,
b^\ell)), $$
hence
$$\mmod((0, t); (a, b))
\ge
\mmod((0, t^\ell); (a^\ell, b^\ell))
\ge
\ell^{-1} \mmod((0, t^\ell); (a^\ell, b^\ell)). $$
Finally, by Lemma~\ref{lem:grotzch},
$$\mmod((-1,1); (a, b))
\ge
\mmod((-1,1); (0,t))+\mmod((0,t); (a, b)). $$
Combining these estimates, we complete the proof of the lemma.
\end{proof}

\begin{proof}[Proof of Lemma~\ref{lem:Rmodulus}]
\

\partn{1}
For the first inequality, by Lemma~\ref{lem:realmodulus}, it
suffices to prove that for any constant $\lambda' \in (0, 1)$, we
have $\Cr(\tU_s, U_s)\ge \lambda' \Cr(\tU_1, U_1)$, provided that
$\diam (\tU_s)$ is sufficiently small.
But this is well-known: if $f$ has negative Schwarzian derivative, we actually have
$\Cr(\tU_s, U_s)\ge \Cr(\tU_1, U_1)$; otherwise, we may apply \cite[Theorem C(3)]{vStVar04} which claims that the first entry map to a small neighborhood of $f(\Crit'(f))$ has negative Schwarzian derivative,
see the proof of~\cite[Proposition~$1$]{BruRivShevSt08}, for details.

For the
second inequality we may assume $\tU_0$ contains a critical point
$c$ and $\mmod(\tU_1; U_1)$ is large, \emph{i.e.}, $\Cr(\tU_1, U_1)$ is small. Note that $|\tU_s|$ small
implies $|\tU_0|$ small since $f$ has no wandering interval. So by the non-flatness of critical points,
we have $\Cr(\tU_0, U_0)\le K_0'\Cr(\tU_1, U_1)^{1/\ell_c}$ for
some~$K_0'$ depending only on~$\ell_c$. The second inequality
follows by Lemma~\ref{lem:realmodulus} again.

\partn{2} It suffices to prove the following two
inequalities:
\begin{equation}\label{eqn:mod+}
\mmod(V_1^\sharp; U_1^\sharp)\ge \mmod(V_1; U_1);
\end{equation}
\begin{equation}\label{eqn:realmodreg}
\mmod( V_0; U_0)\ge \ell_{c}^{-1} \mmod (V_1^\sharp; U_1^\sharp).
\end{equation}

Let us prove the inequality~\eqref{eqn:mod+}.
If $f(c) \not \in U_1$, then $V_1^\sharp\supset V^1$, $U_1^\sharp=U_1$, so this inequality
clearly holds.
Assume $f(c) \in U_1$ and let $L, R$ denote the
components of $V_1 \setminus U_1$.
Then
$$\Cr(V_1, U_1)=\frac{|V_1||U_1|}{|L||R|}\ge
\frac{4|V_1||U_1|}{(|L|+|R|)^2}=\frac{4|V_1||U_1|}{(|V_1|-|U_1|)^2}=
\Cr(V_1^\sharp, U_1^\sharp),$$
which implies the inequality~\eqref{eqn:mod+} by Lemma~\ref{lem:realmodulus}.

The inequality~\eqref{eqn:realmodreg} follows from the local behavior of~$f$ near~$c$.
If~$U_1 \ni f(c)$, then both~$V_1^\sharp$ and~$U_1^\sharp$ are centered at~$f(c)$ and, by definition of modulus, we see that~\eqref{eqn:realmodreg} holds with equality.
If $U_1\not\ni f(c)$, then the statement follows from Lemma~\ref{lem:realmcriti}.
\end{proof}

\subsection{Bad pull-backs of a nice set}\label{ss:badpullback}
Let~$f \in \sA$ and let~$V$ be a nice set for~$f$. It is easy to see
that for each $W\in \sM(V)$, either $W\cap V=\emptyset$ or $W\subset
V$.
Moreover, for any integers $0\le m_1\le m_2$ and $W_1\in\sM_{m_1}(V)$, $W_2\in\sM_{m_2}(V)$,
$$\mbox{either } W_1\cap W_2=\emptyset\mbox{ or }W_2\subset W_1.$$
Recall that $W\in\sM_m(V)$, $m\ge 1$ is called a bad pull-back~of~$V$ by~$f^m$ if, for every integer~$m' \in \{ 1, \ldots, m \}$ such that~$f^{m'}(W)\subset V$, the pull-back of~$V$ by~$f^{m'}$ containing~$W$ is not diffeomorphic. As before, we use
$\fB_m(V)$ to denote the collection of all bad pull-backs of $V$ by
$f^m$.
\begin{lemm}\label{lem:db0}
Let $V$ be a nice set of $f\in\sA$ and let
$W\in\sM_m(V)$ with $m\ge 1$.
Then the following are equivalent:
\begin{enumerate}
\item[1.] $W\in\fB_m(V)$;
\item[2.] for any $1\le m'\le m$, $W$ is not contained in any
diffeomorphic pull-back of $V$ by $f^{m'}$;
\item[3.] for any $1\le m'\le m$, $W$ is disjoint from any
diffeomorphic pull-back of $V$ by $f^{m'}$.
\end{enumerate}
\end{lemm}
\begin{proof}
By definition, $1~\Leftrightarrow~2$.
The assertion $3~\Rightarrow~2$
is trivial, while $2~\Rightarrow~3$ holds since any pull-back of $V$
by $f^{m'}$ is either disjoint from $W$ or contains $W$.
\end{proof}
\begin{lemm}\label{lem:db}
For nice sets $V'\supset V$ for~$f$, the following properties hold.
\begin{itemize}
\item[1.]
For every integer~$m \ge 1$ and every bad pull-back~$W$ of~$V$
by~$f^m$, the pull-back of~$V'$ by~$f^m$ containing~$W$ is bad.
\item[2.]
$\delta_{\bad} (V')\ge \delta_{\bad}(V)$.
\end{itemize}
\end{lemm}
\begin{proof}
\

\partn{1}
Let~$W'$ be the pull-back of~$V'$ by~$f^m$ that contains~$W$.
Arguing by contradiction, assume that~ $W'\not\in\fB_m(V')$.
Then there exists a~$m_0\in \{1,2,\ldots, m\}$ such that $W'$ is contained in a diffeomorphic pull-back~$W_0$ of~$V'$ by~$f^{m_0}$.
Then $f^{m_0}(W)\not\subset V$, so $m_0<m$ and there exists a minimal $m_1\in \{m_0+1, m'+2, \ldots, m\}$ such that~$f^{m_1}(W)\subset V$.
By the minimality of~$m_1$ and niceness of~$V$, we obtain that $f^{m_1-m_0}$ maps a simply connected set~$U \supset f^{m_0}(W)$ diffeomorphically onto a component of~$V$.
By the niceness of~$V'$ we obtain that $U \subset V'$.
So the pull-back of~$U$ by~$f^{m_0}$ that contains~$W$ is diffeomorphic.
It follows that the pull-back of~$V$ by~$f^{m_1}$ that contains~$W$ is diffeomorphic, contradicting the assumption that~$W\in\fB_m(V)$.

\partn{2}
By part~$1$ each $W\in \fB_m(V)$ is contained in an element of~$\fB_m(V')$.
Clearly, for each $W'\in \fB_m(V')$, and any $t\ge 0$,
$$ \sum_{W\in \fB_m(V) \, : \, W \subset W'} d_V(W)\diam (W)^t
\le
d_{V'}(W')\diam (W')^t.$$
It follows that
$$\sum_{m=1}^\infty\sum_{W\in \fB_m(V)} d_V(W)\diam (W)^t
\le
\sum_{m=1}^\infty \sum_{W'\in \fB_m(V')} d_{V'}(W')\diam (W')^t,$$
hence
$\delta_{\bad}(V')\ge \delta_{\bad}(V)$.
\end{proof}

\subsection{Expansion away from critical points}
\label{ss:expansion away}
\begin{lemm}\label{lem:definitesize}
Let $f\in\sA$ be expanding away from critical points, let~$\rho > 0$ be given, and let~$(\hV, V)$ be a nice couple such that for each~$c \in \Crit'(f)$ we have~$\diam(\hV^c) < \rho$.
Then there are constants~$\kappa_0 > 1$, $K_0 > 1$ and $\rho_0 \in (0, \rho)$, such that the following property holds.
Let~$D$ be the domain of the canonical induced map associated to~$(\hV, V)$.
Then for each~$x \in J(f)$ and $\delta \in (0, \rho/2)$ there is an integer~$n \ge 0$ such that one of the following properties holds:
\begin{enumerate}
\item[1.] the distortion of~$f^n$ on~$B(x, \delta)$ is bounded by~$K_0$, and
$$ \rho_0 < \diam(f^n(B(x, \delta))) < \rho; $$
\item[2.] $f^n(B(x, \kappa_0 \delta)) \subset V$, $|Df^n(x)| \ge \rho_0$, the distortion of~$f^n$ on~$B(x, \kappa_0 \delta)$ is bounded by~$K_0$, and every connected component of~$D$ intersecting~$f^n(B(x, \delta))$ is contained in~$f^n(B(x, \kappa_0 \delta))$.
\end{enumerate}
\end{lemm}
\begin{proof}
Fix a compact neighborhood~$\tV$ of~$\overline{V}$ contained
in~$\hV$. For each $m\ge 1$ and each component $W$ of $f^{-m}(V)$,
let $\hW\supset \tW$ be the components of $f^{-m}(\hV)\supset
f^{-m}(\tV)$ that contain $W$.
By the Koebe principle there are
constants~$\kappa_0 > 1$, $K_1 > 1$ and~$\rho_1 > 0$ such that if~$f^m|\hW$ is a diffeomorphism onto
a connected component of~$\hV$, then the following properties hold:
\begin{itemize}
\item
the distortion of~$f^m|\tW$ is bounded by~$K_1$;
\item
for each $y \in W$ we have~$|Df^m(y)| \ge \rho_1$;
\item
$\dist (\partial \tW, W)\ge 2(\kappa_0-1)^{-1} \diam (W)$.
\end{itemize}

Since~$f$ is uniformly expanding on~$K(V) \cap J(f)$, it follows
that there is a constant~$\rho_2 > 0$ such that for each~$x' \in J(f)$ and
$\delta' \in (0, \rho/2)$ such that~$B(x', \kappa_0 \delta')$
intersects~$K(V)$, there is an integer~$n \ge 0$ such
that~$f^n(B(x', \delta')) \supset B(f^n(x'), \rho_2)$, such
that~$\diam(f^n(x', \delta')) < \rho$, and such that the distortion
of~$f^n$ on~$B(x', \kappa_0\delta')$ is bounded by~$2$.
Replacing~$\rho_2$ by a smaller constant if necessary, we assume that
$(\kappa_0^2-1) \rho_2$ is less than the minimal
diameter of the components of~$V$.

We prove the assertion of the lemma with
$$ K_0 = 2K_1
\text{ and }
\rho_0 = \min \{ \rho_1, K_1^{-1} \rho_2 \}. $$
Let~$x \in J(f)$ and~$\delta > 0$ be given.
If~$B(x, \kappa_0\delta)$ intersects~$K(V)$, then assertion~$1$ holds by definition of~$\rho_2$.
So we assume that~$B(x, \kappa_0\delta)$ does not intersect~$K(V)$.
If
$$ B(x, \kappa_0\delta) \subset V
\text{ and }
B(x, \kappa_0\delta) \not\subset D, $$
then we put~$m = 0$, and we denote by~$W_0$ the connected component of~$V$ containing~$B(x, \kappa_0 \delta)$.
Otherwise there is an integer $m\ge 1$ such that $f^m(B(x,\kappa_0\delta))\subset V$ and such that~$f^m$ maps a neighborhood of~$B(x, \kappa_0 \delta)$ diffeomorphically onto a connected component of~$\hV$.
We assume that~$m$ is the largest integer with this property, and let~$W_0$ be
the pull-back of~$V$ by~$f^m$ that contains~$x$.
In both cases the distortion of~$f^m$ on~$B(x, \kappa_0 \delta)$ is bounded by~$K_1$, and
$$ f^m(B(x, \kappa_0 \delta)) \subset V
\text{ and }
f^m(B(x, \kappa_0 \delta)) \not \subset D. $$
If every connected component
of~$D$ intersecting~$f^m(B(x, \delta))$ is contained in~$f^m(B(x, \kappa_0 \delta))$, then assertion~$2$ holds with~$n = m$.
Otherwise there is a connected component~$W$ of~$D$ that intersects~$f^m(B(x, \delta))$ but is not contained in~$f^m(B(x, \kappa_0 \delta))$.
Let us prove that assertion~$1$ holds for some $n\ge m+m(W)$ in this case, where~$m(W)$ is the canonical inducing time of~$W$ with respect to~$(\hV, V)$.
Indeed, $W' \= \left(f^m|{W_0} \right)^{-1}(W)$ is a pull-back of~$V$ by~$f^{m + m(W)}$ such that~$f^{m+m(W)}$ maps a neighborhood of~$W'$ diffeomorphically onto a connected component of~$\hV$.
Since $W'\cap B(x,\delta)\not=\emptyset$, from the definition of~$\kappa_0$ it follows that~$B(x, \delta)\subset \tW'$, so the distortion of $f^{m+m(W)}|B(x, \delta)$ is bounded by~$K_1$.
On the other hand, by the choice of~$m$, we have $B(x, \kappa_0\delta)\not\subset W'$.
Suppose $\diam (W')\le (\kappa_0^2-1) \delta$, so that
$$ \frac{\diam(B(x, \delta))}{\diam(W')}
\ge
\frac{1}{\kappa_0^2 - 1}. $$
Letting~$n \= m + m(W)$, the set~$f^n(W')$ is a connected component of~$V$, so
$$ \diam(f^n(B(x, \delta)))
\ge
K_1^{-1} \frac{1}{\kappa_0^2 - 1} \diam(f^n(W'))
\ge
K_1^{-1} \rho_2. $$
This proves assertion~$1$ with~$n = m + m(W)$ in the case $\diam (W')\le (\kappa_0^2-1) \delta$.
Suppose now $\diam (W')> (\kappa_0^2-1) \delta$.
Then $B(x,\kappa_0\delta)\subset \tW'$, so $f^{m+m(W)}|B(x,\kappa_0\delta)$ has distortion bounded by $K_1$.
Moreover, the set $f^{m+m(W)}(B(x,\kappa_0\delta))$ intersects $\partial V\subset K(V)$, so assertion~$1$ holds by the choice of~$\rho_2$.
\end{proof}

In the case of complex rational maps, the following lemma is an easy
consequence of~\cite[Lemma~$6.3$]{PrzRiv07}.
The proof extends to the case of interval maps without change.
Recall that for a nice set~$V$ we denote by~$\fL_V$ the collection of components of $\dom(f) \setminus K(V)$, where~$K(V)$ is as in~\eqref{eqn:defKV}.
For each element~$U$ of~$\fL_V$, there exists a unique integer~$l(U)\ge 0$ \index{$l(\cdot)$} such that~$f^{l(U)}$ maps~$U$ diffeomorphically onto a component of~$V$.
\begin{lemm}\label{lem:finitelanding}
Let~$f \in \sA$ be expanding away from critical points. Then for
each nice set~$V$ for~$f$ there exist constants~$\alpha_0 \in (0, \HDhyp(f))$, $C_0$ and~$\varepsilon_0 > 0$ such that, for every integer~$m \ge 0$,
\begin{equation}\label{eqn:finitelanding}
\sum_{U \in \fL_V \, : \, l(U) \ge m} \diam (U)^{\alpha_0}
<
C_0 \exp(-\varepsilon_0 m).
\end{equation}
Moreover, if $\Crit'(f)\not=\emptyset$, then for each conformal measure~$\mu$ supported on~$J(f)$, there exist constants~$C' > 0$ and~$\kappa \in (0, 1)$
such that for each integer~$m \ge 0$,
$$ \mu\left(\{z\in J(f): z, f(z), \ldots, f^{m-1}(z)\not\in V \}\right)
\le C' \kappa^m.$$
\end{lemm}
\begin{proof}
Let~$V_0$ be a sufficiently small neighborhood of~$\Crit'(f)$
contained in~$V$, such that for each~$c \in \Crit'(f)$ the set $K \=
K(V_0) \cap J(f)$ intersects~$V^c$. Thus each element of~$\fL_V$
intersects~$K$.
Since by hypothesis~$f$ is uniformly expanding on~$K$, it follows that there are constants~$C_1 > 0$ and~$\varepsilon_1 > 0$ such that for each $U \in \fL_V$ we have~$\diam(U) \le C_1 \exp(-\varepsilon_1 l(U))$.

By~\cite[Lemma~$6.3$]{PrzRiv07} there is a constant~$\alpha_1 \in (0, \HDhyp(J(f)))$ such that,
$$ C_2
\=
\sum_{U \in \fL_V} \diam (U)^{\alpha_1}, $$
is finite.
Fix~$\alpha_0 \in (\alpha_1, \HDhyp(J(f)))$ and put~$\varepsilon_0 \= \varepsilon_1 (\alpha_0 - \alpha_1)$.
We thus have for every integer~$m \ge 0$,
\begin{align*}
\sum_{\substack{U \in \fL_V \\  l(U) \ge m}} \diam (U)^{\alpha_0}
& \le
\max_{\substack{U \in \fL_V\\ l(U)\ge m}} \diam(U)^{\alpha_0 - \alpha_1} \sum_{\substack{U \in \fL_V \\  l(U) \ge m}} \diam (U)^{\alpha_1}
\\ & \le
C_2 C_1^{\alpha_0 - \alpha_1} \exp(-\varepsilon_0 m).
\end{align*}

To prove the last assertion of the lemma, notice first that the exponent~$\alpha$ of~$\mu$ satisfies~$\alpha \ge \HDhyp(f)$~\cite{McM00}, so~$\alpha \ge \alpha_0$.
Thus for every~$m \ge 1$
$$ \sum_{U \in \fL_V \, : \, l(U) \ge m} \diam(U)^\alpha
\le
\max_{U\in\fL_V} \diam (U)^{\alpha - \alpha_0} \sum_{U \in \fL_V  \, : \, l(U) \ge m} \diam(U)^{\alpha_0} $$
is exponentially small in~$m$.
Moreover~$\mu(K(V)) = 0$ because~$f$ is uniformly expanding on~$K(V) \cap J(f)$.
On the other hand, by the Koebe principle
there is a constant~$C_3 > 0$ such that for each~$U \in \fL_V$ we
have $\mu(U) \le C_3 \diam(U)^\alpha$.
Thus, the last assertion of the lemma follows from the first from the inclusion,
$$ \{z\in J(f): z, f(z), \ldots, f^{m-1}(z)\not\in V \}
\subset
\left( K(V) \cup \bigcup_{U \in \fL_V \, : \, l(U) \ge m} W \right). $$
\end{proof}

\subsection{Backward contracting maps}
\label{ss:properties of bc maps}
The following lemma is simple consequence of~\cite[Proposition~$2$]{LiShe08a}.
\begin{lemm} \label{lem:blsquantify}
For each $\ell>1$ and $\kappa\in (0, \ell^{-1})$ there exists a constant~$r>1$
such that, if $f\in\sA^*(\ell)$ is backward contracting with
constant~$r$, then there is a constant~$C > 0$ such that for each
subset~$Q$ of~$\dom(f)$ intersecting $J(f)$ and each pull-back~$P$
of~$Q$,
$$ \diam(P)
\le
C \diam (f(Q))^\kappa. $$
\end{lemm}
\begin{proof} Fix $\ell$ and $\kappa$ and assume that $f\in \sA^*(\ell)$ is backward contracting
with a sufficiently large constant $r$. By (the proof
of)~\cite[Proposition~$2$]{LiShe08a}, there exists a constant~$C'>0$ such that
when $Q=\tB(c, \varepsilon)$ for some $c \in \Crit'(f)$
and~$\varepsilon > 0$, then each pull-back $P$ of $Q$ satisfies
$\diam (P)\le C' \eps^\kappa$.

For the general case, let us fix a small constant $\eps_0>0$.
Since~$f$ is uniformly expanding on~$J(f)\cap
K(\tB(\Crit'(f);\eps_0/2))$, we may assume without loss of
generality that~$Q$ is connected and contained in~$\tB(c; \eps_0)$,
for some $c\in\Crit'(f)$.
Let $\eps_1 \in (0, \eps_0]$ be minimal such that~$Q$ is contained in the closure of $\tB(c,
\eps_1)$, and write $Q'=\tB(c,\eps_1)$, $Q''=\tB(c, 2\eps_1)$.
Provided that~$\eps_0$ was chosen small enough,
\begin{equation}\label{eqn:Qfarawaycrt}
\frac{\diam (f(Q))}{\diam (Q)}\asymp \frac{\diam (f(Q'))}{\diam (Q')}\asymp \eps_1^{\ell_c-1}.
\end{equation}

Let $P\in \sM_n(Q)$ for some $n\ge 1$, and consider the chains
$\{P_j\}_{j=0}^n$, $\{P_j'\}_{j=0}^n$ and $\{P_j''\}_{j=0}^n$ with
$P=P_0\subset P_0'\subset P_0''$ and $P_n''=Q''$, $P_n'=Q'$,
$P_n=Q$.  If $f^n: P_0''\to Q''$ is a diffeomorphism, then
$f^n|P_0'$ has bounded distortion, so
$$\frac{\diam (P)}{\diam (P_0')}\asymp \frac{\diam (Q)}{\diam (Q')}\asymp \frac{\diam (f(Q))}{\diam (f(Q'))}\le \left(\frac{\diam (f(Q))}{\diam (f(Q'))}\right)^\kappa,$$
which implies that $\diam (P)\le C\diam (f(Q))^\kappa$ since $\diam (P_0')\le C' \eps_1^\kappa. $
Otherwise, let $m<n$ be maximal such that $P_m''$ contains a
critical point, say $c'$. By the backward contracting property, we have $\diam(P_{m+1}'')\le 2\eps_1 r^{-1}<\eps_0$.
Since~$f^{n-m-1}|P_{m+1}'$ has bounded distortion, using~\eqref{eqn:Qfarawaycrt} we obtain
$$\diam (P_{m+1})\asymp \diam (P_{m+1}')\frac{\diam (f(Q))}{\diam (f(Q'))}\le \diam (f(Q)) r^{-1},$$
which implies that $\diam (P_{m+1})< \diam (f(Q))$ provided that $r$ is large enough.
Since $P_m\subset \tB(c', \eps_0)$, we may repeat the above argument with $Q$ replaced by $P_m$. By an induction on $n$, we complete the proof of the lemma.
\end{proof}

Given a nice set $V$, a component of the set~$f^{-1}(\dom(f)
\setminus K(V))\cap V$
is called a \emph{return domain} \index{return domain}.
These are maximal pull-backs of~$V$ that are contained in~$V$.
Given $\lambda > 0$ we will say that~$V$ is
\emph{$\lambda$\nobreakdash-nice} \index{$\lambda$\nobreakdash-nice}
if for return domain~$W$ of~$V$
we have~$\overline{W} \subset V$ and,
$$ \mmod(V; W)\ge  \lambda. $$
The following is essentially a combination of~\cite[Proposition~$6.6$]{Riv07} and \cite[Lemma~$3$]{BruRivShevSt08}.
\begin{lemm}\label{l:very nice}
Given $\ell > 1$ and $\lambda>0$, there is a constant~$r > 4$ such that for every~$f \in \sA(\ell)$ that is backward contracting with constant~$r$, the
following property holds.
For every sufficiently small~$\delta > 0$, there is a symmetric nice couple $(\hV, V)$, such that~$V$ is $\lambda$-nice, and such that for each~$c \in \CJ$,
$$ \tB(c,r\delta/2)\subset \hV^c\subset
\tB(c,r\delta),\hspace{2mm}\tB(c, \delta) \subset V^c \subset \tB(c,
2\delta). $$
\end{lemm}
\begin{proof}
We assume that~$f$ is backward contracting with constant~$r \ge 2$.
Then there is a symmetric nice set~$\hV$ for~$f$ such that
$$\tB(c,r\delta/2)\subset \hV^c\subset \tB(c,r\delta),$$
see \cite[Proposition~$3$]{BruRivShevSt08} in the
case of interval maps, and \cite[Proposition~$6.5$]{Riv07} in the case
of rational maps.
For each $c\in \Crit'(f)$, let~$V_{c, *}$ be the union of $\tB(c,\delta)$ and all the return domains of~$\hV$ that intersect~$\tB(c,\delta)$.
By definition, $f(\partial V_{c, *}) \subset \partial f(V_{c, *})$, and by the backward contraction assumption, $V_{c, *}\subset \tB(c, 2\delta)$.
In the real case, let $V^c=V_{c,*}$ and in the complex case, let~$V^c$ be the filling of~$V_{c, *}$, \emph{i.e.}, the union of~$V_{c, *}$
and the components of $\ov{\C}\setminus V_{c,*}$ that are contained
in~$\hV^c$.
Then, in both cases, $V^c$ is simply connected,
$$ f(\partial V^c) \subset \partial f(V^c),
\text{ and }
V^c \subset \tB(c, 2\delta) \Subset \hV^c. $$
Let $V \= \bigcup_{c \in \Crit'(f)} V^c$.
Note that for each $x\in\partial V$ and $k\ge 1$,
$f^k(x)\not\in \hV$, hence~$(\hV, V)$ is a symmetric nice couple.
Provided~$r$ is large enough,
$$\mmod(\hV^c; V^c)\ge (2\ell_{\max}(f))^{-1}\log
(r/4)$$ is large. It follows that $V$ is a $\la$-nice set for a
large~$\la$. Indeed, if $U$ is a return domain of $V$ with return
time~$s$, then the pull-back of~$\hV$ by~$f^{s}$ that contains~$U$
is either diffeomorphic or unicritical, and it is contained in~$V$.
By either Lemma~\ref{lem:Cmodulus} or Lemma~\ref{lem:Rmodulus}, we
obtain that $\mmod(V; U)$ is large.
\end{proof}

\begin{defi}\label{def:child}
For a map~$f \in \sA$ and an integer~$m \ge 1$ we will say that a
pull-back~$W$ of an open set~$V$ by $f^m$ is a \emph{child of~$V$} \index{child}
if it contains precisely one critical point of~$f$, and if $f^{m-1}$
maps a neighborhood of $f(W)$ diffeomorphically onto a component of
of~$V$. We shall write $m_V(W)= m$. \index{$m_\cdot(\cdot)$}
\end{defi}
In the case of interval maps the following lemma is a variant of~\cite[Lemma~$4$]{BruRivShevSt08}.
\begin{lemm}\label{l:children}
For each~$s > 0$ and $\ell > 1$ there is a constant~$r > 4$ such that for every~$f \in \sA(\ell)$ that is backward contracting with constant~$r$, the following property holds.
For each sufficiently small $\delta > 0$ there is a nice set $V = \bigcup_{c \in \CJ} V^c$ such that for each~$c \in \CJ$,
$$ \tB(c, \delta) \subset V^c \subset \tB(c, 2\delta), $$
and such that
$$ \sum_{\text{children $Y$ of } V} \diam(f(Y))^s
\le
\delta^s. $$
\end{lemm}
\begin{proof}
Assume that~$f$ is backward
contracting with a large constant~$r$.
By Lemma~\ref{l:very nice}, for each sufficiently small $\delta > 0$ there exists a $\lambda$-nice set $V = \bigcup_{c \in \CJ} V^c$  such that for each~$c \in \CJ$,
$$ \tB(c, \delta) \subset V^c \subset \tB(c, 2\delta), $$
where $\lambda>0$ is a large constant.
Let us prove that the conclusion of the lemma holds for this choice of~$V$.

Take $c\in \CJ$ and let~$Y_k(c)$ be the $k$-th largest child of~$V$ containing~$c$.
By the backward contracting property, we have $Y_1(c)\subset \tB(c, 2r ^{-1}\delta)$.
Let $s_k=m_V(Y_k(c))$.
Then~$f^{s_k}(Y_{k+1}(c))$ is contained in a return domain of~$V$, hence $\mmod(V; f^{s_k}(Y_{k+1}(c)))\ge \lambda$.
By the definition of child, $Y_k(c)$ is a unicritical
pull-back of $V$. Thus by Lemma~\ref{lem:Cmodulus} or
Lemma~\ref{lem:Rmodulus}, we obtain that $\mmod(Y_k(c);
Y_{k+1}(c))\ge \lambda'$, where $\lambda'\to \infty$ as $\lambda\to
\infty$. By Lemma~\ref{lem:modsize}, it follows that
$\diam(Y_{k+1}(c))/\diam(Y_k(c))$ is small provided that $r$ is
large. The conclusion of the lemma follows.
\end{proof}

\section{Polynomial shrinking of components}
\label{s:polynomial shrinking}

In this section, we study the size of pull-backs of a small set. The
main result is the following proposition, from which we shall derive
Theorem~\ref{t:polynomial shrinking}.

\begin{prop}\label{prop:shrink}
For each $\ell>1$ and $\kappa_0 \in (0, \ell^{-1})$, there exists a constant~$R>3$ such that if $f\in\sA(\ell)$ is expanding away from critical
points and backward contracting with a constant $r>R$, then the
following holds.
For any $\eta_0>0$ sufficiently small there exist constants $C_0, A_0>0$ and for any chain $\{W_j\}_{j=0}^s$ with
$W_s\subset \tB(\Crit'(f),\eta_0)$, there exists an integer $\nu\ge
0$ such that
\begin{equation}\label{eqn:West2}
\diam (W_0) \le C_0 \min \left\{ r^{-\kappa_0 \nu}, \exp \left(- \kappa_0^{\nu} (A_0s + \mu) \right) \right\},
\end{equation}
where $\mu=\mmod(\tB(\Crit'(f),3\eta_0); W_s)$.
\end{prop}

The proof of this proposition in the real case is more complicated
than in the complex case. We shall state and prove a preparatory
lemma for the real case. The readers who are only interested in the
complex case may skip this part.

\begin{defi}
Consider $f\in\sA_\R$.
A sequence $\{U_j\}_{j=0}^s$ of open intervals is called a \emph{quasi-chain} \index{quasi-chain}
if for each $0\le j<s$,
the set~$U_j$ contains a component of~$f^{-1}(U_{j+1})$.
The \emph{order} of the
quasi-chain is the number of $j\in\{0,1,\ldots, s-1\}$ such that
$U_j$ contains a critical point.
\end{defi}

Given a chain $\{V_j\}_{j=0}^t$, we can construct a quasi-chain
$\{\hV_j\}_{j=0}^t$ with $\hV_j\supset V_j$ as follows.
First of all, $\hV_t=V_t$.
Once $\hV_j$ has been defined for some $1\le j\le t$, let~$V_{j-1}'$ be the component of $f^{-1}(\hV_j)$ that contains~$V_{j-1}$, and let
$$
\hV_{j-1}=
\begin{cases}
\tB(c; |\hV_{j}|) & \text{ if }V_{j-1}'\text{ contains a unique
critical point } c; \\
V_{j-1}' & \text{ otherwise.}
\end{cases}
$$
Note that $\hV_{j-1}$ contains a component of $f^{-1}(\hV_j)$ and in
the former case, $\hV_{j-1}$ is the component of
$f^{-1}((\hV_j)^\sharp)$ that contains~$c$, where $\hV_j^\sharp$ is
as in~\eqref{eqn:sharp}. We shall say that $\{\hV_j\}_{j=0}^t$ is
the \emph{preferred quasi-chain} for the chain $\{V_j\}_{j=0}^t$.
\index{preferred quasi-chain}

\begin{lemm}\label{lem:quasichain}
Consider a map $f\in\sA_{\R}(\ell)$ that is normalized near
critical points and fix $\kappa_0\in (0,\ell^{-1})$.
For each~$\eta_0>0$ small enough the following holds.
Let~$\{V_j\}_{j=0}^s$ and~$\{W_j\}_{j=0}^s$ be chains with $V_j \Supset
W_j$, for $j=0,1,\ldots, s$, and such that $V_s\subset \tB(\CJ,
3\eta_0)$, and let $\{\hV_j\}_{j=0}^s$ and $\{\hW_j\}_{j=0}^s$ be
the corresponding preferred quasi-chains. Assume that $V_j\cap
\Crit(f)=\emptyset$ for all $1\le j<s$. Then
$$\mmod(\hV_0; \hW_0)
\ge \kappa_0 \mmod (V_s; W_s).$$
\end{lemm}
\begin{proof}
Let $\lambda = \ell \kappa_0\in (0,1)$ and let $\eta>0$ be the
constant given by part~$1$ of Lemma~\ref{lem:Rmodulus}. Assuming
that $\eta_0$ is sufficiently small, we have $|V_s|, |V_1|<\eta$
since $f$ has no wandering interval. By construction, $\hV_j=V_j$,
$\hW_j=W_j$ for all $j=1,2,\ldots, s$. If $V_0\cap
\Crit(f)=\emptyset$, then $\hV_0=V_0$, $\hW_0=W_0$, and $f^s: V_0\to
V_s$ is a diffeomorphism, so the desired inequality follows from
part~$1$ of Lemma~\ref{lem:Rmodulus}.
In the case $V_0\cap\Crit(f) \neq \emptyset$,
\[
\mmod(V_1; W_1)\ge \lambda \mmod(V_s; W_s).
\]
Moreover, by part~$2$ of Lemma~\ref{lem:Rmodulus},
\begin{equation}\label{eqn:stepi}
\mmod(\hV_0; \hW_0)\ge \ell_c^{-1} \mmod (V_1; W_1).
\end{equation}
Combining these two inequalities above gives us the desired
estimate.
\end{proof}

\begin{proof}[Proof of Proposition~\ref{prop:shrink}]
Fix~$\kappa_0 \in (0, \ell^{-1})$ and assume~$f$ is backward contracting with a sufficiently large constant~$r$ that the conclusion of Lemma~\ref{lem:blsquantify} holds with $\kappa=\kappa_0$.
Let $\eta_0>0$ be a small constant such that for all $\delta\in (0,\eta_0)$, each pull-back $W$ of $\tB(\CJ, r\delta)$ with $\dist(W, f(\CJ))<\delta$ satisfies $\diam(W)<\delta$.
Moreover, when considering an interval map, we assume that~$f$ is normalized near critical points (after a $C^3$ conjugacy) and reduce~$\eta_0$ if necessary so that Lemma~\ref{lem:quasichain} holds.

Let $t_1<t_2<\ldots<t_k=s$ be all the positive integers such that $f^{t_i}(W_0)\cap \tB(\CJ,\eta_0)\not=\emptyset$.
By the backward contraction property, $f^{t_i}(W_0) \subset \tB(c_i, 3 \eta_0)$ for~$i = 1, \ldots, k$.
For each $i$, let~$c_i$ be the critical point in~$\Crit'(f)$ closest to~$f^{t_i}(W)$ and let $\{Y_i^j\}_{j=0}^{t_i}$ be the chain with
$Y_i^{t_i}=\tB(c_i, 3\eta_0)$ and $Y_i^0\supset W_0$ and write
$Y_i=Y_i^0$. For $1\le i\le k$ and $0\le j\le t_i$, let
$\hY_i^j=Y_i^j$ in the complex case; and let
$\{\hY_i^j\}_{j=0}^{t_i}$ be the preferred quasi-chain for
$\{Y_i^j\}_{j=0}^{t_i}$ in the real case. Moreover, let $\hW_j=W_j$,
$j=0,1,\ldots, s$ in the complex case;  and let $\{\hW_j\}_{j=0}^s$
be the preferred quasi-chain for $\{W_j\}_{j=0}^s$ in the real case.

Let $\nu_i$ be the order of the (quasi-)chain
$\{\hY_i^j\}_{j=0}^{t_i}$ and let $\nu=\max_{i=1}^k \nu_i$.
We first prove there exists a constant~$C_1>0$ such that
\begin{equation}\label{eqn:CshrinkW0nd}
\diam (W_0)\le\diam(\hW_0) \le  C_1 r^{-\kappa_0 \nu}.
\end{equation}
Indeed, take $i_0\in \{1,2,\ldots, k\}$ such that $\nu=\nu_{i_0}$ and let $0\le j_0< j_1<\cdots< j_{\nu-1}< j_{\nu}= t_{i_0}$ be the integers such that $\hY_{i_0}^{j_m}$ intersects $\CJ$, $m=0,1, \ldots, \nu$.
Then by the backward contracting property (and
the construction of the quasi-chains in the real case), we prove
inductively that $\hY_{i_0}^{j_{\nu-m}}\subset \tB(\CJ, 3 r^{-m}
\eta_0)$. In particular, $\hY_{i_0}^{j_0}\subset \tB(\CJ, 3r^{-\nu}
\eta_0)$. By Lemma~\ref{lem:blsquantify}, we obtain~\eqref{eqn:CshrinkW0nd}.

Next, let us prove there exist constants $C_2, A_0 > 0$ such that
\begin{equation}\label{eqn:CshrinkW0st}
\diam (W_0)\le C_2 \exp\{-\kappa_0^{\nu}(A_0s + \mu)\}.
\end{equation}

To this end, let $\hD_i\=\hY_{i+1}^{t_i}$ for $i=1,2,\ldots, k-1$,
let $\hD_k=W_s$, and put
\[
\mu_i=\mmod(\tB(\CJ, 3\eta_0); \hD_i), \text{ for }i=1,2,\ldots, k.
\]
So $\mu_k=\mu$.

For $i \in \{1, \ldots, k - 1\}$ and $t_i < j < t_{i+1}$, the set~$Y_{i+1}^j$ is
disjoint from $\tB(\CJ,\eta_0)$ so we have $\hY_{i+1}^j=Y_{i+1}^j$.
Hence by the backward contraction property,
$$ \hD_i\subset \tB(\CJ, 2\eta_0) \Subset \tB(\CJ, 3\eta_0). $$
Moreover, since~$f$ is uniformly expanding outside $\tB(\CJ,\eta_0)$, $\diam (\hD_i)$ is exponentially small in terms of $t_{i+1}-t_i$. For a similar reason,
$\diam (\hY_1)$ is exponentially small in terms of $t_1$. Thus there
are constants $A_0>0$ and $C_3>0$ such that
\begin{equation}\label{eqn:Cshrinkmod}
\mu_i\ge A_0(t_{i+1}-t_i), i=1,2,\ldots, k-1
\end{equation}
and
\begin{equation}\label{eqn:CshrinkY1}
\diam (\hY_1)\le C_3 \exp(-A_0t_1).
\end{equation}

In the complex case, by Lemma~\ref{lem:Cmodulus}  we obtain that
$$\mmod (\hY_i; \hY_{i+1})
\ge
\kappa_0^{\nu_i} \mmod (\tB(\CJ, 3\eta_0); \hD_i)\ge \kappa_0^{\nu} \mu_i,$$
where $\hY_{k+1}\=\hW_0$. This equality also holds in the real case
by repeatedly applying Lemma~\ref{lem:quasichain}. Indeed, if
$$ 0 \le j_0< j_1<\cdots< j_{\nu_i-1}< j_{\nu_i}= t_i $$
are the integers such that~$\hY_{i}^{j_m}$ intersects~$\CJ$, then for any $m=1,2,\ldots,
\nu_i$, $\hY_i^{j_m}\subset \tB(c_0,3\eta_0)$ so that
$$\mmod(\hY_i^{j_{m}}; \hY_{i+1}^{j_{m}})\ge \kappa_0
\mmod(\hY_i^{j_{m-1}}; \hY_{i+1}^{j_{m-1}}).$$
Thus, in both cases we have by Lemma~\ref{lem:grotzch} that
\[
\mmod (\hY_1; \hY_{k+1})\ge \sum_{i=1}^{k} \mmod (\hY_i; \hY_{i+1})
\ge \kappa_0^{\nu} \sum_{i=1}^{k} \mu_i.
\]
By~\eqref{eqn:Cshrinkmod}, this implies
\[
\mmod (\hY_1; \hW_0)=\mmod (\hY_1; \hY_{k+1})\ge \kappa_0^{\nu}
(A_0(s-t_1) +\mu).
\]
Using~\eqref{eqn:CshrinkY1} and applying Lemma~\ref{lem:modsize}, we
obtain~\eqref{eqn:CshrinkW0st}.

Combining the inequalities~\eqref{eqn:CshrinkW0st} and~\eqref{eqn:CshrinkW0nd}, we obtain the inequality~\eqref{eqn:West2}
with $C_0=\max(C_1, C_2)$.
\end{proof}

\begin{proof}[Proof of Theorem~\ref{t:polynomial shrinking}]
Fix a small constant $\eta_0>0$. By assumption, $f$ is uniformly expanding on the maximal invariant set~$K$ of~$f$ in~$J(f) \cap \tB(\Crit'(f), \eta_0/2)$.
It follows that there are constants~$\rho>0$ and~$A_1 > 0$ so that the following property holds: for every $y \in J(f)$, every
integer $t \ge 1$ and every chain $\{ W_j' \}_{j = 0}^t$ satisfying
$$ W_t' = B(y, \rho),
W_0' \cap \tB(\Crit'(f), \eta_0/2) \neq \emptyset, $$
and such that for every~$j \in \{ 1, \ldots, t - 1 \}$
the set~$W_j'$ is disjoint from~$\tB(\Crit'(f), \eta_0/2)$, we have
$$ W_0' \subset \tB(\Crit'(f), \eta_0),
\text{ and } \mmod(\tB(\Crit'(f), 3 \eta_0); W_0') \ge A_1 t. $$

Let $x\in J(f)$, $m\ge 1$ and let~$W$ be the component of
$f^{-m}(B(f^m(x),\rho))$ containing~$x$. We shall prove that
\begin{equation}\label{eqn:W}
\diam(W)
\le
C \min \left\{ r^{- \kappa_0 \nu}, \exp \left( - \kappa_0^{\nu} A m \right)\right\},
\end{equation}
holds for some integer $\nu\ge 0$, where $C, A > 0$ are constants.

Consider the chain $\{W_j\}_{j=0}^m$ with $W_m=B(f^m(x), \rho)$ and $W_0\ni x$.
If for every~$s \in \{ 0, \ldots, m \}$ the set~$W_s$ is disjoint from~$\tB(\Crit'(f), \eta_0/2)$, then the desired inequality follows with~$\nu = 0$ from the assumption that~$f$ is uniformly expanding on~$K$.
So we suppose that there is an integer $s \in \{0, \ldots, m \}$ such that~$W_s$ intersects $\tB(\Crit'(f), \eta_0/2)$, and assume that~$s$ is maximal with this property.
By our choice of~$\rho$ we have~$W_s \subset \tB(\Crit'(f), \eta_0)$, and
\begin{equation}\label{eqn:mu}
 \mu \= \mmod(\tB(\Crit'(f), 3\eta_0); W_s) \ge A_1(m - s).
\end{equation}
Applying Proposition~\ref{prop:shrink} to the chain
$\{W_j\}_{j=0}^s$, we obtain a non-negative integer~$\nu$ such that
\begin{equation*}
\diam (W)
\le
C_0 \min \left\{ r^{-\kappa_0 \nu}, \exp \left( - \kappa_0^{\nu} (A_0 s + \mu) \right) \right\},
\end{equation*}
which together with~\eqref{eqn:mu} implies that~\eqref{eqn:W} holds
with $A=\min (A_0, A_1)$.

To conclude the proof, let~$\beta > 0$ satisfy~$\beta < \log r/ (\kappa_0^{-1}\log \kappa_0^{-1})$, so that
$$ \varepsilon \= 1 - \beta (\kappa_0^{-1} \log \kappa_0^{-1})/ \log r > 0. $$
If~$\nu$ is such that~$r^{-\kappa_0 \nu} \ge m^{-\beta}$, then
$$ \nu \le \beta \log m / (\kappa_0 \log r)
\text{ and }
\exp(-\kappa_0^{\nu} A m) \le \exp(- m^{\varepsilon} A). $$
Thus, \eqref{eqn:W} implies that~$f$ satisfies the Polynomial Shrinking Condition with exponent~$\beta$.
\end{proof}

\section{Bounding the badness exponent}\label{s:finiteness of bad}
This section is devoted to the proof of Theorem~\ref{t:finiteness of
bad}. We shall prove a recursive formula for the size of \emph{relatively bad} pull-backs.

\subsection{Relatively bad pull-backs}\label{ss:relatively bad}
Let~$f \in \sA$ and let~$V_0$ be a nice set for~$f$. Recall that
given an integer~$m \ge 1$ we say that $W\in\sM_m(V_0)$ is a bad
pull-back~of~$V_0$ by~$f^m$, if for every integer~$m' \in \{ 1,
\ldots, m \}$ such that~$f^{m'}(W) \subset~V_0$ the pull-back
of~$V_0$ by~$f^{m'}$ containing~$W$ is not diffeomorphic.

For a subset~$V$ of~$V_0$ and an integer~$m \ge 1$, we say that a
pull-back~$W$ of~$V$ by~$f^m$ is \emph{bad relative to~$V_0$}, \index{relatively bad pull-back}
if the pull-back of~$V_0$ by~$f^m$ containing~$W$ is bad.
Denote by~$\fBr_0(V) \=\sM_0(V)$ the collection of components of~$V$; and for $m\ge 1$, denote by~$\fBr_m(V)$ the collection of all pull-backs of~$V$ by~$f^m$ that are bad relative to~$V_0$.
Moreover, denote by~$\fBr_{m,o}(V)$ the collection of
all elements~$W$ of~$\fBr_m(V)$ for which~$f^m$ maps~$W$
diffeomorphically onto a component of~$V$. Clearly, for any integer
$m\ge 1$ the following properties hold:
\begin{itemize}
\item $W\in \fBr_m(V_0)$ if and only if $W$ is a bad pull-back of $V_0$ by $f^m$;
\item for $V\subset V_0$ and $W\in\sM_m(V)$, $W\in \fBr_m(V)$
if and only if for any $m'\in\{1,2,\ldots, m\}$, $W$ is not
contained in any diffeomorphic pull-back of $V_0$ by $f^{m'}$;
\item if $\tV\subset V\subset V_0$ and $\tW\in \fBr_m(\tV)$, then the component of~$f^{-m}(V)$ containing~$\tW$ belongs to~$\fBr_m(V)$.
\end{itemize}

\begin{lemm}\label{l:child decomposition}
For every~$V\subset V_0$ and every $m\ge 1$,
$$ \fBr_m(V)
=
\fBr_{m,o}(V) \cup \left( \bigcup_{\text{children~$Y$ of } V} \fBr_{m-m(Y)}(Y) \right), $$
where $m(Y)=m_V(Y)$ is as in Definition~\ref{def:child}.
\end{lemm}
\begin{proof}
For each $W \in \fBr_m(V) \setminus \fBr_{m,o}(V)$, there is~$m' \in
\{1, \ldots, m \}$ such that the connected component~$Y$
of~$f^{-m'}(V)$ containing~$f^{m - m'}(W)$ contains a critical point
of~$f$. If~$m'$ is the minimal integer with this property, then~$Y$
is a child of~$V$ and $m'=m(Y)$. If $m'=m$ then $W=Y$; otherwise, we
have $W \in \fBr_{m-m'}(Y)$ since $Y\subset V_0$.
This proves that the set on the left-hand side is contained in the set on the right-hand side.

To prove  the other direction, we first note that by
definition~$\fBr_{m,o}(V) \subset \fBr_m(V)$. It remains to show
that for a child~$Y$ of~$V$ we have~$\fBr_{m-m(Y)}(Y) \subset
\fBr_m(V)$.
Indeed, since~$Y$ contains a critical point of~$f$ we have~$Y \in \fBr_{m(Y)}(V)$, so the conclusion holds if $m(Y)=m$.
Now assume that $m_0\=m-m(Y)>0$ and consider~$W\in\fBr_{m_0}(Y)$. To
prove $W\in \fBr_m(V)$ let $W_0$ be the pull-back of $V_0$ by $f^m$
that contains $W$. Let~$m' \in \{ 1, \ldots, m \}$ be such
that~$f^{m'}(W_0) \subset~V_0$. If~$m' \le m_0$ then the
pull-back~$V_0$ by~$f^{m'}$ containing~$W_0$ is not univalent
because~$W \in \fBr_{m_0}(Y)$ and $Y\subset V_0$. If~$m' \ge m_0 +
1$, then the pull-back of~$V_0$ by~$f^{m' - m_0}$
containing~$f^{m_0}(W)$ is not diffeomorphic because it
contains~$Y$. This shows that~$W \in \fBr_m(V)$.
Thus~$\fBr_{m-m(Y)}(Y) \subset \fBr_m(V)$.
\end{proof}

\subsection{Proof of Theorem~\ref{t:finiteness of bad}}
Fix~$f \in \sA$, put~$\tau \= \taubase^{-\ell_{\max}(f)}$ and
fix~$\delta_0 > 0$ sufficiently small so that for every~$\delta \in
(0, \delta_0]$, every integer~$i \ge 1$, if $W$ is a
pull-back~of~$\tB(\CJ, 2\tau^i\delta)$ by $f^m$ for some $m\ge 1$
and the pull-back~$\tW$ of~$\tB(\CJ, \delta)$ by~$f^{m}$
containing~$W$ is diffeomorphic, then we have $\diam (\tW) \ge
\koebefactor^{-1} \cdot \taubase^i\diam (W)$, where $\koebefactor >
1$ is a universal (Koebe) constant.

In the rest of this section we fix~$t > 0$, and put $\tilde{\ell} =
\ell_{\max}(f)$ in the complex case and $\tilde{\ell}=2$ in the real
case. We assume that~$f$ is backward contracting with a large
constant~$r$ so that the conclusion of Lemma~\ref{l:children} holds
with~\begin{equation}\label{eqn:constants} s = t/(4\ell_{\max}(f))
\end{equation}
 and so that
\begin{equation}\label{eqn:smalleps}
(2r^{-1})^{s}
\le
\varepsilon \= \tilde{\ell}^{-1} \koebefactor^{- t} \taubase^{- t} \left(1 - \taubase^{-t/2} \right).
\end{equation}
So, reducing~$\delta_0 > 0$ if necessary, for each integer~$n \ge 0$ there exists a nice set~$V_n = \bigcup_{c \in \CJ} V_n^c$ such that for each~$c \in \CJ$,
$$ \tB(c, \tau^n\delta_0)
\subset
V_n^c
\subset
\tB(c, 2\tau^n \delta_0), $$
and
$$ \sum_{\text{children~$Y$ of } V_n} \diam(f(Y))^{s}
\le
(\tau^n \delta_0)^s. $$
Note that for each integer~$n \ge 0$, and each child~$Y$
of~$V_n$, we have~$\diam(f(Y)) \le (2r^{-1}) \tau^n \delta_0$, hence
\[\sum_{\text{children~$Y$ of } V_n} \diam(f(Y))^{2s}
\le (2r^{-1}\tau^n\delta_0)^{s}\sum_{\text{children~$Y$ of } V_n}
\diam(f(Y))^{s}.\] It follows that for each integer~$n \ge 0$,
\begin{equation}
  \label{e:children contribution}
\sum_{\text{children~$Y$ of } V_n} \diam(f(Y))^{2s}
\le \varepsilon (\tau^n \delta_0)^{2s}.
\end{equation}

Given an integer~$m \ge 0$ and a subset~$V$ of~$V_0$ we put
$$ \Xi_t(V, m) \= \sum_{j=0}^{m-1}\sum_{W \in \fBr_j(V)} d_V(W) \diam(W)^t,$$
and
$$\Xi^*_t(V, m)\= \sum_{j=0}^{m-1} \sum_{W\in\fBr_j(V)\setminus \fBr_{j,o}(V)} d_V(W)\diam (W)^t.
$$
Note that by definition~$\Xi_t(V, 0) = 0$ and that for~$\tV\subset
V\subset V_0$ such that every connected component of~$V$ contains at
most one connected component of~$\tV$, we have that for each $j\ge
0$ and $W\in\sM_j(V)$,
\begin{equation}\label{eqn:degreecomp}
\sum_{\tW \in \sM_j(\tV) \, : \, \tW\subset W} d_{\tV}(\tW)\le d_{V}(W).
\end{equation}
In particular, we have that for each $m\ge 0$,
$$ \Xi_t(\tV, m) \le \Xi_t(V, m). $$

\begin{lemm}
Under the above circumstances, for each integer~$m \ge 1$,
\begin{equation}\label{e:zero level}
\Xi_t(V_0, m)
\le
\sum_{c\in\Crit'(f)}\diam(V_0^c)^t + \tilde{\ell} \sum_{\text{children~$Y$ of } V_0} \Xi_t(Y, m - 1)
\end{equation}
and, for each~$n \ge 1$,
\begin{equation}\label{e:higher levels}
\Xi_t(V_n, m)
\le
\koebefactor^{t} \tilde{\ell} \sum_{i = 0}^{n - 1} \taubase^{(i + 1 - n) t} \sum_{\text{children~$Y$ of } V_{i}} \Xi_t(Y, m - 1).
\end{equation}
\end{lemm}
\begin{proof}
\

\partn{1}
To prove the first assertion observe that by definition no univalent
pull-back of $V_0$ is bad relative to~$V_0$,
so~$\fBr_{j,o}(V_0)=\emptyset$  for all $j\ge 1$.
Thus, by Lemma~\ref{l:child decomposition},
\begin{equation*}
\Xi_t(V_0, m)
\le
\sum_{W\in\fBr_0(V_0)}  \diam (W)^t +\sum_{\text{children~$Y$ of } V_0} d(Y) \Xi_t(Y, m - m_{V_0}(Y)),
\end{equation*}
where $m_{V_0}(Y)\ge 1$ is as in Definition~\ref{def:child}.
Since~$\fBr_0(V_0)$ is the collection of components of~$V_0$, the desired inequality follows.

\partn{2}
To prove the second assertion fix~$n \ge 1$. For $W \in \fBr_j(V_n)$ and~$i \in \{ 0, \ldots, n - 1 \}$, denote by~$W^i$ the component of~$f^{-j}(V_{i})$ containing~$W$.
Thus $W^{i} \in \fBr_j(V_{i})$, and by definition~$W^0 \not \in \fBr_{j,o}(V_0)$. We denote by~$i(W)$ the largest integer~$i \in \{ 0, \ldots, n - 1 \}$ such that~$W^{i} \not \in \fBr_{j,o}(V_{i})$.

\partn{2.1}
Let us prove that for each $i\in\{0,1,\ldots, n-1\}$,
\begin{equation}\label{e:diffeomorphic estimate}
\sum_{j=0}^{m-1}\sum_{W \in \fBr_j(V_n)  \, : \, i(W) = i} d_{V_n}(W) \diam(W)^t
\\ \le
\koebefactor^{t} \cdot \taubase^{(i + 1 - n) t} \Xi^*_t(V_i, m).
\end{equation}
To this end, we first prove that for each~$W \in \fBr_j(V_n)$ with~$i(W) = i$,
\begin{equation}\label{eqn:diamwwi}
\diam(W)
\le
\koebefactor \cdot \taubase^{i + 1 - n} \diam(W^{i}).
\end{equation}
Indeed, this is trivial if $i=n-1$.
If $i<n-1$, then $W^{i+1}$ is a diffeomorphic pull-back of $V_{i+1}$, so, by the definition of~$\koebefactor$ and the inclusion~$W^{i + 1} \subset W^i$,
$$ \diam(W)
\le
\koebefactor \cdot \taubase^{i + 1 - n} \diam(W^{i + 1})
\le
\koebefactor \cdot \taubase^{i + 1 - n} \diam(W^{i}). $$
Together with~\eqref{eqn:degreecomp}, the inequality~\eqref{eqn:diamwwi} implies that for each $j\ge 0$ and each $W'\in \fBr_j(V_i)$,
$$ \sum_{\substack{W \in\fBr_j(V_n) \, : \\ i(W) = i, W^i= W'}} d_{V_n}(W) \diam(W)^t
\le
d_{V_i}(W')\diam(W')^t \koebefactor^{t} \cdot \taubase^{(i + 1 - n) t}. $$
Summing over all $W'\subset  \fBr_j(V_n)\setminus\fBr_{j,o}(V_n)$, we obtain
 the inequality~\eqref{e:diffeomorphic estimate}.

\partn{2.2}
In view of Lemma~\ref{l:child decomposition}, for each~$i \in \{0, \ldots, n - 1 \}$,
\begin{align*}
\Xi^*_t(V_i, m)
& \le
\sum_{\text{children~$Y$ of } V_i} d_{V_i}(Y) \Xi_t(Y, m - m(Y))
\\ & \le
\tilde{\ell} \sum_{\text{children~$Y$ of } V_i} \Xi_t(Y, m - 1).
\end{align*}
Together with~\eqref{e:diffeomorphic estimate}, this implies,
$$
\Xi_t(V_n, m)
\le
\koebefactor^{t} \tilde{\ell} \sum_{i = 0}^{n - 1} \taubase^{(i + 1 - n) t} \sum_{\text{children~$Y$ of } V_{i}} \Xi_t(Y, m - 1).
$$
\end{proof}

\begin{proof}[Proof of Theorem~\ref{t:finiteness of bad}]
Let~$C > 0$ be a sufficiently large constant that
\begin{equation}\label{e:for zero level}
\sum_{c\in\Crit'(f)}\diam(V_0^c)^t
\le
C \delta_0^{2s} \left( 1 - \varepsilon \tilde{\ell} \tau^{-2s} \right).
\end{equation}

With the notation introduced above we need to show that
$$ \lim_{m \to \infty} \Xi_t(V_0, m) < \infty. $$
We will prove by induction in~$m \ge 0$ that for every~$n \ge 0$,
\begin{equation}
  \label{e:controlling bad}
\Xi_t(V_n, m)
\le C (\tau^n \delta_0)^{2s},
\end{equation}
which clearly implies the desired assertion.

Since for each integer~$n \ge 0$ we have~$\Xi_t(V_n, 0) = 0$, when~$m = 0$ inequality~\eqref{e:controlling bad} holds trivially for every~$n \ge 0$.
Let~$m \ge 1$ be given and assume by induction that inequality~\eqref{e:controlling bad} holds for every~$n \ge 0$, replacing~$m$ by~$m - 1$.

Fix~$i \ge 0$. Let us prove that
\begin{equation}\label{e:level i}
\sum_{\text{children~$Y$ of } V_i} \Xi_t(Y, m - 1)
\le
C \varepsilon (\tau^{i-1} \delta_0)^{2s}.
\end{equation}
Indeed, for a child~$Y$ of~$V_i$, letting~$k$ be the largest integer such that~$Y \subset V_k$, we have
$\diam (f(Y))\ge \tau^{k+1}\delta_0$.
By the induction hypothesis, it follows that,
\begin{equation*}
\Xi_t(Y, m - 1)
\le
\Xi_t(V_k, m - 1)
\le
C (\tau^k \delta_0)^{2s}
\le
C ( \tau^{-1}\diam(f(Y)))^{2s},
\end{equation*}
thus
\begin{equation*}
\sum_{\text{children~$Y$ of } V_i} \Xi_t(Y, m - 1)
\le
C \sum_{\text{children~$Y$ of } V_i} (\tau^{-1} \diam(f(Y)))^{2s},
\end{equation*}
which implies~\eqref{e:level i} by~\eqref{e:children contribution}.

Taking~$i = 0$ in~\eqref{e:level i}, we obtain by~\eqref{e:for zero level} and~\eqref{e:zero level},
$$ \Xi_t(V_0, m)
\le
C \delta_0^{2s}, $$
which proves~\eqref{e:controlling bad} for~$n = 0$.
By~\eqref{e:level i} and~\eqref{e:higher levels}, for a given integer~$n \ge 1$ we obtain,
\begin{align*}
\Xi_t(V_n, m)
& \le
\koebefactor^t \tilde{\ell} \sum_{i=0}^{n-1} \taubase^{(i+1-n)t} C\eps(\tau^{(i - 1)} \delta_0)^{2s}
\\ & \le
C(\tau^n \delta_0)^{2s} \eps \koebefactor^t \tilde{\ell}\tau^{-4s}\sum_{i=0}^{n-1} (\taubase ^t \tau^{2s})^{i+1-n}
\\ & \le
C(\tau^n \delta_0)^{2s} \eps \koebefactor^t \tilde{\ell}\taubase^{-t} (1-\taubase^{-t/2})^{-1},
\end{align*}
where we used $\tau=\taubase^{-\ell_{\max}(f)}$ and $s=t/(4\ell_{\max}(f)).$
Inequality~\eqref{e:controlling bad} follows by applying~\eqref{eqn:smalleps}, thus completing the proof of the theorem.
\end{proof}

\section{Induced Markov maps}
\label{s:tail estimate} This section is devoted to
the proof of Theorem~\ref{t:tail estimate} and
Corollary~\ref{coro:inducing}. We shall first prove
in~\S\ref{ss:dimensionestimate} the desired dimension estimate
applying arguments in~\cite{PrzRiv07}, see
Proposition~\ref{prop:dimensionestimate}. Then we proceed to the
tail estimate, where a Whitney decomposition type argument
originated in~\cite{PrzRiv11} plays an important role. We first
reduce the proof of Theorem~\ref{t:tail estimate} to
Proposition~\ref{prop:whitney} in~\S\ref{ss:tail estimate}, and then
give the proof of this proposition
in~\S\ref{ss:proof of whitney}. We also deduce
Corollary~\ref{coro:inducing} from Theorem~\ref{t:tail estimate} in
\S\ref{ss:proof of inducing}.

We fix throughout this section a map~$f$ in the class~$\sA^*$ defined in~\S\ref{sss:inducing}.
Moreover, we denote by~$\dom(f)$ the domain of~$f$.
In this section we use the hyperbolic dimension~$\HDhyp(f)$ and the conical Julia set~$\Jcon(f)$ defined in~\S\ref{ss:backward contracting maps}, as well as the badness exponent~$\delta_{\bad}(f)$ defined in Definition~\ref{def:badpb}.
\subsection{Dimension estimate}\label{ss:dimensionestimate}
Let us first prove the following:
\begin{prop}\label{prop:dimensionestimate}
Assume that $f\in \sA^*$ satisfies $\delta_{\bad}(f)< \HD(J(f))$.
Then,
$$ \HD(J(f)) = \HD(\Jcon(f)) = \HDhyp(f), $$
and for each sufficiently small nice couple $(\hV, V)$ and each $c \in \Crit'(f)$,
$$\HD(J(F)\cap V^c)=\HD(J(f)),$$
where $F$ denotes the canonical induced map associated with $(\hV, V)$, defined in~\S\ref{sss:inducing}.
\end{prop}

For an open neighborhood $V$ of $\Crit'(f)$, let $K(V)$ be as in
\S\ref{ss:notation}. It follows from the definition of~$\sA^*$ that
$K(V)\cap J(f)$ is a hyperbolic set for~$f$. For a nice set~$V$ and
$m\ge 1$, we use~$\fB_m(V)$ to denote the collection of all bad pull-backs of~$V$ by $f^m$, see Definition~\ref{def:badpb}.

We need the following lemma.
\begin{lemm}\label{lemm:dimensionD}
For each $f \in \sA^*$,
$$ \HD(J(f)\setminus \Jcon (f))
\le
\delta_{\bad}(f). $$
Furthermore, for each nice couple $(\hV, V)$ of~$f$ such that $\delta_{\bad}(\hV)<\HD(J(f))$,
$$\HD((J(f)\cap V)\setminus J(F))
<
\HD(J(f)),$$
where~$F$ denotes the canonical induced map associated with $(\hV, V)$.
\end{lemm}
\begin{proof}
To prove the first inequality it is enough to prove that for each nice couple~$(\hV, V)$ for~$f$,
\begin{equation}\label{eqn:jf-jcon}
\HD(J(f)\setminus \Jcon (f))\le \delta_{\bad}(\hV).
\end{equation}

To this end, let $N$ be the subset of $(V \cap J(f)) \setminus J(F)$
of those points that return at most finitely many times to~$V$ under
forward iteration, and let
$$ I \= ((V \cap J(f)) \setminus J(F)) \setminus N. $$
To estimate~$\HD(I)$, let~$\tI$ be the subset of~$I$ of those points~$x$
such that for every integer~$m \ge 1$ with~$f^m(x) \in V$,
the pull-back of~$\hV$ by~$f^m$ containing~$x$ is not diffeomorphic.
This implies that every pull-back of~$\hV$ containing~$x$ is bad.
Therefore, for every integer~$m \ge 1$,
$$ \tI
\subset
\bigcup_{j=m}^\infty\bigcup_{\tW \in \fB_j(\hV)} \tW. $$
The definition of badness exponent implies that~$\HD(\tI) \le
\delta_{\bad}(\hV)$.
Noting that for every~$y$ in~$I$ there is an integer~$n \ge 0$ such
that~$f^n(y)$ is in~$\tI$, we conclude that~$\HD(I) \le \HD(\tI) \le
\delta_{\bad}(\hV)$.

Let us prove that $N\setminus \Jcon (f)$ is at most countable. Indeed, $K(V)\cap J(f)$ is a hyperbolic set, hence
$K(V)\cap J(f)\subset \Jcon (f)$.
Since
\begin{equation}\label{eqn:setN}
N \subset \bigcup_{n=1}^\infty f^{-n}(K(V)\cap J(f))
\end{equation}
and since
$f^{-1}(\Jcon
(f))\subset \Jcon (f)\cup \Crit'(f)$, we conclude that
$N\setminus \Jcon(f)\subset \bigcup_{n=0}^\infty f^{-n}(\Crit'(f))$ is at most countable.

Since $J(F)\subset \Jcon (f)$
and $K(V)\cap J(f)\subset \Jcon (f)$, it follows that
$$\HD(J(f)\setminus \Jcon(f))=\HD((J(f)\setminus \Jcon(f))\cap V)
\le
\HD(I)\le \delta_{\bad}(\hV).$$
This proves~\eqref{eqn:jf-jcon}, hence the first equality of the lemma.

To prove the last inequality we need the following result: for any
hyperbolic set $A$ of $f$, $\HD(A)< \HD_{\hyp}(f)$. This is proved
in \cite[Lemmas~$6.2$]{PrzRiv07} in the complex case, as a consequence
of the (essentially) topologically exact
property of the Julia set.
The proof works without change for maps $f\in \sA^*_{\R}$.
Since $K(V)\cap J(f)$ is a hyperbolic set, by~\eqref{eqn:setN} we conclude that~$\HD(N) < \HD(J(f))$. Since
$$ \HD((J(f) \cap V) \setminus J(F))
=
\HD(N \cup I)
\le
\max \{ \HD(N), \delta_{\bad}(\hV) \}, $$
$\delta_{\bad}(\hV) < \HD(J(f))$ implies $ \HD((J(f) \cap V) \setminus J(F)) < \HD(J(f))$.
\end{proof}

\begin{proof}[Proof of Proposition~\ref{prop:dimensionestimate}]
By Lemma~\ref{lem:db}, for a sufficiently small nice couple $(\hV, V)$ we have $\delta_{\bad}(\hV)< \HD(J(f))$.
By Lemma~\ref{lemm:dimensionD}, it follows that~$\HD(J(f)) = \HD(\Jcon(f))$, and that
\begin{equation}\label{eqn:dimcan}
\HD((J(f)\cap V)\setminus J(F))<\HD(J(f)).
\end{equation}
Hence $\HD(J(F)\cap V^c) = \HD(J(f))$ for each $c \in \Crit'(f)$.

It remains to show that~$\HDhyp(f) \ge \HD(J(f))$.
To do this let $D$ be the domain of $F$ and consider an enumeration~$( W_n )_{n \ge 1}$ of the connected components of~$D$.
For each integer~$n_0 \ge 1$ let~$F_{n_0}$ be the restriction of~$F$ to~$\bigcup_{n = 1}^{n_0} W_n$.
Then the maximal invariant set~$J(F_n)$ of~$F_n$ is contained in a uniformly hyperbolic set of~$f$.
Together with~\cite[Theorem~$4.2.13$]{MauUrb03} this implies that
$$ \HDhyp(f) \ge \lim_{n \to \infty} \HD(J(F_n)) = \HD(J(F)) = \HD(J(f)). $$
\end{proof}
Let us mention the following consequence of
Lemma~\ref{lemm:dimensionD} and Theorem~\ref{t:finiteness of bad} to
conclude this section.
\begin{coro}\label{cor:HDhyp}
Assume that $f\in \sA$ is expanding away from critical points and backward contracting.
In the real case, assume furthermore that $f$ is essentially topologically exact on the Julia set.
Then
$$\HD(J(f)\setminus \Jcon (f))=0.$$
\end{coro}
\begin{proof}
If~$\Crit'(f) = \emptyset$ then~$f$ is uniformly hyperbolic and the
result is immediate. Otherwise the assumptions imply that $f\in
\sA^*$ by Fact~\ref{f:nice couples}. By Theorem~\ref{t:finiteness of
bad}, $\delta_{\bad}(f)=0$ so the assertion follows by
Lemma~\ref{lemm:dimensionD}.
\end{proof}

This was shown in~\cite[Proposition~$7.3$]{GraSmi09} for rational maps satisfying the summability condition with each positive exponent, and in~\cite[\S$1.4$]{PrzRiv07} for rational maps satisfying the TCE condition.
See also~\cite{Sen03} for related results in the case of Collet-Eckmann interval maps.

\begin{rema}
A direct consequence of Corollary~\ref{cor:HDhyp} and of~\cite[Theorem~$0.2$]{Hai01} is that a backward contracting rational map that is expanding away from critical points, and that is not a Latt{\`e}s example, has no invariant line fields and is quasi-conformally rigid.
In fact, the conclusion of Corollary~\ref{cor:HDhyp} shows that such a map is ``uniformly weakly hyperbolic'' in the sense of~\cite{Hai01}.
\end{rema}

\subsection{Tail estimate}\label{ss:tail estimate}
Let us start with some notation.
Let~$(\hV, V)$ be a nice couple for~$f$.
Recall that for each integer~$m \ge 0$, we denote by~$\sM_m(\hV)$ the collection of connected components of~$f^{-m}(\hV)$ (\S\ref{ss:notation}).
Moreover, for~$m \ge 1$ we denote by~$\fB_m(\hV)$ the collection of bad pull-backs of~$\hV$ by~$f^m$ (Definition~\ref{def:badpb}).
In what follows, $\fB_0(\hV) \= \sM_0(\hV)$.

Let~$\fL_V$ be the collection of components of $\dom(f) \setminus K(V)$.
For $U\in \fL_V$, let~$l(U)=l_V(U)$ denote the landing time of~$U$ into~$V$.
Then for each $U\in \fL_V$, there exists a set~$\hU\supset U$ such that~$f^{l(U)}$ maps~$\hU$ diffeomorphically onto a component of~$\hV$.
Moreover, if $U\not\subset V$, then $\hU\cap V=\emptyset$.

For each $\tY\in \sM_{\wtm}(\hV)$ with $\wtm\ge 0$, we use~$\fD_{\tY}$
to denote the collection of all simply connected sets~$W$ for which
the following holds: there exist $\tY \supset \hW \supset W$ and $U \in \fL_V$ such that $U\subset f(V)$ and such that~$f^{\wtm + 1}$
maps~$W$ diffeomorphically onto~$U$ and maps~$\hW$ diffeomorphically
onto~$\hU$.

We will need the following lemma, which is~\cite[Lemma~$3.5$]{PrzRiv11}.
It is worth noticing that this is the only place where we use a nice couple, as opposed to a nested pair of nice sets.
\begin{lemm}\label{l:decompfD}
Let~$F : D \to V$ be the canonical induced map associated to~$(\hV, V)$, let~$\fD$ be the collection of all connected components of~$D$ and let~$m(x)$ be the canonical inducing time of $x\in D$.
Then
\begin{equation*}
\fD = \bigcup_{\wtm=0}^\infty \bigcup_{\tY\in \fB_{\wtm}(\hV)}
\fD_{\tY},
\end{equation*}
and for each $\wtm\ge 0$, $\tY\in \fB_{\wtm}(\hV)$ and $x\in D\cap
\tY$,
\begin{equation}\label{eqn:mltime}
m(x)=\wtm+ 1+l(f^{\wtm+1}(x)).
\end{equation}
\end{lemm}
\begin{proof}
Clearly for each $\tY\in\sM_{\wtm}(\hV)$ and $x\in W\in \fD_{\tY}$ we
have $x\in D$ and $m(x)\le \wtm+1+l(f^{\wtm+1}(x))$.
Moreover, if $\tY\in\fB_{\wtm}(\hV)$, then $m(x)>\wtm$, since~$\tY$ is disjoint
from any diffeomorphic pull-back of $\hV$ by $f^{m'}$ for any $m'\le
\wtm$. It follows that for $\tY\in \fB_{\wtm}(\hV)$,
\eqref{eqn:mltime} holds for all $x\in D\cap \tY$ and
$\fD_{\tY}\subset \fD$.

It remains to prove that a connected component~$W$ of~$D$ belongs to~$\fD_{\tY}$ for some $\tY\in \fB_{\wtm}(\hV)$.
If~the canonical inducing time $m(W)$ is the first return time
of~$W$ to~$V$, then~$f(W) \in \fL_V$ and, if we denote by~$\tY$ the
connected component of~$\hV$ containing~$W$, then~$W \in \fD_{\tY}$.
Suppose now that~$m(W)$ is not the first return time of~$W$ to~$V$,
let~$n \in \{1, \ldots, m(W) - 1 \}$ be the penultimate return
time of~$W$ to~$V$, and put~$W' \= f^n(W)$.
As we clearly have~$f(W') \in \fL_V$, we just need to show that the
pull-back~$\tY$ of~$\hV$ by~$f^n$ containing~$W$ is bad. Arguing by
contradiction, assume the contrary.
Then by Lemma~\ref{lem:db0}, $\tY$ is contained in a diffeomorphic pull-back of~$\hV$ by~$f^j$ for some~$j \in \{ 1, \ldots, n \}$.
Then~$f^j(W)$ is contained in a component~$U$ of~$\dom (f)\setminus K(V)$.
Clearly~$j+l(U)\le n$, $f^{j+l(U)}(W)\subset V$ and~$f^{j+l(U)}$ maps a neighborhood of $W$ diffeomorphically onto a
component $\hV$. This implies that the canonical time of $W$ is not
greater than $n$, which is a contradiction.
\end{proof}

The following proposition is a crucial estimate.

\begin{prop} \label{prop:whitney}
Assume that $f\in\sA^*$ satisfies the $\Theta$-Shrinking Condition
for some slowly varying and monotone decreasing sequence of positive
numbers $\Theta=\{\theta_n\}_{n=1}^\infty$. Then for each
sufficiently small symmetric nice couple $(\hV, V)$ for~$f$, with
$\delta_{\bad}(\hV)< \HDhyp(f)$, there exists a constant~$\alpha_0\in
(\delta_{\bad}(\hV), \HDhyp(f))$ such that, for real numbers
$\alpha, t,$ with
$$\alpha\ge \alpha_0, \,\, t\in (\delta_{\bad}(\hV),\alpha),$$
the following holds: there is a constant $C_1 > 0$ such that for
$Y\subset\tY\in \sM_{\wtm}(\hV)$ with $\wtm\ge 0$ and each integer~$m \ge 1$, if we put
$$ D(\tY) \= d_{\hV}(\tY)\left(\log d_{\hV}(\tY)+1\right), $$
then
\begin{equation}\label{eqn:pbD0s}
\sum_{W\in \fD_{\tY} \, : \, W \subset Y, m(W) \ge m} \diam (W)^\alpha
\le
C_1D(\tY)\diam(Y)^{t}\left( \sum_{i=m}^\infty \theta_i^{\alpha-t}\right),
\end{equation}
where~$m(W)$ is the canonical inducing time on~$W$ with respect to~$(\hV, V)$.
\end{prop}

To prove this proposition, we will apply a technique based on a \emph{Whitney decomposition} of the complement of the critical values of
$f^{\wtm+1}:\tY\to f(\hV)$.
The proof is rather long and we suspend it to~\S\ref{ss:proof of whitney} and complete the proof of Theorem~\ref{t:tail estimate} now.

\begin{proof}[Proof of Theorem~\ref{t:tail estimate}]
By Proposition~\ref{prop:dimensionestimate}, the first part of the theorem holds.
Now fix $t \in (\delta_{\bad}(f), \HD(J(f)))$, and assume that $f$
satisfies the $\Theta$-Shrinking Condition for some slowly varying
and monotone decreasing sequence of positive numbers
$\Theta=\{\theta_n\}_{n=1}^\infty$. Let $(\hV, V)$ be a sufficiently
small nice couple so that the conclusion of
Proposition~\ref{prop:whitney} holds and such that $\delta_{\bad}
(\hV) < t$.
Such a nice couple exists by Lemma~\ref{lem:db}.
Then there exists~$\eta > 0$ such that
$$ \sum_{\wtm=0}^\infty \sum_{\tY\in \fB_{\wtm}(\hV)} d_{\hV}(\tY) \diam(\tY)^{t-\eta}<\infty.$$
As $D(\tY)/d_{\tV}(\tY)^{t/(t-\eta)}$ is bounded from above, it
follows that
$$ C_0
\=
\sum_{\wtm=0}^\infty \sum_{\tY\in \fB_{\wtm}(\hV)} D(\tY) \diam(\tY)^t
<
\infty.$$
Fix an integer~$m \ge 1$, let~$D$ be the domain of the canonical induced map associated to
$(\hV, V)$, and let~$\fD$ be the collection of its connected
components.
By Lemma~\ref{l:decompfD},
\begin{equation}\label{e:bad decomposition}
\sum_{\substack{W\in\fD \, : \\ W \subset Y,  m(W) \ge m}}\diam (W)^\alpha
=
\sum_{\wtm=0}^\infty\sum_{\tY\in\fB_{\wtm}(\hV)} \sum_{\substack{W \in \fD_{\tY} \, : \\ W \subset Y,  m(W) \ge m}}\diam (W)^\alpha.
\end{equation}
Applying Proposition~\ref{prop:whitney} with~$t$ replaced by~$t +
\sigma$, we obtain that there is a constant~$C_1 > 0$ such that for
each integer~$\wtm \ge 0$ and each~$\tY \in \fB_{\wtm}(\hV)$,
\begin{equation*}
\sum_{\substack{W\in \fD_{\tY} \, : \\  W \subset Y,  m(W) \ge m}} \diam (W)^\alpha
\le
C_1 D(\tY) \diam(Y \cap \tY)^{t + \sigma} \left(\sum_{i=m}^\infty \theta_i^{\alpha-t-\sigma}\right).
\end{equation*}
Combined with~\eqref{e:bad decomposition} and the inequality $\diam(Y \cap \tY)^{t + \sigma} \le \diam(\tY)^t \diam(Y)^\sigma$, we obtain,
\begin{multline*}
\sum_{\substack{W\in\fD \, : \\ W \subset Y, m(W) \ge m}}\diam (W)^\alpha
\\ \le
C_1 \left( \sum_{\wtm=0}^\infty\sum_{\tY\in\fB_{\wtm}(\hV)} D(\tY)
\diam(\tY)^t \right) \diam(Y)^\sigma \left(\sum_{i=m}^\infty
\theta_i^{\alpha-t-\sigma}\right).
\end{multline*}
This proves the desired upper bound with $C=C_0C_1$.
\end{proof}

\subsection{Proof of Proposition~\ref{prop:whitney}}
\label{ss:proof of whitney}
The whole section is devoted to the proof of
Proposition~\ref{prop:whitney}.
By assumption, there exist constants~$C_1 > 0$ and~$\rho>0$ such that for any $x\in J(f)$ and any $n\ge 1$, the
component of $f^{-n}(B(f^n(x), \rho))$ that contains~$x$ has
diameter not greater than $C_1 \theta_n$. Let $(\hV, V)$ be a
symmetric nice couple for~$f$ so that
$$ \delta_{\bad}(\hV) < \HDhyp(f)
\text{ and }
\hV \subset \tB(\Crit'(f),\rho/4). $$
Furthermore, let~$K_0 > 1$ and $\rho_0 \in (0, \rho)$ be the constants given by
Lemma~\ref{lem:definitesize} for this choice of~$(\hV, V)$
and~$\rho$ and let~$\alpha_0 \in (0, \HDhyp(f))$ and~$C_0,
\varepsilon_0 >0$ be the constants given by
Lemma~\ref{lem:finitelanding} for the choice of~$V$. We fix~$\alpha
\ge \alpha_0$, $t \in (\delta_{\bad}(\hV), \alpha)$, an integer~$\wtm
\ge 0$, a connected component~$\tY$ of~$f^{-\wtm}(V)$, a subset~$Y$
of~$\tY$, and an integer~$m \ge 1$. Put~$s \= \min \{ m - \wtm - 1,
0\}$.

Let~$E$ be the set of all critical values of $f^{\wtm+1}: \tY\to
f(\hV)$. Since this map is a composition of unicritical maps, we
have that for some~$C_2 > 0$ independent of~$\tY$,
\begin{equation}\label{eqn:carE}
\# E \le C_2(\log d(\tY)+1).
\end{equation}

We shall define a family  $\sQ$ of intervals/squares (of Whitney
type) that cover $J(f)\cap f(V)\setminus E$ and then pull it back
by~$f^{\wtm + 1}$ to obtain a family $\sP$ of subsets of $\tY$.
For each $P\in \sP$, we shall estimate the total size of elements of~$\fD_{\tY}$ contained in~$Y$, which are roughly contained in~$P$.
For technical reasons, in the case that~$f$ is a rational map, we shall assume
that $f(\hV)$ is bounded in $\C$.
This assumption causes no loss of
generality since we may conjugate $f$ by a rotation in such a
way that~$\infty$ is not in the closure of~$f(\hV)$.

We identify~$\C$ with~$\R^2$ in the usual way. For an integer $n$,
by a \emph{dyadic interval of (geometric) depth~$n$}, we mean an
interval of~$\R$ of the following form: $[k\cdot 2^{-n}, (k+1)\cdot
2^{-n})$, where~$k$ is an integer.
A \emph{dyadic square of
(geometric) depth~$n$} is the product of two dyadic intervals of the
same depth in~$\C$. For a dyadic interval (resp. square)~$Q$, we use
$\dep(Q)$ to denote its depth. Moreover, we use $Q''$ and~$Q'$ to denote the
closed concentric interval (resp. square) such that
$$ \diam (Q'') = 2 \diam(Q') = 4\diam (Q),$$
and use~$\hQ$ to denote the smallest dyadic interval/square with
$\hQ\supsetneq Q$.

\partn{1}
In the real (resp. complex) case,  let $\sQ$ be the collection of
all dyadic intervals (resp. squares) $Q$ such that
$$ {Q}\cap f(V) \cap J(f) \not=\emptyset, Q''\subset f(\hV)\setminus E, $$
and such that~$\hQ$ does not satisfy these properties.
Note that the elements of~$\sQ$ are pairwise disjoint and that
\begin{equation}\label{eqn:sQE}
\bigcup_{Q\in\sQ} {Q}\supset (f(V)\cap J(f))\setminus E.
\end{equation}
On the other hand, by the maximality in the definition of~$\sQ$, it follows that there is a constant~$C_3 > 0$ independent of~$\tY$ such that for each~$Q \in \sQ$ we have either
\begin{equation}
  \label{e:Whitney}
  \hQ'' \cap E \neq \emptyset
\text{ or }
\diam(Q) \ge C_3 \min_{c \in \Crit'(f)} \diam(\hV^c).
\end{equation}

For each $Q\in\sQ$, let $\sP(Q)$ be the collection of all components
of $f^{-\wtm-1}(Q)\cap \tY$ and let $\sP=\bigcup_{Q\in\sQ} \sP(Q)$.
Furthermore, for each~$P \in \sP(Q)$ we denote by~$P'$ the pull-back of~$Q'$ by~$f^{\wtm + 1}$ containing~$P$.

\partn{2}
We will now complete the proof of the proposition in the special
case where there exist~$Q \in \sQ$ and~$P \in \sP(Q)$ such that~$Y
\subset P'$.
We assume that there is at least one element
of~$\fD_{\tY}$ contained in~$Y$, otherwise the desired estimate is
trivial. Let~$n$ be given by Lemma~\ref{lem:definitesize} with some
$x \in Z\=f^{\wtm + 1}(Y)$ and with~$\delta = \diam(Z)$. Since there
is at least one element of~$\fD_{\tY}$ contained in~$Y$, it follows
that~$Z$ contains an element of~$\fL_V$. So we must fall into the
first case of this lemma. Thus, the distortion of~$f^n$ on~$f^{\wtm +
1}(Y) \subset B(x, \diam(Z))$ is bounded by~$K_0$, and
$$ \rho_0/(2K_0) < \diam(f^{n + \wtm + 1}(Y)) < \rho. $$
Since our hypotheses imply that the distortion of~$f^{\wtm + 1}$ on~$Y$ is uniformly bounded,
it follows that there is a constant~$C_4 > 0$ independent of~$\tY$ such that
$$ \sum_{\substack{W \in \fD_{\tY} \, : \\ W \subset Y, m(W) \ge m}} \diam(W)^\alpha
\le
C_4 \diam(Y)^\alpha \sum_{\substack{U \in \fL_V \, : \\ U \subset f^{n + \wtm + 1}(Y), \\ l(U) \ge m - (n + \wtm + 1)}} \diam(U)^\alpha. $$
By Lemma~\ref{lem:finitelanding}, if we put~$m_0 \= \max \{ m - (n + \wtm + 1), 0 \}$, then,
$$ \sum_{W \in \fD_{\tY} \, : \, W \subset Y, m(W) \ge m} \diam(W)^\alpha
\le
C_4C_0 \diam(Y)^\alpha \exp(- \varepsilon_0 m_0). $$
Since~$\diam(Y) \le C_1 \theta_{n + \wtm + 1}$, the desired estimate follows in this case from the inequality~$\diam(Y)^\alpha \le C_1^{\alpha - t} \diam(Y)^t \theta_{n + \wtm + 1}^{\alpha - t}$, and from the fact that~$\Theta$ is slowly varying.

\partn{3}
From now on we assume that for each~$P \in \sP$ the set~$Y$ is not
contained in~$P'$. This implies that there is a constant~$C_5 > 0$
independent of~$\tY$ such that for each~$P \in \sP$ intersecting~$Y$,
\begin{equation}
  \label{e:Psize}
  \diam(P) \le C_5 \diam(Y).
\end{equation}

Fix a neighborhood $V_0$ of $\Crit'(f)$ with $V_0\Subset V$.
For each~$U\in \fL_V$, choose a point $z_U\in U\setminus E$ with~$f^{l(U)}(z_U)\in V_0$.
By the Koebe principle, there exists a constant~$\kappa>0$ such that for all~$U \in \fL_V$,
\begin{equation}\label{eqn:Ushape}
U\supset B(z_U, \kappa\diam (U)).
\end{equation}

Recall that we have fixed an integer~$m \ge 1$ and that $s = \min\{m
- \wtm - 1, 0 \}$.
For $Q\in\sQ$ and
$P\in\sP(Q)$, let
\begin{align*}
\fL(Q;s)
& \=
\{U \in \fL_V: z_U \in {Q}, U\subset f(V), l(U) \ge s \},
\\
\fD_{Y}(P;s)
& \=
\{W\in \fD_{\tY}: W\subset Y, W\cap P\not=\emptyset, f^{\wtm+1}(W)\in \fL(Q;s)\},
\end{align*}
\begin{equation*}
Q^* \= Q \cup \left(\bigcup_{U\in \fL(Q; s)} U\right),
\text{ and }
P^*_Y \= P \cup \left(\bigcup_{W\in \fD_{Y}(P; s)} W\right).
\end{equation*}
Clearly~$f^{\wtm + 1}(P_Y^*) \subset Q^*$ and by~\eqref{e:Psize},
\begin{equation}\label{e:P*size}
\diam (P_Y^*)\le (C_5+1) \diam (Y).
\end{equation}
Furthermore, we put
$$ \sQ^{\sharp}
\= \{ Q \in \sQ : \text{ there is $P \in \sP(Q)$ such that $\fD_Y(P;
s) \neq \emptyset$} \}. $$
Clearly each~$Q \in \sQ^{\sharp}$ is such that~$\fL(Q; s) \neq \emptyset$.
On the other hand, by~\eqref{eqn:sQE} for each~$U \in \fL_V$ contained in~$f(V)$ and with~$l(U) \ge s$, the point~$z_U$ is contained in a unique~$Q \in \sQ$.
Therefore
\begin{equation}
  \label{e:decomposition}
  \{ W \in \fD_{\tY}: W \subset Y, m(W)\ge m \}
=
\bigcup_{Q \in \sQ^{\sharp}} \bigcup_{P \in \sP(Q)} \fD_Y(P; s).
\end{equation}

\partn{4}
For each~$Q \in \sQ^{\sharp}$ fix~$x_Q \in Q$ and let~$n_Q \ge 0$ be the integer given by Lemma~\ref{lem:definitesize} with~$x = x_Q$ and~$\delta = \diam(\hQ'' \cup Q^*)$. Since $Q^*$ contains an element of $\fL_V$, we must fall into the first case of the lemma.
So the distortion of~$f^{n_Q}$ on the ball~$B_Q \= B(x_Q, \diam(\hQ'' \cup Q^*))$, and hence on~$\hQ'' \cup Q^*$, is bounded by~$K_0$ and we have,
$$ \rho > \diam(f^{n_Q}(B_Q)) > \rho_0.$$
Since~$\diam(Q') / \diam(\hQ'')$ is bounded independently of~$\tY$,
it follows that there is a constant $\rho_1
> 0$ independent of~$\tY$ such that,
\begin{equation}
  \label{eqn:nQsize}
  \rho > \diam(f^{n_Q}(Q^*)) > \rho_1.
\end{equation}

\partn{5}
For each~$n \ge 0$ let $\sQ_n^\sharp \= \{Q\in\sQ^\sharp: n_Q=n\}$.
We will prove that there is a constant~$C_6 > 0$ independent of~$\tY$ such that for each integer~$n$,
\begin{equation}\label{e:QEn}
\# \sQ_n^\sharp \le C_6 (\# E+1).
\end{equation}
To prove this, we decompose $\sQ_n^\sharp$ into the following subsets:
\begin{align*}
\sQ^1_n
& \=
\{Q\in \sQ^\sharp_n: \hQ''\cap E=\emptyset\},\\
\sQ^2_n
& \=
\{Q\in\sQ^\sharp_n\setminus\sQ_n^1: \text{there is } U\in\fL(Q;s) \text{ such that } U\supset \hQ''\},\\
\sQ^3_n
& \=
\sQ_n^\sharp\setminus (\sQ_n^1\cup\sQ_n^2).
\end{align*}

We first observe that from~\eqref{e:Whitney}, and from the fact that the elements of~$\sQ$ are pairwise disjoint, it follows that $\#\sQ_n^1$ is bounded from above by a constant independent of~$\tY$.

For $Q \in \sQ_n \setminus \sQ_n^1$, there exists an~$e\in E\cap \hQ''$.
Clearly, for each $e\in E$, $\sQ_n^2$ contains at most one element~$Q$ with $\hQ''\ni e$.
Thus~$\#\sQ_n^2\le \# E$.

To complete the proof of~\eqref{e:QEn}, it suffices to prove that for
each $e\in E$, the cardinality of $\sQ_n^3(e)=\{Q\in\sQ_n^3:
\hQ''\ni e\}$ is bounded from above independently of~$\tY$. Since
$\dist(e, Q)/\diam (Q)\le \diam (\hQ'')/\diam (Q)$ is uniformly
bounded for $Q\in\sQ_n^3(e)$, the statement follows once we prove
that any two elements~$Q_1$, $Q_2$ of~$\sQ_n^3(e)$ have comparable
diameters.
To prove this we first observe that, by~\eqref{eqn:Ushape}, for each~$Q \in \sQ_n^3$ the quotient~$\diam(Q^*) / \diam(Q)$ is uniformly bounded.
On the other hand, for $Q_1, Q_2 \in \sQ_n^3(e)$ the sets~$\hQ_1''$ and~$\hQ_1''$
both contain~$e$, so the distortion of~$f^n$ on~$\hQ_1''\cup Q_1^*\cup\hQ_2''\cup Q_2^*$ is uniformly bounded.
Since furthermore, $\diam(f^{n}(Q_1^*))\asymp \diam (f^n(Q_2^*))\asymp 1$, we have
$\diam (Q_1^*)\asymp \diam (Q_2^*)$, and hence~$\diam(Q_1) \asymp
\diam(Q_2)$.
This completes the proof of~\eqref{e:QEn}.

\partn{6}
For~$Q \in \sQ^\sharp$, put
\begin{equation}\label{eqn:sQ}
s_Q \= \inf \{l(U): U\in \fL(Q;s)\}\in \{s, s+1, \ldots\}.
\end{equation}
For each $U\in \fL(Q; s)$, we have $l(U)\ge n_Q$, since $f^{l(U)}(U)$
contains a critical point, while $U\subset Q^*$, so $f^{n_Q} : U \to
f^{n_Q}(U)$ is a diffeomorphism.
Thus
$$ s_Q\ge n_Q,
f^{n_Q}(U) \in \fL_V
\text{ and }
l(f^{n_Q}(U))\ge s_Q-n_Q. $$

Let us prove that there exists a constant~$C_7>0$ such
that for each~$Q\in\sQ^\sharp$ and~$P\in \sP(Q)$,
\begin{equation}
\label{eqn:pullbackQlarges}
\sum_{W \in \fD_{Y}(P;s)} \diam(W)^\alpha
\le
C_7 \diam(P^*_Y)^\alpha \exp(-\eps_0(s_Q-n_Q)).
\end{equation}

To this end, we first show that $f^{n_Q + \wtm + 1}|P^*_Y$ has
uniformly bounded distortion. Indeed, since $f^{\wtm+1}(P^*_Y)\subset
Q^*$, and $f^{n_Q}|Q^*$ has bounded distortion, it suffices to prove
that $f^{\wtm+1}|P^*_Y$ has bounded distortion. Since $Q'' \cap E =
\emptyset$, the pull-back of~$Q''$ by $f^{\wtm+1}$ that contains $P$
is diffeomorphic, so by the Koebe principle, $f^{\wtm + 1}|P$ has
uniformly bounded distortion.
Moreover, for each $W\in \fD_{\tY}$, we have
$U\=f^{\wtm+1}(W)\in\fL_V$ and $f^{\wtm +1+l(U)-j}|f^j(W)$ has uniformly
bounded distortion for $j=0,1,\ldots, \wtm+1+l(U)$, since it extends
to a diffeomorphism onto the component of $\hV$ that contains
$f^{l(U)}(U)(\subset V)$.
Therefore, $f^{\wtm+1}|W$ has uniformly bounded distortion. It follows that $f^{\wtm+1}|P^*_Y$ has uniform bounded distortion.

Consequently, there is a constant~$C_8 > 0$ independent of~$\tY$ such that
\begin{equation*}
  \sum_{W\in\fD_{Y}(P;s)}\diam (W)^\alpha
 \le C_8 \frac{\diam (P^*_Y)^\alpha}{\diam
(f^{n_Q}(Q^*))^\alpha} \sum_{U\in\fL(Q;s)}\diam (f^{n_Q}(U))^\alpha.
\end{equation*}
Together with~\eqref{eqn:nQsize} and Lemma~\ref{lem:finitelanding}, this implies~\eqref{eqn:pullbackQlarges}.

\partn{7}
We are ready to complete the proof of the proposition. For $P\in\sP$
with $\fD_{Y}(P;s)\not=\emptyset$, let
$Q=f^{\wtm+1}(P)\in\sQ^\sharp$.
Since
$$ f^{n_Q + \wtm + 1}(P^*_Y)\subset f^{n_Q}(Q^*)
\text{ and }
\diam(f^{n_Q}(Q^*))<\rho, $$
we have
\begin{equation}\label{eqn:p*1}
\diam (P^*_Y) \le C_1\theta_{n_Q + \wtm + 1}.
\end{equation}
So by~\eqref{e:P*size} and \eqref{eqn:pullbackQlarges}, if we put~$C_9 = C_7 (C_5+1)^t C_1^{\alpha - t}$, then
\begin{equation*}
\sum_{W\in \fD_Y(P;s)}\diam (W)^\alpha
\le
C_9 \diam(Y)^t \theta_{n_Q + \wtm + 1}^{\alpha - t} \exp (-\eps_0 (s_Q-n_Q)).
\end{equation*}
This inequality certainly holds also in the case~$\fD_{Y}(P;s)=\emptyset$.
For each $Q\in\sQ^\sharp$, we have
$\#\sP(Q)\le d_{\hV}(\tY)$, so
\begin{multline*}
 \sum_{P\in\sP(Q)}\sum_{W\in\fD_Y(P;s)} \diam (W)^\alpha
\\ \le
C_9 d_{\hV}(\tY) \diam(Y)^t \theta_{n_Q + \wtm + 1}^{\alpha - t} \exp
(-\eps_0 (s_Q-n_Q)).
\end{multline*}
Recall that for each~$Q \in \sQ^{\sharp}$, we have~$s_Q \ge \max (n_Q,s)$.
Thus, by~\eqref{e:decomposition},
\begin{multline*}
\sum_{W \in \fD_{\tY}(s) \, : \, W \subset Y}\diam (W)^\alpha
=
\sum_{n=0}^\infty \sum_{Q\in \sQ_n^\sharp}
\sum_{P\in\sP(Q)}\sum_{W\in\fD_Y(P;s)} \diam (W)^\alpha
\\
\begin{aligned}
& \le
C_9 d(\tY) \diam(Y)^t \sum_{n = 0}^{\infty} \#\sQ_n^{\sharp} \theta_{n+\wtm+1}^{\alpha - t} \exp (-\eps_0 \max(0, s-n)).
\\ & \le
C_9 C_6 (C_2+1) D(\tY) \diam(Y)^t \sum_{n = 0}^{\infty}
\theta_{n+\wtm+1}^{\alpha - t} \exp (-\eps_0 \max(0, s-n)),
\end{aligned}
\end{multline*}
where in the last inequality we used~\eqref{e:QEn} and~\eqref{eqn:carE}.
The desired inequality follows from the fact that the sequence~$\{ \theta_n \}_{n = 1}^\infty$ is slowly varying.

The proof of Proposition~\ref{prop:whitney} is completed.

\subsection{Proof of Corollary~\ref{coro:inducing}}\label{ss:proof of inducing}
The whole section is devoted to the proof of
Corollary~\ref{coro:inducing}.
The crucial step is to prove existence of a conformal measure
supported on~$\Jcon(f)$ and the uniform estimate on its local
dimension.
The rest is a rather simple application of Young's result.
We shall use the following lemma, which is proved in~\cite[Theorem~$2$]{PrzRiv07} in the complex case.
\begin{lemm}
\label{lem:confmeas}
Assume that $f \in \sA^*$ has a nice couple $(\hV, V)$ such that the
associated canonical induced map $F: D\to V$ satisfies the
following:
\begin{enumerate}
\item[1.]
For every $c\in \Crit'(f)$, we have $\HD(J(F)\cap V^c)=\HD(J(f))$.
\item[2.]
There exists a constant~$\alpha\in (0, \HD(J(f)))$ such that $\sum_{W\in\fD} \diam (W)^\alpha < \infty$, where $\fD$ is the collection of components of $D$.
\end{enumerate}
Then there is a conformal measure~$\mu$ of exponent $\HD(J(f))$ for~$f$ that is ergodic, supported on $\Jcon(f)$, and satisfies
\begin{equation}\label{eqn:propertymu}
 \HD(\mu)=\HD(J(f))
\text{ and }
\mu(V \setminus D) = 0.
\end{equation}
Furthermore, any other conformal measure for~$f$ supported on~$J(f)$ is of exponent strictly larger than~$\HD(J(f))$ and supported on~$J(f) \setminus \Jcon(f)$.
\end{lemm}

\begin{proof}
First of all, it suffices to prove that there exists a conformal measure~$\mu$ of exponent~$\HD(J(f))$ that is supported on $\Jcon(f)$ and that satisfies \eqref{eqn:propertymu}.
The ergodicity of~$\mu$, as well as the assertions concerning the other conformal measures, follow from the fact that~$\mu$ is supported on~$\Jcon(f)$, see~\cite{DenMauNitUrb98} or~\cite{McM00}, where only the complex case was considered, but the proof extends without change to the real case.

The following is a slight modification of the proof
of~\cite[Theorem~$2$]{PrzRiv07}, given for rational maps. The
modification is necessary for the real case since it is \emph{a
priori} unknown whether $f$ has a conformal measure of exponent $\HD(J(f))$.
As in the complex case, the assumptions imply that~$F$ has a conformal measure~$\nu$ of exponent $t_0\=\HD(J(F))=\HD(J(f))$, with $\nu(J(F))=1$, and $\HD(\nu)=\HD(J(F))$.
Note that we do not need~$F$ to be topologically mixing since $\HD(J(F)\cap V^c)$ does not depend on~$c$.

Let $G: \dom(f)\setminus K(V)\to V$ denote the first landing map
to~$V$, \emph{i.e.}, for each $x\in \dom(f) \setminus K(V) $,
$G(x)=f^s(x)$, where $s$ is the minimal non-negative integer such
that $f^s(x)\in V$. For each component $W$ of $\dom(f)\setminus
K(V)$, we define a measure $\nu_W$ as follows:
$$\nu_W (E)=\int_{G(E)} |D(G|W)^{-1}|^{t_0} d\mu, \text{ for }
E\subset W.$$ Since $f$ is expanding outside the critical points,
the distortion of $G|W$ is bounded from above by a constant
independent of $W$.
Thus there is a constant $C>0$ such that $\nu_W(W)\le C\diam (W)^{t_0}$ for every component $W$ of $\dom
(f)\setminus K(V)$. Since $\sum_W \diam (W)^\alpha < \infty$ holds
for some $\alpha< \HD_{\hyp}(f)$
(Lemma~\ref{lem:finitelanding} with $m = 1$), the measure
$\mu_0\=\sum_W \nu_W$ is finite. Let $\mu$ be the normalization of $\mu_0$.
Then $\mu$ is supported on $\Jcon(f)$, satisfying~\eqref{eqn:propertymu}.
It remains to show that~$\mu$ is a conformal measure of exponent $t_0$: for any Borel set $A\subset J(f)$ for which $f|A$ is injective,
$$\mu(f(A))=\int_A |Df|^{t_0}d\mu.$$

Indeed, by writing~$A$ as a finite union of subsets, we only need to
consider the following cases:

\noindent {\bf Case 1.}
$A\cap (K(V)\cup (V\setminus D))=\emptyset$.
Then this equality holds, as shown in Cases~$1$ and~$2$ of the proof of
Proposition B.$2$ of~\cite{PrzRiv07}.

\noindent {\bf Case 2.}
$A\subset K(V)$.
Then the equality holds since both sides are equal to~$0$.

\noindent {\bf Case 3.}
$A$ is a finite set.
As clearly~$\mu$ has no atom, the equality holds.

\noindent {\bf Case 4.}
$A\subset V\setminus (D\cup \Crit'(f))$.
In this case, $\mu(A)=0$, so we only need to prove $\mu(f(A))=0$.
Since $x\in A$ has no good time, $f(x)$ can only have at most
finitely many good times, so either $f(x)\in K(V)$ or
$G(f(x))\not\in J(F)$.
Thus $G(f(A)\setminus K(V))\subset V\setminus J(F)$, so~$\mu(f(A))=0$.

This proof is completed.
\end{proof}

\begin{proof}[Proof of Corollary~\ref{coro:inducing}]

\

\partn{1}
Assume $\gamma(f)>1$.
Choose
$$ \sigma_0 \in (\HD(J(f)) - \varepsilon, \HD(J(f)) - \delta_{\bad}(f) -
\beta_{\max}(f)^{-1}), $$
$t \in (\delta_{\bad}(f), \HD(J(f)))$ and
$\beta \in (0, \beta_{\max}(f))$, so that
$$ \beta (\HD(J(f)) - t - \sigma_0) > 1. $$
Let~$(\hV, V)$ be a nice couple for~$f$ given by Theorem~\ref{t:tail
estimate} for this choice of~$t$ and for~$\Theta = \{ n^{- \beta}
\}_{n = 1}^\infty$.
Applying this theorem with~$\alpha = \HD(J(f)), \sigma = 0, m = 1$, and with $Y$ equal to each of the connected
components of~$V$, we conclude that the nice couple~$(\hV, V)$
satisfies the hypotheses of Lemma~\ref{lem:confmeas}.
So, there exists a conformal measure~$\mu$ of exponent~$\HD(J(f))$ satisfying
all the desired properties except that~\eqref{e:upper regularity conformal} is to be shown.
To do this, take $\rho>0$ and let $\rho_0$, $\kappa_0$ and $K_0$ be given by
Lemma~\ref{lem:definitesize} for this choice of $\rho$ and $(\hV,
V)$. Given~$\delta
> 0$ and $x \in J(f)$, let~$n \ge 0$ be given by this lemma. In the first case of this lemma, it follows
from the conformality of~$\mu$ and the distortion bound of~$f^n$
on~$B(x, \delta)$, that there is a constant~$C_0 > 0$ independent
of~$\delta$ and~$x$ such that,
$$ \mu(B(x, \delta)) \le C_0 \delta^{\HD(J(f))}. $$
Suppose now that we are in the second case of Lemma~\ref{lem:definitesize}, and denote by~$\fD$ the collection of connected components of~$D$.
Then, using~$\mu(V \setminus D) = 0$ (Lemma~\ref{lem:confmeas}),
$$ \mu(f^n(B(x, \delta)))
\le
\sum_{W \in \fD \, : \, W \cap f^n(B(x, \delta)) \neq \emptyset} \mu(W)
\le
\sum_{W \in \fD \, : \, W \subset f^n(B(x, \kappa_0 \delta))} \mu(W). $$
Since for each~$W \in \fD$ the map~$f^{m(W)}|W$ extends to a diffeomorphism onto a connected component of~$\hV$, it follows from the Koebe principle that there is a constant $C_1 > 0$ such that for each~$W \in \fD$,
\begin{equation}
  \label{e:bound for good}
  \mu(W) \le C_1 \diam(W)^{\HD(J(f))}.
\end{equation}
We thus have,
$$ \mu(f^n(B(x, \delta)))
\le C_1 \sum_{W \in \fD \, : \, W \subset f^n(B(x, \kappa_0 \delta))}
\diam(W)^{\HD(J(f))}. $$
Applying Theorem \ref{t:tail estimate} with
$\alpha = \HD(J(f))$, $\sigma = \sigma_0$, $Y = f^n(B(x,
\kappa_0\delta))$, and $m = 1$, we conclude that there is a
constant~$C_2
> 0$ independent of~$\delta$ and~$x$ such that,
\begin{equation*}
\mu(f^n(B(x, \delta))) \le C_2 \diam(f^n(B(x, \kappa_0
\delta)))^{\sigma_0}\le C_2'\diam (f^n(B(x, \delta)))^{\sigma_0},
\end{equation*}
where $C_2'=C_2(\kappa_0 K_0)^{\sigma_0}$. By the conformality
of~$\mu$, the distortion bound of~$f^n$ on~$B(x, \delta)$, and the
fact that $|Df(y)|\ge \rho_0$ for all $y\in B(x,\delta)$, we
conclude that there is a constant~$C_3
> 0$ independent of~$\delta$ and~$x$ such that $\mu(B(x, \delta))
\le C_3 \delta^{\sigma_0}$. This completes the proof
of~\eqref{e:upper regularity conformal}.

The fact that either $\HD(J(f)) < \HD(\dom(f))$ or~$J(f)$ has
nonempty interior follows from the existence of a conformal measure
supported on~$\Jcon(f)$, see for example~\cite[\S$8.2$]{GraSmi09}.

The assertions concerning conformal measures that are not proportional to~$\mu$ follow from Lemma~\ref{lemm:dimensionD} and Lemma~\ref{lem:confmeas}.

\partn{2}
Assume $\gamma(f)>2$, fix $\gamma\in
(0,\gamma(f)-2)$, and put $\tgamma = \gamma + 2$. Taking $t >
\delta_{\bad}(f)$ closer to $\delta_{\bad}(f)$, and $\beta \in (0,
\beta_{\max}(f))$ closer to~$\beta_{\max}(f)$ if necessary, we
assume that~$\beta(\HD(J(f)) - t) > \tgamma$.
Applying Theorem~\ref{t:tail estimate} with~$\alpha = \HD(J(f))$, $\sigma = 0$ and with~$Y$ equal to each of the connected components of~$V$, we conclude that there exists a constant~$C_4 > 0$ such that for each~$m \ge 1$,
\begin{multline*}
\sum_{W\in \fD \, : \, m(W)\ge m}\diam (W)^{\HD(J(f))}
\le
C_4 \sum_{n=m}^\infty n^{-\beta(\HD(J(f))-t)}
\\ \le
C_4 \sum_{n=m}^\infty n^{-\tgamma}
\le
2 C_4 m^{-\tgamma+1}.
\end{multline*}
Thus by~\eqref{e:bound for good} we obtain,
\begin{equation*}
\mu( \{ x \in D : m(x) \ge m \})
\le
2 C_1 C_4 m^{-\tgamma+1}.
\end{equation*}
Taking $(\hV, V)$ smaller if necessary, we may assume that for
some~$\widetilde{c} \in \Crit'(f)$ the set
$$ \{ m(W) : W \text{ connected component of $D \cap V^{\widetilde{c}}$ such that $F(W) = V^{\widetilde{c}}$} \},$$
is nonempty and its greatest common divisor is equal to~$1$. This
last result is proven in \cite[Lemma~$4.1$]{PrzRiv07} for rational
maps and its proof works without change for interval maps
in~$\sA^*$.
Then we proceed in a similar way
as in the proof of Theorem~B and Theorem~C of~\cite{PrzRiv07},
applying L.S.~Young's results in~\cite{You99} to the first return
map of~$F$ to~$V^{\widetilde{c}}$.
\end{proof}

\section{Poincar{\'e} series}\label{s:regularity density}
In this section we give the proofs of Theorems~\ref{t:regularity density}, \ref{t:fractal dimensions} and~\ref{t:holomorphic removability}, based on estimates of the Poincar{\'e} series and their integrated versions, that we state as Propositions~\ref{prop:cPzinV} and~\ref{prop:sLall}.

We fix throughout this section a map~$f$ in~$\sA$ and denote by~$\dom(f)$ its domain.
Recall that for~$s> 0$ and for a point~$x_0 \in \dom(f)$, the Poincar{\'e} series of~$f$ at~$x_0$ with exponent~$s$, is defined as
$$\cP(x_0;s)=\sum_{m=0}^\infty \cP_m(x_0;s),$$
where
$$\cP_m(x_0;s) = \sum_{x\in f^{-m}(x_0)}|Df^m(x)|^{-s}.$$
\index{$\cP_m(x_0;s)$}
Clearly, if $\mu$ is a conformal measure of exponent~$s$ without
an atom, then $d((f^m)_*\mu)/d\mu =\cP_m(\, \cdot \, ; s)$ on a set of full measure with respect to~$\mu$.

Recall that for a subset~$Q$ of~$\dom(f)$ and an integer $m\ge 0$, we denote by~$\sM_m(Q)$ the collection of all pull-backs of~$Q$ by~$f^n$.
Let
$$\theta_m(Q)
\=
\sup\{\diam (P): P\in\sM_m(Q)\},
\text{ and }
\theta(Q)
\=
\sup_{m=0}^\infty \theta_m(Q).$$
\index{$\theta_m(\cdot)$, $\theta(\cdot)$}
Moreover, for $s\ge 0$ we let
$$\sL_m(Q; s)=\sum_{P\in \sM_m(Q)} d_Q(P) \diam (P)^s,
\text{ and }
\sL(Q; s)=\sum_{m=0}^\infty \sL_m(Q;s),$$
\index{$\sL_m$, $\sL$}
where~$d_Q(P)$ is defined as in~\S\ref{subsec:be}.

Note that if $x\in J(f)$ is disjoint from the critical
orbits, then $$\cP_m(x;s)=\lim_{\eps\to 0} \frac{\sL_m(B(x, \eps);
s)}{\diam (B(x,\eps))^s}.$$

For $z\in J(f)$ and $m\ge 0$, let
$$\Delta_m(z)
\=
\dist\left(z, \bigcup_{j=0}^m f^j(\Crit'(f))\right).$$
Let~$\varepsilon_0 \in (0, 1/2)$ be sufficiently small so that the constant~$K(2\varepsilon_0)$ given by the Koebe principle (Lemma~\ref{lem:koebeC}) satisfies~$K(2\varepsilon_0) \le 2$, and put
$$\xi_m(z)
\=
\theta_m \left(B \left(z, \varepsilon_0 \Delta_{m}(z) \right) \right).$$
\index{$\Delta_m(\cdot)$}
\index{$\xi_m(\cdot)$}

Given a nice set~$\hV$, let $\fB_0(\hV) = \sM_0(\hV)$, and for $m\ge 1$,
let~$\fB_m(\hV)$ denote the collection of all elements $\tY\in
\sM_m(\hV)$ that are bad pull-backs of~$\hV$.
Moreover, for $s\ge 0$, let
$$\sL_m^{\bad} (\hV; s)
\=
\sum_{\tY\in \fB_m(\hV)} d_{\hV}(\tY) \diam (\tY)^s
\text{ and }
\sL^{\bad}(\hV;s)
\=
\sum_{m=0}^\infty \sL_m^{\bad}(\hV; s).$$
\index{$\sL_m^{\bad}, \sL^{\bad}$}

Our main technical results of this section are the following:
\begin{prop}\label{prop:cPzinV}
Assume that $f\in\sA^*$ has a conformal measure of exponent
$h_0>\delta_{\bad}(f)$. Then for each sufficiently small nice couple
$(\hV, V)$, the following hold:
\begin{enumerate}
\item[1.]
For any $s > h_0$ and~$t \in (0, s)$, there exists a constant $C>0$ such that for each $z\in V\cap J(f)$,
\begin{equation}\label{eqn:cPzs}
\cP(z;s)
\le
C \sum_{m=0}^\infty \sL_m^{\bad}(\hV; t) \xi_m(z)^{s-t}\Delta_m(z)^{-s}.
\end{equation}
\item[2.]
For each $t\in (0, h_0)$ there exists a constant~$C>0$ such
that for each $z\in V \cap J(f)$ and each integer~$n \ge 1$,
\begin{equation}\label{eqn:cPnzh}
\cP_n (z;h_0)
\le
C\sum_{m=0}^n \sL_m^{\bad}(\hV; t) \xi_m(z)^{h_0-t} \Delta_m(z)^{-h_0}.
\end{equation}
\end{enumerate}
 \end{prop}
In the following proposition we use the conical Julia set~$\Jcon(f)$ defined in~\S\ref{ss:backward contracting maps}, the best polynomial shrinking exponent~$\beta_{\max}(f)$ defined in~\S\ref{def:psc}, and the badness exponent~$\delta_{\bad}(f)$ defined in Definition~\ref{def:badpb}.
\begin{prop}\label{prop:sLall}
Assume that $f\in\sA^*$ has a conformal measure~$\mu$ of exponent
$h_0 >\delta_{\bad}(f)$ such that
$\beta_{\max}(f)(h_0-\delta_{\bad}(f))>1$ and such that for each
open set $U$ intersecting $\Crit'(f)$ we have $\mu(U)>0$.
Then there exists a constant~$\delta_0>0$ such that for each $z\in J(f)$ and each $s>h_0$,
$$ \sL(B(z,\delta_0); s)<\infty. $$
Moreover, if $\mu(\Jcon(f))=0$, then
we also have $\sL(B(z,\delta_0); h_0)<\infty$ for each $z\in J(f)$.
\end{prop}
Notice that in the proposition above the conformal measure~$\mu$ might not charge~$J(f)$.

We shall suspend the proof of these propositions
until~\S\ref{ss:poincare series}.
Let us now deduce from them Theorems~\ref{t:regularity density}, \ref{t:fractal dimensions} and~\ref{t:holomorphic removability}, in \S\S\ref{ss:proof of regularity density}, \ref{ss:proof of fractal dimensions} and~\ref{ss:proof of holomorphic removability}, respectively.
\subsection{Proof of Theorem~\ref{t:regularity density}}
\label{ss:proof of regularity density}
This section is devoted to the proof of Theorem~\ref{t:regularity density}.
So assume that $f\in\sA^*$ satisfies $\gamma(f)>2$.
We put~$h_0 = \HD(J(f))$.

Let $\beta \in (0, \beta_{\max}(f))$, $t > \delta_{\bad}(f)$, $q > q(f)$ and~$q' < p/(p - 1)$ be such that $t + 2 \beta^{-1} < \HD(J(f))$, and
\begin{equation*}
h
\=
\frac{q'}{q' - 1} \left(h_0 - (h_0 - t - 2 \beta^{-1})/ q \right)
<
h_0 - \delta_{\bad}(f) - \beta_{\max}(f)^{-1}.
\end{equation*}

We will prove that there is a constant $C_0 > 0$ such that for each
Borel subset~$A$ of~$J(f)$, and each integer~$n \ge 0$,
\begin{equation}\label{eqn:uniformac}
\mu(f^{-n}(A)) \le C_0 \mu(f(A))^{1/q'}.
\end{equation}
Note that this will complete the proof of the theorem. Indeed, this
implies that any accumulation point~$\nu'$ of the sequence of
measures $\left\{ \tfrac{1}{n} \sum_{i=0}^{n-1} f^i_*\mu \right\}_{n
= 0}^\infty$ is such that for each Borel subset~$A$ of~$J(f)$, we
have $\nu'(A)\le C_0\mu(A)^{1/q'}$.
Thus~$\nu'$ is an invariant probability measure that is absolutely continuous with respect to~$\mu$, and since~$p < q'/(q' - 1)$ by our choice of~$q'$, we
also have that the density of~$\nu'$ with respect to~$\mu$ belongs
to~$L^p(\mu)$.
By ergodicity of~$\mu$, we have $\nu=\nu'$.

\partn{1}
We first prove that there exists a constant a~$C_1>0$ such that for each Borel subset~$B$ of~$\dom(f)$,
\begin{equation}\label{eqn:integcond}
 \int_B \Delta_{m}^{-h(1 - 1/q')} d\mu
\le
C_1 (m + 1) \mu(B)^{1/q'}.
\end{equation}
Indeed, since~$h < h_0 - \delta_{\bad}(f) -
\beta_{\max}(f)^{-1}$, by part~$1$ of Corollary~\ref{coro:inducing}
there exists a constant~$C_2 > 0$ such that for each $z_0\in J(f)$,
$$\int_{J(f)} \dist(z_0, z)^{-h} d\mu(z) \le C_2.$$
By H{\"o}lder inequality, for each Borel subset~$B$ of~$\dom(f)$,
$$\int_{B} \dist (z_0,z)^{-h(1 - 1/q')} d\mu(z)
\le
C_2^{1 - 1/q'} \mu(B)^{1/q'}.$$
Thus, the desired inequality holds with $C_1 = \# \Crit'(f) C_2^{1 - 1/q'}$.

\partn{2}
Now let $(\hV, V)$ be a sufficiently small nice couple so that $\delta_{\bad}(\hV)<t$. Since $f$ has a conformal measure of exponent $h_0 = \HD(J(f))$, by Proposition~\ref{prop:cPzinV} there exists a constant~$C_3>0$ such that for each~$z \in V$
\begin{equation}
\cP_n (z;h_0)
\le
C_3\sum_{m=0}^n \sL_m^{\bad}(\hV; t+2\beta^{-1}) \xi_m(z)^{h_0-t-2\beta^{-1}} \Delta_m(z)^{-h_0}.
\end{equation}
Let us prove that there exists a constant~$C_4>0$ such that
\begin{equation} \label{eqn:poinestinV}
\frac{\cP_n(z; h_0)}{|Df(z)|^{h_0}}
\le
C_4\sum_{m=0}^n \sL_m^{\bad} (\hV; t+2\beta^{-1}) \Delta_{m+1}(f(z))^{-h(1 - 1/q')}.
\end{equation}

Indeed, there exists a constant~$C_5 > 1$ such that for each~$z \in
V$,
\begin{equation}\label{e:critical distance}
\Delta_{m+1}(f(z)) \le C_5 |Df(z)|\Delta_m(z).
\end{equation}
Since $q > q(f)$, by our choice of~$\varepsilon_0$, we have for some
constant $C_6>0$,
$$\xi_m(z)
\le \theta_{m + 1} (B(f(z), 2 \varepsilon_0 |Df(z)| \Delta_{m}(z)))
\le C_6 (|Df(z)| \Delta_{m}(z))^{1/q}.$$
Inequality~\eqref{eqn:poinestinV} follows using~\eqref{e:critical distance}, and the definition of~$h$.

\partn{3} Let $(\hV, V)$ be as above.
We prove that~\eqref{eqn:uniformac} holds for all Borel sets $A\subset V$.
Without loss of generality, we may assume that $f|A$ is injective.
Then
$$\mu(f^{-n}(A))=\int_A \cP_n(z; h_0)d\mu(z)= \int_{f(A)} \frac{\cP_n(z; h_0)}{|Df(z)|^{h_0}} d\mu(w),$$
where $w=f(z)$.
By~\eqref{eqn:poinestinV} and~\eqref{eqn:integcond}, this implies
\begin{align*}
\mu(f^{-n}(A))
& \le
C_4 \sum_{m=0}^n \sL_m^{\bad} (\hV, t+2\beta^{-1})\int_{f(A)} \Delta_{m+1}(w)^{-h(1-1/q')}d\mu(w)
\\ & \le
C_4 C_1\sum_{m=0}^n \sL_m^{\bad} (\hV, t+2\beta^{-1})(m+2) \mu(f(A))^{1/q'}.
\end{align*}

Now fix~$\beta' \in (\beta, \beta_{\max}(f))$.
Then there is a constant~$C_7 > 0$ such that for each~$\tY\in
\fB_m(\hV)$,
$$d(\tY) \diam(\tY)^{ t + 2 \beta^{-1}}
\le
C_7 d(\tY) \diam(\tY)^t (m+2)^{-2\beta'\beta^{-1}},  $$
which implies
\begin{align*}
C_8
\=
\sum_{m=0}^\infty \sL_m^{\bad} (\hV; t + 2 \beta^{-1}) (m+2)
& \le
C_7 \sum_{m=0}^\infty \sL_m^{\bad} (\hV; t) (m + 2)^{1 - 2 \beta' \beta^{-1}}
\\ & \le
C_7 \sL^{\bad} (\hV; t) \sum_{m = 0}^\infty (m + 2)^{1 - 2 \beta' \beta^{-1}}
\\ & <
\infty.
\end{align*}
This proves that for Borel sets $A\subset V$, the
inequality~\eqref{eqn:uniformac} holds with $C_0=C_1C_4 C_8$.

\partn{4}
It remains to prove~\eqref{eqn:uniformac} holds for all Borel sets
$A\subset J(f)\setminus V$. The case $n=0$ is trivial, so we shall
assume $n\ge 1$. For $m\ge 1$, let
$$ X_m
\= \{z\in J(f): z, f(z), \ldots, f^{m-1}(z)\not\in V\} $$
and $A_m=\{z\in X_m: f^{m-1}(z) \in A\}$.
Then for any $n\ge 1$,
$$ f^{-n}(A)
=
A_{n+1} \cup \left(\bigcup_{m=1}^{n} f^{-(n-m)} ( f^{-1} (A_m)\cap V)\right). $$
By what we have proved in part~$3$, it suffices to show there exist constants~$C_9>0$ and~$\kappa\in (0,1)$ such that
\begin{equation}\label{eqn:muam}
\mu(A_m)\le C_9\kappa^m \mu(f(A)).
\end{equation}
To this end, let $V'\Subset V$ be a nice set and let
$$ X_m'
\=
\{z\in J(f): z, f(z), \ldots, f^{m-1}(z)\not\in V'\}. $$
Then by the latter part of Lemma~\ref{lem:finitelanding}, $\mu(X_m')$ is exponentially small in~$m$.
Clearly, there exists a small constant $\rho>0$ such
that for each $z\in X_m$, the map~$f^m$ maps a neighborhood~$U(z)$ of~$z$ diffeomorphically onto~$B(f^m(z),\rho)$ with uniformly bounded
distortion, and such that $U(z)\cap J(f)\subset X_m'$. It follows
that for $w\in J(f)$,
$$\cP_m^*(w; h_0)
\=
\sum_{z\in X_m \, : \, f^m(z)=w} |Df^m(z)|^{-h_0} \asymp \sum_{z\in X_m \, : \, f^m(z)=w} \mu (U(z))
\le
\mu (X_m').$$ Since
$$\mu(A_m)
=
\int_{f(A)} \cP_m^*(w; h_0) d\mu(w),$$
inequality~\eqref{eqn:muam} follows.

We have completed the proof of Theorem~\ref{t:regularity density}.

\subsection{Proof of Theorem~\ref{t:fractal dimensions}}
\label{ss:proof of fractal dimensions}

This section is devoted to the proof of Theorem~\ref{t:fractal dimensions}, which is based on the following proposition; see~\S\ref{sss:conformal and invariant} for the definition of~$\gamma(f)$.
\begin{prop}[Poincar{\'e} series]
\label{p:Poincare series}
Assume that~$f \in \sA^*$ satisfies $\gamma(f)>1$.
Then~$\poincare(f) = \HD(J(f))$.
More precisely,
\begin{enumerate}
\item[1.]
For every $x_0 \in \dom(f)$ that is not asymptotically exceptional, we have $\cP(x_0; \HD(J(f))) = \infty$.
\item[2.] There is a subset~$E$ of~$J(f)$ with $\HD(E)< \HD(J(f))$ and a neighborhood~$U$ of~$J(f)$ such that for every $x_0 \in U \setminus E$, and every~$s
> \HD(J(f))$, the Poincar{\'e} series~$\cP(x_0; s)$ converges.
\end{enumerate}
\end{prop}
\begin{proof}
By Corollary~\ref{coro:inducing}, $f$ has a conformal measure~$\mu$
of exponent~$\HD(J(f))$ that is supported on~$\Jcon(f)$.

Part~$1$ follows from the existence of such a conformal measure, see
for example~\cite[Theorem~$5.2$]{McM00}.

To prove part~$2$, take $t_0 \in (\delta_{\bad}(f), \HD(J(f)))$ and
$\beta \in (0, \beta_{\max}(f))$ such that $\gamma=(\HD(J(f))-t_0)\beta>1$,
and let $h=\HD(J(f))/\gamma$.  Put
$$ E_0 \= \bigcap_{n_0 \ge 1} \left(\bigcup_{n \ge n_0}\bigcup_{c\in\Crit'(f)} B(f^n(c), n^{-1/h})\right),\, \,
E_1 \= E_0\cup \left(\bigcup_{n=1}^\infty f^n(\Crit'(f))\right).$$
Moreover, let $(\hV, V)$ be a nice couple such that
$\delta_{\bad}(\hV)<t_0$ and put
$$E = (K(V)\cap J(f)) \cup \bigcup_{n=0}^\infty f^{-n}(E_1).$$
Then $\HD(E_1)=\HD(E_0)\le h<\HD (J(f))$, so $\HD(E)<\HD(J(f))$.

By the (essentially) topologically exact property of~$J(f)$, we have
that $\mu(B(x, \delta))>0$ for every $x\in\Crit'(f)$ and every~$\delta>0$.
Let $\delta_0>0$ be the constant given by Proposition~\ref{prop:sLall} for the conformal measure~$\mu$ and~$h_0 = \HD(J(f))$.
Let~$U$ be the $\delta_0$-neighborhood of $J(f)$.
Reducing~$\delta_0$ if necessary we assume that~$U \setminus J(f)$ is disjoint from~$\bigcup_{n = 1}^\infty f^n(\Crit(f))$.

We first prove that for $x\in U\setminus J(f)$, and
$s>\HD(J(f))$, we have $\cP(x;s)<\infty$.
To this end, take $z\in J(f)$ such
that $x\in B(z,\delta_0)$, and take $\delta>0$ small such that $B(x,
2\delta) \subset B(z, \delta_0)\setminus J(f)$.
Then by the Koebe principle, we obtain
$$\cP(x; s)\asymp \sL_s(B(x, \delta))\le \sL_s(B(z,
\delta_0))<\infty.$$

To complete the proof, let us prove that $\cP(x;s)<\infty$ for all
$x\in J(f)\setminus E$ and $s>\HD(J(f))$.
Since $x\not\in K(V)$, there exists an integer~$n\ge 0$ such that $x_0\=f^n(x)\in V\setminus E_1$.
It suffices to prove that $\cP(x_0;s)<\infty$. To this end, we first
observe that there exists a constant $C(x_0)>0$ such that
$\Delta_m(x_0)\ge C(x_0) (m+1)^{-1/h}$ for all $m=0,1,\ldots$. Next,
letting $t\in (0,s)$ be such that $\beta (t-t_0)>s/h$, we have that
there is a constant~$C_0 > 0$ such that for each $\tY\in\fB_m(\hV)$,
$$\diam (\tY)^{t-t_0}\Delta_m(x_0)^{-s}
\le
C_0 C(x_0)^{-s} m^{-\beta(t-t_0)}m^{s/h}\le C_0 C(x_0)^{-s},$$
so
$$\sL_m^{\bad} (\hV; t)\Delta_m(x_0)^{-s}
\le
C_0 C(x_0)^{-s} \sL_m^{\bad} (\hV; t_0).$$
By~\eqref{eqn:cPzs}, we obtain that there are constants~$C > 0$ and~$C_1 > 0$ such that,
\begin{multline*}
\cP(x_0;s)\le C \sum_{m=0}^\infty \sL_m^{\bad} (\hV; t)\xi_m(x_0)^{s-t} \Delta_m(x_0)^{-s}\\
\le C_1 \sum_{m=0}^\infty \sL_m^{\bad} (\hV; t) \Delta_m(x_0)^{-s}
\le C_1 C_0 C(x_0)^{-s} \sum_{m=0}^\infty \sL_m^{\bad} (\hV; t_0)<
\infty.
\end{multline*}
\end{proof}

\begin{proof}[Proof of Theorem~\ref{t:fractal dimensions}]
By Theorem~\ref{t:tail estimate} and Corollary~\ref{coro:inducing}, we have $\HD(J(f))= \HD_{\hyp}(f)$, and there is a conformal measure of exponent~$\HD(J(f))$ supported on the conical Julia set of~$f$.
On the other hand, by Proposition~\ref{p:Poincare series} the Poincar{\'e} exponent of~$f$ is equal to~$\HD(J(f))$.
So, to complete the proof of the theorem it would be enough to prove
$\BD(J(f))\le \delta_{\textrm{Poin}}(f)$.
If~$J(f)$ has a nonempty
interior, then there is nothing to prove.
So let us assume the
contrary. Then the Julia set has zero Lebesgue measure by part~$1$
of Corollary~\ref{coro:inducing}.
In the case $f\in\sA_\C$, the
conclusion then follows from~\cite[Fact~$8.1$ and Lemma~$8.2$]{GraSmi09}, in which $\BD(J(f))=\delta_{\textrm{Poin}}(f)$
was proved directly. The proof extends to the case of~$f\in \sA_\R$
with the following minor modifications and gives us the desired
inequality:
\begin{itemize}
\item
Instead of taking one point~$z_j$ from each cycle of periodic components
of~$\sF$, we may need to take two points, as in the real case, we
may only find a ``fundamental domain'' that is the union of two
intervals;
\item
Instead of the displayed formula~($24$) in page~$392$ of~\cite{GraSmi09} derived from~\cite[Lemma~$7$]{GraSmi98}, we apply
the Koebe principle and obtain a one-sided inequality: $\dist (y,
J(f))\ge C^{-1} |Df^n(y)|^{-1}$ for $y\in f^{-n}(z_j)$, where $C$ is
a Koebe constant.
\end{itemize}
\end{proof}
\subsection{Proof of Theorem~\ref{t:holomorphic removability}}
\label{ss:proof of holomorphic removability}

If~$\Crit'(f) = \emptyset$ then~$f$ is uniformly hyperbolic and the
result follows easily from the removability result~\cite[Theorem~$5$]{JonSmi00}, see also~\cite[Fact~$9.1$]{GraSmi09}.
The latter statement of Theorem~\ref{t:holomorphic removability} follows from the former one by Theorems~\ref{t:polynomial
shrinking} and~\ref{t:finiteness of bad}, by Fact~\ref{f:nice couples} and by~\cite[Corollary~$8.3$]{Riv07}.
To prove the former statement of Theorem~\ref{t:holomorphic removability}, assume $\beta_{\max}(f)(2-\delta_{\bad}(f))>1$.
By~\cite[Theorem~$5$]{JonSmi00}, it suffices to prove that for every
$x \in J(f)$ there exists a constant~$\delta_0=\delta_0(x)>0$ such that
$\sL(B(x,\delta_0); 2)<\infty$.
To do this, take $\beta \in (0, \beta_{\max}(f))$ and $t > \delta_{\bad}(f)$ such that~$\beta (2-t)>1$.
Since the normalized Lebesgue measure $\mu$ is a conformal
measure of exponent~$2$ and~$\mu(\Jcon(f))=0$,
Proposition~\ref{prop:sLall} applies and gives us the desired property.

\subsection{Proof of Propositions~\ref{prop:cPzinV} and~\ref{prop:sLall}}
\label{ss:poincare series}
Throughout this section we fix a map~$f$ in the class~$\sA^*$ defined in~\S\ref{sss:inducing}.
Moreover, we fix a nice couple~$(\hV, V)$ for~$f$.
For each integer~$n \ge 1$, each $W\in\sM_n(\hV)$ and each $z\in \hV \cap J(f)$, let
$$\cP_{W}(z;s)
\=
\sum_{y\in f^{-n}(z)\cap W} |Df^n(y)|^{-s}.$$
\index{$\cP_{W}$}
Moreover, for $Q\subset \hV$ let $\sM_{W}(Q)\subset\sM_n(Q)$ be the
collection of components of $f^{-n}(Q)\cap W$ and let
$$ \sL_{W}(Q; s)
\=
\sum_{P\in\sM_{W}(Q)} d_Q(P)\diam (P)^s. $$
\index{$\sL_{W}$}
Let $\fG_0(\hV)=\sM_0(\hV)$ and for $n\ge 1$, let $\fG_n(\hV)$ be the collection of all diffeomorphic pull-backs of $\hV$
by $f^n$.\index{$\fG_n(\hV)$}
For $Q\subset \hV$, define
$$\sM_n^o(Q)
\=
\bigcup_{W\in \fG_n(\hV)} \sM_{W}(Q),
\text{ and }
\sM^o(Q)
\=
\bigcup_{n=0}^\infty \sM_n^o(Q).$$
Moreover, let~$\sL_n^o(Q;s)$, $\sL^o(Q;s),$ $\cP_n^o(z;s)$ and $\cP^o(z;s)$ be defined in a self-evident way.
\index{$\sM_n^o, \sM^o$}
\index{$\sL_n^o, \sL^o$}
\index{$\cP_n^o, \cP^o$}
Now let us prove Proposition~\ref{prop:cPzinV}.
We start with a lemma proving upper bounds for $\sL_n^o$ and
$\sL^o$.
\begin{lemm} \label{lem:diffpbfinite}
Assume that~$f$ has a conformal measure~$\mu$ of
exponent~$h_0$ that charges each open set intersecting
$\Crit'(f)$.
Then the following hold:
\begin{enumerate}
\item[1.] For $s>h_0$, we have $\sL^o (V;s)<\infty$;
\item[2.]
$\sup_{n=0}^\infty\sL_n^o (V; h_0)<\infty$;
\item[3.] If furthermore $\mu(\Jcon(f))=0$, then $\sL^o (V; h_0)<\infty$.
\end{enumerate}
\end{lemm}
\begin{proof}
\partn{1}
Let us first prove
\begin{equation}
\label{eqn:sLVin}
\sum_{W\in\sM^o(V) \, : \, W\subset V} \diam (W)^s
<
\infty.
\end{equation}
Indeed, each $W\in\sM^o(V)$ with $W\subset V$ is a component of
$\dom (F^n)$ for some $n\ge 0$, and $F^n|W$ has  uniformly bounded
distortion.
So there is a constant $C_1 > 0$ independent of~$n$, such that
$$\sum_{\text{components~$W$ of } \dom (F^n)}\diam (W)^{h_0}
\le
C_1 \mu (\dom (F^n)).$$
Since $\diam (W)$ is exponentially small in terms of $n$,  the same
sum, but with the exponent~$h_0$ replaced by~some $s>h_0$, is
exponentially small with~$n$.
Hence~\eqref{eqn:sLVin} holds.

Let us spread the estimate to all $W\in \sM^o(V)$. Let $L: \dom
(f)\setminus K(V)\to V$ denote the first landing map onto $V$. Since
the distortion of $L$ on each component of its domain is uniformly
bounded, as above we obtain that
$$ \sum_{U\in \fL_V}\diam (U)^{h_0}
<
\infty, $$
where~$\fL_V$ denotes the collection of connected components of
$\dom(f)\setminus K(V)$.
Hence
\begin{equation}
  \label{e:landing content}
\sum_{U\in\fL_V} \diam (U)^s
<
\infty.
\end{equation}
Note that each $W\in\sM^o(V)$ is contained
in some $U\in \fL_V$, and $L (W)\in \sM^o(V)$. By the bounded
distortion property of $L$, there is a constant~$C_2 > 0$ that only depends on~$(\hV, V)$, such that
$$ \sum_{W\in\sM^o(V) \, : \, W\subset U} \diam (W)^s
\le
C_2 \frac{\diam (U)^s}{\diam (L(U))^s} \sum_{W \in \sM^o (V) \, : \, W\subset L(U)}\diam (W)^s.$$
Since $L(U)$ is a component of $V$, we obtain for some constant~$C_3 > 0$,
$$ \sL^o (V;s)
\le
C_3\sum_{U\in\fL_V}\diam (U)^s \sum_{W\in\sM^o(V) \, : \, W\subset V} \diam (W)^s<\infty.$$

\partn{2}
For each $W\in \sM_n^o(V)$, $f^n$ maps $W$  diffeomorphically
onto a component of $V$ and the map has uniformly bounded distortion, so
$$C^{-1}\diam (W)^{h_0}\le \mu (W)\le C\diam (W)^{h_0},$$
where $C$ is a constant independent of $W$. It follows that
$$\sL_n^o (V; h_0)=\sum_{W\in \sM_n^o (V)} \diam (W)^{h_0}\le C \mu(\dom
(f))\le C.$$

\partn{3}
Arguing as in the proof of part~$1$, it suffices to prove that~$\mu(\dom (F^n))$ is exponentially small in terms of $n$.
But the assumption that $\mu(\Jcon(f))=0$ implies that for some $n_0 \ge 1$ and each component $V_c$ of $V$,  $\mu(V_c\setminus \dom(F^{n_0}))>0$, hence  by
the Koebe principle, $\mu(\dom (F^n))$ decreases exponentially fast.
\end{proof}
\begin{lemm}\label{lem:sLmQV}
For each $m\ge 0$, $W\in\sM_m(\hV)$, each $s\ge t>0$ and each connected $Q \subset \hV$,
\begin{equation}\label{eqn:sLmtYQ}
\sL_{W}(Q; s)\le d_{\hV}(W)\diam (W)^t \theta_m(Q)^{s-t}.
\end{equation}
Moreover, there exists a constant~$C>0$ such that for every~$z \in J(f)\cap V$,
\begin{equation}\label{eqn:cPmtYz}
\cP_{W}(z;s)\le C d_{\hV}(W)\diam (W)^t \xi_m(z)^{s-t}
\Delta_m(z)^{-s}.
\end{equation}
\end{lemm}
\begin{proof}
For each $P\in\sM_{W}(Q)$,
$$ \diam (P)\le \diam (W)
\text{ and }
\diam (P)\le \theta_m(Q). $$
Then inequality~\eqref{eqn:sLmtYQ} follows from,
\begin{multline*}
\sL_{W}(Q;s)
\le
\sum_{P\in\sM_{W}(Q)} d_Q(P)\sup_{P\in\sM_{W}(Q)}\diam(P)^s
\\ \le
d_{\hV}(W)\sup_{P\in\sM_{W}(Q)}\diam (P)^s.
\end{multline*}

To obtain~\eqref{eqn:cPmtYz}, observe first that each pull-back of~$B(z, \Delta_m(z))$ by~$f^{m}$ is diffeomorphic so, by the definition of~$\varepsilon_0$, for each~$P \in \sM_m(B(z, 2 \varepsilon_0 \Delta_m(z)))$,
$$|Df^m((f^m|P)^{-1}(z))|^{-1}
\le
(2 \varepsilon_0)^{-1} \diam (P)\Delta_m(z)^{-1}.$$
So inequality~\eqref{eqn:cPmtYz} follows from~\eqref{eqn:sLmtYQ} with~$Q = B(z, 2\varepsilon_0 \Delta_m(z))$.
\end{proof}

\begin{lemm} \label{lem:diffbad2all}
For any $s\ge t>0$, there exists a constant~$C' > 0$ such that for each connected $Q \subset V$ and $z\in V\cap J(f)$,
\begin{equation}\label{eqn:slQnall}
\sL_n(Q; s)
\le
C' \sum_{m=0}^n \sL_{n-m}^o (V; s)\sL_m^{\bad}(\hV; t)
\theta_m(Q)^{s-t},
\end{equation}
\begin{equation}\label{eqn:cPnall}
\cP_n(z;s)
\le
C' \sum_{m=0}^n \sL_{n-m}^o (V;s) \sL_m^{\bad} (\hV; t)\xi_m(z)^{s-t} \Delta_m(z)^{-s}.
\end{equation}
\end{lemm}
\begin{proof}
For each $W\in\sM_n(\hV)$, let~$k({W})$ be the minimal integer in $\{0,1,\ldots, n\}$ such that~$f^{n-k(W)}$ maps a neighborhood of~$W$ diffeomorphically onto a component of~$\hV$ and let~$\tY_{W}$ be the component of~$f^{-k(W)}(\hV)$ that contains~$f^{n-k(W)}(W)$.
Then $\tY_{W}\in \fB_{k(W)}(\hV)$.
Note that for each $P \in \sM_{W}(Q)$, we have $P'\=f^{n-k(W)}(P)\subset V$.
Indeed, if $k(W)=0$, then $P'=Q\subset V$; otherwise, $\tY_{W}$ is a bad pull-back of $\hV$ and hence $P'\subset \tY_{W}\subset V$.
Given $k\in\{0,1,\ldots, n\}$ and $U\in\sM_{n-k}^o(\hV)$, by Koebe
principle and~\eqref{eqn:sLmtYQ} there are constants~$C_0 > 0$
and~$C_1 > 0$ such that,
\begin{align*}
\sum_{\substack{W\in\sM_n(\hV) \, : \\ W\subset U, k(W) = k}}\sL_{W}(Q;s)
& \le
C_0 {\diam (U)^s}
\sum_{\tY \in \fB_k^{\bad}(\hV)} \sL_{\tY}(Q; s)
\\ & \le
C_1 \diam (U)^s \sL_k^{\bad}(\hV; t) \theta_k(Q)^{s-t},
\end{align*}
where in the second inequality, we used~\eqref{eqn:sLmtYQ}.
Summing over all $U \in \sM_{n-k}^o(\hV)$ and then over all $k=0,1,\ldots, n$, we obtain~\eqref{eqn:slQnall}.

Repeating the argument, using~\eqref{eqn:cPmtYz} instead of~\eqref{eqn:sLmtYQ}, we obtain~\eqref{eqn:cPnall}.
\end{proof}

\begin{proof}[Proof of Proposition~\ref{prop:cPzinV}]
Inequality~\eqref{eqn:cPnzh} follows from~\eqref{eqn:cPnall} with~$s
= h_0$, and from part~$2$ of Lemma~\ref{lem:diffpbfinite}.
Again by~\eqref{eqn:cPnall}, when~$s > h_0$,
\begin{align*}
\cP(z;s)
& =\sum_{n=0}^\infty \cP_n(z;s)
\\ & \le
C' \sum_{n=0}^\infty \sum_{m=0}^n \sL_{n-m}^o (V;s) \sL_m^{\bad} (\hV; t)\xi_m(z)^{s-t} \Delta_m(z)^{-s}
\\ & =
C' \sum_{n=0}^\infty \sL_n^o(V;s) \sum_{m=0}^\infty \sL_m^{\bad} (\hV; t)\xi_m(z)^{s-t}\Delta_m(z)^{-s},
\end{align*}
which implies~\eqref{eqn:cPzs} by part~$1$ of Lemma~\ref{lem:diffpbfinite}.
\end{proof}

\begin{proof}[Proof of Proposition~\ref{prop:sLall}] Take
$0<\beta<\beta_{\max}(f)$ and $t>\delta_{\bad}(f)$ such that $\beta
(h_0-t)>1$.
Let $(\hV, V)$ be so small such that $\delta_{\bad}
(\hV)< t$ and $\theta (V)\le 1$.
Fix a constant $s>h_0$ if
$\mu(\Jcon(f))>0$ and $s\ge h_0$ otherwise.
We first prove that
there exists a constant~$C>0$ such that for any $Q\subset V$,
\begin{equation}\label{eqn:smallsLQ}
\sL(Q;s)\le C \theta (Q)^{s-t}.
\end{equation}

Indeed, in case $s> h_0$ by part~$1$ of Lemma~\ref{lem:diffpbfinite}
and in case $\mu(\Jcon(f))=0$ by part~$3$ of that lemma, we have
$\sL^o(V;s)<\infty$.
So by~\eqref{eqn:slQnall},
\begin{align*}
\sL(Q;s)
& =
\sum_{n=0}^\infty\sL_n(Q; s)
\\ & \le
C \sum_{n=0}^\infty\sum_{m=0}^n \sL_{n-m}^o (V; s)\sL_m^{\bad}(\hV; t) \theta_m(Q)^{s-t}
\\ & =
C \sum_{k=0}^\infty \sum_{m=0}^\infty \sL_k^o(V;s) \sL_m^{\bad}(\hV; t) \theta_m(Q)^{s-t}
\\ & \le
C\sL^o(V; s) \sL^{\bad}(\hV; t)\theta(Q)^{s-t},
\end{align*}
thus~\eqref{eqn:smallsLQ} holds.

To prove the proposition, it suffices to show that each $x\in J(f)$
has a neighborhood $B_x$ such that $\sL(B_x; s)<\infty$ since $J(f)$ is compact.

By~\eqref{eqn:smallsLQ}, $\sL(V; s)<\infty$, so for $x\in V\cap
J(f)$, we may take $B_x=V$.
For~$x \in J(f)\setminus K(V)$, letting
$n\ge 0$ be such that $f^n(x)\in V$ and taking $B_x\ni x$ such that
$f^n(B_x)\subset V$, we have $\sL(B_x;s)<\infty$.
So let us assume
that~$x \in K(V)\cap J(f)$.

Let~$\delta_0 > 0$ be sufficiently small so that every pull-back
of~$B(x, 2\delta_0)$ intersecting~$\Crit'(f)$ is contained in~$V$,
and such that for every~$y \in K(V)$ and every~$n \ge 1$ the
pull-back of~$B(f^n(y), 2\delta_0)$ by~$f^n$ containing~$y$ is
diffeomorphic. Let us prove that $\sL(B(x,\delta_0);s)<\infty$.
Let~$\fH^o$ be the collection of all those pull-backs of~$B(x,
\delta_0)$ such that the corresponding pull-back of~$B(x,
2\delta_0)$ is diffeomorphic and let~$\fH'$ (resp.~$\fH''$) be the collection of all pull-backs of~$B(x, 2\delta_0)$ that are non-diffeomorphic (resp. that intersect~$\Crit'(f)$).
Let~$W_*$ be a pull-back of~$V$ contained in~$B(x, \delta_0)$.
Then there is a distortion constant~$C_0 > 1$ such that,
$$ \sum_{W \in \fH^o} d_{B(x,\delta_0)}(W)\diam (W)^{s}
\le
C_0 \sL(W_*;s)
<
\infty. $$
It is thus enough to prove that
$$ \sum_{W \in \fH'} d_{B(x, \delta_0)}(W) \diam(W)^s < \infty. $$

Not that for each integer~$n \ge 1$ there is at most one pull-back of~$B(x, 2\delta_0)$ by~$f^n$ containing a given element of~$\Crit'(f)$.
On the other hand, since~$f$ satisfies the Polynomial Shrinking Condition with
exponent~$\beta$, there is a constant~$C_1 > 0$ such that for each integer~$n \ge 1$ and each pull-back~$Q$ of~$B(x, 2\delta_0)$ by~$f^n$, we have $\theta(Q) \le C_1n^{-\beta}$.
Thus
$$\sum_{Q\in \fH'' }\theta (Q)^ {s-t}
\le
\#\Crit'(f) C_1^{s - t} \sum_{n=1}^\infty n^{-\beta (s-t)}<\infty. $$
Therefore,
\begin{multline*}
\sum_{W \in\fH'} d_{B(x, \delta_0)}(W) \diam (W)^s
\le
\sum_{Q\in\fH''}  2 \ell_{\max}(f) \sL(Q;s)
\\ \le
2 \ell_{\max}(f) C \sum_{Q\in \fH''}\theta(Q)^{s-t}
<
\infty.
\end{multline*}
\end{proof}

\bibliographystyle{alpha}

\printindex

\end{document}